\documentclass[11pt]{article}
\usepackage{amssymb,amsmath,epsfig}
\usepackage[colorlinks=false,breaklinks=true,linkcolor=blue]{hyperref}
\usepackage{graphics,psfrag,graphicx,color}
\usepackage{diagbox}
\usepackage{float,hhline}

\numberwithin{equation}{section}
\newcommand{\ds}{\displaystyle}

\newcommand{\bgamma}{{\boldsymbol\gamma}}

\newcommand{\bomega}{{\boldsymbol\omega}}
\newcommand{\bsi}{{\boldsymbol\sigma}}

\newcommand{\bPi}{{\mbox{\boldmath $\Pi$}}}
\newcommand{\btau}{{\boldsymbol\tau}}
\newcommand{\bzeta}{{\boldsymbol\zeta}}

\newcommand{\bchi}{{\boldsymbol\chi}}

\newcommand{\bv}{{\mathbf{v}}}
\newcommand{\bw}{{\mathbf{w}}}
\newcommand{\f}{\mathbf{f}}

\newcommand{\bc}{\mathbf{c}}
\newcommand{\bi}{\mathbf{i}}

\newcommand{\bu}{\mathbf{u}}

\newcommand{\by}{{\mathbf{y}}}
\newcommand{\bz}{{\mathbf{z}}}

\newcommand{\bn}{{\mathbf{n}}}

\def\bs{\mathbf{s}}
\newcommand{\0}{{\mathbf{0}}}
\def\bA{\mathbf{A}}
\def\bB{\mathbf{B}}
\def\bC{\mathbf{C}}

\def\bF{\mathbf{F}}
\def\bG{\mathbf{G}}

\def\bI{\mathbf{I}}

\def\bW{\mathbf{W}}
\def\bM{\mathbf{M}}
\def\bT{\mathbf{T}}

\def\bP{\mathbf{P}}

\def\bRT{\mathbf{RT}}

\def\bx{\mathbf{x}}
\newcommand{\bL}{\mathbf{L}}
\newcommand\bH{\mathbf{H}}

\newcommand\bbM{\mathbb{M}}
\newcommand\bbI{\mathbb{I}}

\newcommand\bbH{\mathbb{H}}

\newcommand\bbL{\mathbb{L}}

\newcommand{\bbRT}{\mathbb{RT}}

\newcommand{\cE}{\mathcal{E}}

\newcommand{\cT}{\mathcal{T}}

\newcommand{\cO}{\mathcal{O}}
\newcommand{\cP}{\mathcal{P}}
\newcommand{\cR}{\mathcal{R}}
\def\rp{\mathrm{p}}
\def\rq{\mathrm{q}}

\def\R{\mathrm{R}}

\def\H{\mathrm{H}}
\def\L{\mathrm{L}}
\def\M{\mathrm{M}}

\def\W{\mathrm{W}}

\def\rd{\mathrm{d}}
\def\re{\mathsf{e}}
\def\sr{\mathsf{r}}
\def\rD{\mathrm{D}}

\def\rP{\mathrm{P}}
\def\rt{\mathrm{t}}
\def\ttd{\mathtt{d}}
\def\tF{\mathtt{F}}

\def\bdiv{\mathbf{div}}
\def\bcurl{\mathbf{curl}}

\def\ucurl{\underline{\mathrm{curl}}}
\newcommand{\ubcurl}{\underline{\bcurl}}
\newcommand{\ubgamma}{\underline{\bgamma}}
\def\tr{\mathrm{tr}}

\def\div{\mathrm{div}}
\def\dist{\mathrm{dist}\,}
\def\pil{\left<}
\def\pir{\right>}

\def\coeff{\mathbf{coeff}}
\def\iter{\mathtt{iter}}
\def\tol{\textsf{tol}}
\def\DOF{\mathtt{DOF}}

\def\with{{\quad\hbox{with}\quad}}
\def\qin{{\quad\hbox{in}\quad}}
\def\qon{{\quad\hbox{on}\quad}}
\def\qan{{\quad\hbox{and}\quad}}

\def\ov{\overline}

\def\wt{\widetilde}
\def\wh{\widehat}

\def\jump#1{\text{ $\hspace{-0.1cm}\left[\!\left[#1\right]\!\right]$}}
\newtheorem{thm}{Theorem}[section]
\newtheorem{rem}{Remark}[section]
\newtheorem{lem}[thm]{Lemma}

\newenvironment{proof}{\noindent{\it Proof.}}{\hfill$\square$}

\numberwithin{equation}{section}
\numberwithin{figure}{section}
\numberwithin{table}{section}

\allowdisplaybreaks

\title{An augmented mixed FEM for the convective Brinkman--Forchheimer problem: {\it a priori} and {\it a posteriori} error analysis\thanks{This work was partially supported by 
ANID-Chile through the project {\sc Centro de Modelamiento Matem\'atico} (FB210005)
and Fondecyt project 11220393.}}

\author{{\sc Sergio Caucao}\thanks{Departamento de Matem\'atica y F\'isica Aplicadas, 
Universidad Cat\'olica de la Sant\'isima Concepci\'on, Casilla 297, Concepci\'on, Chile,
and Grupo de Investigaci\'on en An\'alisis Num\'erico y C\'alculo Cient\'ifico, GIANuC$^2$, Concepci\'on, Chile, 
email: {\tt scaucao@ucsc.cl, jesparza@magister.ucsc.cl}}
\quad
{\sc Johann Esparza$^\dag$}}

\date{ }

\begin{document}

\maketitle

\begin{abstract}
\noindent We propose and analyze an augmented mixed finite element method for the pseudostress-velocity formulation of the stationary convective Brinkman--Forchheimer problem in $\R^d,\,d\in \{2,3\}$. 
Since the convective and Forchheimer terms forces the velocity to live in a smaller space than usual, we augment the variational formulation with suitable Galerkin type terms. 
The resulting augmented scheme is written equivalently as a fixed point equation, so that the well-known Schauder and Banach theorems, combined with the Lax--Milgram theorem, allow to prove the unique solvability of the continuous problem. 
The finite element discretization involves Raviart--Thomas spaces of order $k\geq 0$ for the pseudostress tensor and continuous piecewise polynomials of degree $\le k + 1$ for the velocity. 
Stability, convergence, and {\it a priori error} estimates for the associated Galerkin scheme are obtained.
In addition, we derive two reliable and efficient residual-based {\it a posteriori} error estimators for this problem on arbitrary 
polygonal and polyhedral regions. 
The reliability of the proposed estimators draws mainly upon the uniform ellipticity of the form involved, a suitable assumption on the data, a stable Helmholtz decomposition, and the local approximation properties of the Cl\'ement and Raviart--Thomas operators. 
In turn, inverse inequalities, the localization technique based on bubble functions, and known results from previous works, are the main tools yielding the efficiency estimate. 
Finally, some numerical examples illustrating the performance of the mixed finite element method, confirming the theoretical rate of convergence and the properties of the estimators, and showing the behaviour of the associated adaptive algorithms, are reported.
In particular, the case of flow through a $2$D porous media with fracture networks is considered.
\end{abstract}



\maketitle


\section{Introduction}\label{sec:introduction}

This paper focuses on the study of the mathematical and computational modeling of flow of fluids through highly porous media at higher Reynolds numbers using the stationary convective Brinkman--Forchheimer equations.
Such flows occur in a wide range of applications, among which we highlight predicting and controlling processes arising in petroleum, chemical and environmental engineering. 
Fast flows in the subsurface may occur in fractured or vuggy aquifers or reservoirs, as well as near injection and production wells during groundwater remediation or hydrocarbon production. 
Many of the investigations in porous media have focused on the use of Darcy's law. 
Nevertheless, this fundamental equation may be inaccurate for modeling fluid flow through porous media with high Reynolds numbers or through media with high porosity. 
To overcome this limitation, it is possible to consider the convective Brinkman--Forchheimer equations (see for instance \cite{cku2005,zy2012,ll2019,y2023}), where terms are added to Darcy's law in order to take into account high velocity flow and high porosity.

Up to the authors' knowledge, \cite{cku2005} constitutes one of the first works in analyzing the convective Brinkman--Forchheimer (CBF) equations. 
In that work, a complete analysis of the continuous dependence of solutions on the Forchheimer coefficient in $H^1$ norm is proved by the authors.
Later on, an approximation of solutions for the incompressible CBF equations via the artificial compressibility method was proposed and analyzed in \cite{zy2012}, where a family of perturbed compressible CBF equations that approximate the incompressible CBF equations is introduced.
Existence and convergence of solutions for the compressible CBF equations to the solutions of the incompressible CBF equations is proved.
More recently, the two-dimensional stationary CBF equations were analyzed in \cite{ll2019}.
The focus of this work is on the well-posedness of the corresponding velocity-pressure variational formulation.
In particular, error estimates for a mixed finite element approximation were obtained and 
a one-step Newton iteration algorithm initialized using a fixed-point iteration was proposed.
In turn, the existence and uniqueness of an axisymmetric solution to the three-dimensional incompressible CBF equations were proved in \cite{y2023}.
Meanwhile, a mixed pseudostress-velocity formulation but for the unsteady Brinkman--Forchheimer equations was analyzed in \cite{cy2021}.
Here, the existence and uniqueness of a solution are established for the weak formulation in a Banach space framework. 
Semidiscrete continuous-in-time and fully discrete mixed finite element approximations are introduced and sub-optimal rates of convergence are established. 

The goal of the present paper is to develop and analyze a new mixed formulation for the stationary convective Brinkman--Forchheimer problem and study its numerical approximation by an augmented mixed finite element method. 
To that end, unlike previous works, we introduce the pseudostress tensor as in \cite{cot2017-MC} (see also \cite{cgot2016,cgor2017,gos2018}) and subsequently eliminate the pressure unknown using the incompressibi\-lity condition.
Furthermore, the difficulty given by the fact that the fluid velocity lives in $\H^1$ instead of $\L^2$ as usual, is resolved as in \cite{cot2017-MC,cgo2017-CALCOLO,gos2018,cgos2020} by augmenting the variational formulation with residuals arising from the constitutive equation and the Dirichlet boundary condition on the velocity. 
Then, we combine classical fixed-point arguments with the Lax--Milgram theorem to prove the well-posedness of both the continuous and discrete formulations.
In particular, for the continuous formulation, and under a smallness data assumption, we prove existence and uniqueness of solution by means of a fixed-point strategy where the Schauder (for existence) and Banach (for uniqueness) fixed-point theorems are employed.
As for the numerical scheme, whose solvability is established similarly to the continuous case but using the Brouwer fixed-point theorem instead of Schauder's for the existence result, we employ Raviart--Thomas spaces of order $k\geq 0$ for approximating the pseudostress tensor and continuous piecewise polynomials of degree $k + 1$ for the velocity.
In addition, applying an ad-hoc Strang-type lemma, we are able to derive the corresponding {\it a priori} error estimates and prove that the method is convergent with optimal rate. 

Next, we employ the {\it a posteriori} error analysis techniques developed in \cite{gms2010}, \cite{grt2016}, \cite{gos2018}, \cite{cgo2019-CMA}, \cite{cgo2019}, and \cite{cgoz2022} for augmented-mixed formulations in Hilbert spaces, and develop two reliable and efficient residual-based {\it a posteriori} error estimators in 2D and 3D for the present augmented-mixed finite element method. 
More precisely, in each case we derive a global quantity $\Theta$ that is expressed in terms of calculable local indicators $\Theta_T$ defined on each element $T$ of a given triangulation $\cT$. 
This information can be afterwards used to localize sources of error and construct an algorithm to efficiently adapt the mesh. 
In this way, the estimator $\Theta$ is said to be efficient (resp. reliable) if there exists a positive 
constant $C_{\tt eff}$ (resp. $C_{\tt rel}$), independent of the meshsizes, such that
\[
C_{\tt eff}\,\Theta \,+\, {\tt h.o.t.} \,\,\le\,\, \|\mathrm{error}\|\,\,\le\,\,
C_{\tt rel}\,\Theta \,+\, {\tt h.o.t.}\,,
\]
where ${\tt h.o.t.}$ is a generic expression denoting one or several terms of higher order.  
We observe that, up to our knowledge, the present work provides the first {\it a priori} and {\it a posteriori} error analysis of mixed finite element methods for the stationary convective Brinkman--Forchheimer equations.

This paper is organized as follows. The remainder of this section introduces some standard notations and functional spaces.
In Section \ref{sec:continuous-formulation} we introduce the model problem and derive its augmented 
mixed variational formulation.
Next, in Section \ref{sec:analysis-continuous-problem} we establish the well-posedness of this continuous scheme by means of a fixed-point strategy and Schauder and Banach fixed-point theorems. 
The Galerkin finite element approximation and its corresponding {\it a priori} analysis is developed in Section \ref{sec:Galerkin-scheme}. 
In Section \ref{eq:first-a-posteriori-error-estimator} we develop the {\it a posteriori} error analysis. 
In Section \ref{sec:reliability-Theta-1} we employ the uniform ellipticity of the bilinear form involved, a suitable Helmholtz decomposition, the local approximation properties of the Cl\'ement and Raviart--Thomas operators, to derive a reliable residual-based {\it a posteriori} error estimator. 
Then, inverse inequalities, and the localization technique based on element-bubble and edge-bubble functions are utilized in Section \ref{sec:efficiency-Theta-1} to prove the efficiency of the estimator. 
Several numerical results illustrating the performance of the proposed mixed finite element method, confirming the reliability and efficiency of the {\it a posteriori} error estimators, and showing the good performance of the associated adaptive algorithms, are presented in Section \ref{sec:numerical-results}.
Finally, other variables of interest that are recovered by a postprocessing are analyzed in Appendix \ref{sec:appendix-A},
while further properties to be utilized for the derivation of the reliability and efficiency estimates, are provided in Appendix \ref{sec:appendix-B}.
In turn, a second (also reliable and efficient) residual-based {\it a posteriori} error estimator is introduced and studied in Appendix \ref{sec:appendix-C}, where the Helmholtz decomposition is not employed in the corresponding proof of reliability.


\bigskip

\subsection*{Preliminary notations}

Let $\Omega\subset \R^d, d\in\{2,3\}$, be a bounded domain with polyhedral boundary $\Gamma$, 
and let $\bn$ be the outward unit normal vector on $\Gamma$. Standard notation will be adopted 
for Lebesgue spaces $\L^p(\Omega)$ and Sobolev spaces $\W^{s,p}(\Omega)$, with $s\in \R$ and $p > 1$, 
whose corresponding norms, either for the scalar, vectorial, or tensorial case, are denoted by 
$\|\cdot\|_{0,p;\Omega}$ and $\|\cdot\|_{s,p;\Omega}$, respectively. In particular, given a non-negative 
integer $m$, $\W^{m,2}(\Omega)$ is also denoted by $\H^m(\Omega)$, and the notations of its norm and 
seminorm are simplified to $\|\cdot\|_{m,\Omega}$ and $|\cdot|_{m,\Omega}$, respectively.
By $\bM$ and $\bbM$ we will denote the corresponding vectorial and tensorial counterparts of the 
generic scalar functional space $\M$, whereas $\M'$ denotes its dual space, whose norm is defined by $\|f\|_{\M'} :=  \ds\sup_{0\neq v\in \M} \frac{|f(v)|}{\|v\|_\M}$. 
In turn, for any vector fields $\bv=(v_i)_{i=1,d}$ and $\bw=(w_i)_{i=1,d}$, we set the gradient, 
divergence, and tensor product operators, as
\begin{equation*}
\nabla\bv:=\left(\frac{\partial v_i}{\partial x_j}\right)_{i,j=1,d},\quad 
\div (\bv):=\sum_{j=1}^d \frac{\partial v_j}{\partial x_j},\qan 
\bv\otimes\bw:=(v_i w_j)_{i,j=1,d}\,.
\end{equation*}
Furthermore, for any tensor fields $\btau=(\tau_{ij})_{i,j=1,d}$ 
and $\bzeta=(\zeta_{ij})_{i,j=1,d}$, we let $\bdiv(\btau)$ be the divergence operator $\div$ acting 
along the rows of $\btau$, and define the transpose, the trace, the tensor inner product, and the 
deviatoric tensor, respectively, as
\begin{equation*}
\btau^\rt := (\tau_{ji})_{i,j=1,d},\quad 
\tr(\btau) := \sum_{i=1}^d \tau_{ii},\quad
\btau:\bzeta := \sum_{i,j=1}^d \tau_{ij}\,\zeta_{ij},\qan 
\btau^\rd := \btau - \frac{1}{d}\,\tr(\btau)\,\bbI,
\end{equation*}
where $\bbI$ is the identity matrix in $\R^{d\times d}$. In what follows, when no confusion arises, 
$|\cdot|$ will denote the Euclidean norm in $\R^d$ or $\R^{d\times d}$. 
Additionally, we recall the Hilbert space
\begin{equation*}
\bbH(\bdiv;\Omega) := \Big\{\btau\in\bbL^2(\Omega):\quad \bdiv(\btau)\in\bL^ 2(\Omega)\Big\},
\end{equation*}
endowed with the usual norm $\|\btau\|^2_{\bdiv;\Omega} := \|\btau\|^2_{0,\Omega} + \|\bdiv(\btau)\|^2_{0,\Omega}$.
In addition, $\bH^{1/2}(\Gamma)$ is the space of traces of functions of $\bH^1(\Omega)$ and 
$\bH^{-1/2}(\Gamma)$ denotes its dual. Also, by $\pil\cdot,\cdot\pir_{\Gamma}$ we will denote 
the corresponding product of duality between $\bH^{-1/2}(\Gamma)$ and $\bH^{1/2}(\Gamma)$.


\section{The continuous formulation}\label{sec:continuous-formulation}

In this section we introduce the model problem and derive its corresponding weak formulation.

\subsection{The model problem}\label{sec:model-problem}

In what follows we consider the model analyzed in \cite{ll2019} (see also \cite{cku2005,zy2012,y2023}),
which is given by the stationary convective Brinkman--Forchheimer equations. 
More precisely, given a body force $\f$, we focus on finding a velocity 
field $\bu$ and a pressure field $p$, such that
\begin{subequations}\label{eq:convective-Brinkman-Forchheimer-1}
\begin{align}
\ds -\,\nu\,\Delta \bu + (\nabla\bu)\bu 
+ \alpha\,\bu + \tF\,|\bu|^{\rp-2}\bu + \nabla p & = \f \quad\,\,\mbox{ in}\quad \Omega\,, \label{eq:convective-Brinkman-Forchheimer-1a} \\[1ex]
\div(\bu) & =  0 \quad\,\,\,\, \mbox{in}\quad \Omega\,, \label{eq:convective-Brinkman-Forchheimer-1b} \\[1ex]
\bu & = \bu_\rD \quad \mbox{on}\quad \Gamma\,, \label{eq:convective-Brinkman-Forchheimer-1c}
\end{align}
\end{subequations}
where $\nu>0$ is the Brinkman coefficient (or the effective viscosity), $\alpha>0$ is the Darcy coefficient, $\tF>0$ is the Forchheimer coefficient, and $\rp$ is a given number, with $\rp\in [3,4]$.
Owing to the incompressibility of the fluid and the Dirichlet boundary condition for $\bu$, the datum $\bu_\rD\in \bH^{1/2}(\Gamma)$ must satisfy the compatibility condition
\begin{equation}\label{eq:compatibility-condition}
\int_{\Gamma} \bu_\rD\cdot\bn \,=\, 0 \,.
\end{equation}
In addition, due to \eqref{eq:convective-Brinkman-Forchheimer-1a}, and in order to guarantee uniqueness of the pressure, this unknown will be sought in the space
\begin{equation*}
\L^2_0(\Omega) \,:=\, \Big\{ q\in \L^2(\Omega) :\quad \int_{\Omega} q = 0 \Big\} \,.
\end{equation*}

Next, in order to derive a pseudostress-velocity mixed formulation for \eqref{eq:convective-Brinkman-Forchheimer-1},
in which the Dirichlet boundary conditions become natural ones, we now proceed as in \cite{cot2017-MC} (see similar approaches in \cite{cgot2016,cgor2017,gos2018}), and introduce as a further unknown the nonlinear pseudostress tensor $\bsi$, which is defined by
\begin{equation}\label{eq:sigma-definition}
\bsi := \nu\,\nabla\bu - (\bu\otimes\bu) - p\,\bbI\,.
\end{equation}
In this way, applying the matrix trace to the tensor $\bsi$ and utilizing the incompressibility condition \eqref{eq:convective-Brinkman-Forchheimer-1b}, one arrives at
\begin{equation}\label{eq:pressure}
p = -\frac{1}{d}\,\tr(\bsi + \bu\otimes\bu)\,.
\end{equation}
Hence, replacing back \eqref{eq:pressure} into \eqref{eq:sigma-definition},
we find that \eqref{eq:convective-Brinkman-Forchheimer-1} can be rewritten, equivalently, 
as follows: Find $(\bsi,\bu)$ in suitable spaces to be indicated below such that
\begin{subequations}\label{eq:convective-Brinkman-Forchheimer-2}
\begin{align}
\ds \frac{1}{\nu}\,\bsi^\rd + \frac{1}{\nu}\,(\bu\otimes \bu)^\rd & = \nabla\bu \quad \mbox{in}\quad \Omega\,, \label{eq:convective-Brinkman-Forchheimer-2a} \\[1ex]
\ds \alpha\,\bu + \tF\,|\bu|^{\rp-2}\bu - \bdiv(\bsi) & = \f \quad\,\,\,\,\,\, \mbox{in}\quad \Omega\,, \label{eq:convective-Brinkman-Forchheimer-2b} \\[1ex]
\bu & = \bu_\rD \quad\, \mbox{on}\quad \Gamma\,, \label{eq:convective-Brinkman-Forchheimer-2c} \\[1ex] 
\ds \int_{\Omega} \tr(\bsi + \bu\otimes\bu) & =  0\,. \label{eq:convective-Brinkman-Forchheimer-2d}
\end{align}
\end{subequations}
At this point we stress that, as suggested by \eqref{eq:pressure}, $p$ is eliminated from the
present formulation and computed afterwards in terms of $\bsi$ and $\bu$ by using that identity (see Appendix \ref{sec:appendix-A} for details).
This fact, justifies \eqref{eq:convective-Brinkman-Forchheimer-2d},
which aims to ensure that the resulting pressure does belong to $\L^2_0(\Omega)$.

\subsection{The variational formulation}\label{sec:variational-formulation}

In this section we derive the mixed variational formulation 
for the problem given by \eqref{eq:convective-Brinkman-Forchheimer-2}.
To that end, we multiply \eqref{eq:convective-Brinkman-Forchheimer-2a}
by a tensor $\btau\in \bbH(\bdiv;\Omega)$, integrate the resulting expression
by parts, and use the identity $\bsi^\rd:\btau = \bsi^\rd:\btau^\rd$ and
the Dirichlet boundary condition \eqref{eq:convective-Brinkman-Forchheimer-2c}, to get
\begin{equation}\label{eq:CBF-weak-formulation-1}
\frac{1}{\nu} \int_\Omega \bsi^\rd:\btau^\rd
+ \int_\Omega \bu\cdot\bdiv(\btau) 
+ \frac{1}{\nu} \int_\Omega (\bu\otimes\bu)^\rd:\btau 
\,=\, \pil\btau\bn,\bu_\rD\pir_\Gamma \quad \forall\,\btau\in \bbH(\bdiv;\Omega)\,.
\end{equation}
In order to have more flexibility for choosing the finite element subspaces, 
but at the same time avoiding the incorporation of new terms in the resulting variational 
equation, we now proceed similarly as in \cite{ggm2014} (see also \cite{gos2018,cgos2020}), 
and replace $\bu$ in the second term of the left-hand side of \eqref{eq:CBF-weak-formulation-1}
by the expression arising from \eqref{eq:convective-Brinkman-Forchheimer-2b}, 
that is
\begin{equation*}
\bu \,=\, \frac{1}{\alpha}\,\Big\{ \, \bdiv(\bsi) - \tF\,|\bu|^{\rp-2}\bu + \f \, \Big\} \qin \Omega\,.
\end{equation*}
In this way, we arrive at the variational formulation:
Find $\bsi\in \bbH(\bdiv;\Omega)$ and $\bu$ (in a suitable space to be specified below), such that
\begin{equation}\label{eq:CBF-weak-formulation-2}
\begin{array}{l}
\ds \frac{1}{\nu}\int_\Omega \bsi^\rd:\btau^\rd 
+ \frac{1}{\alpha} \int_\Omega \bdiv(\bsi)\cdot\bdiv(\btau) 
+ \frac{1}{\nu} \int_\Omega (\bu\otimes\bu)^\rd:\btau 
- \frac{\tF}{\alpha} \int_\Omega |\bu|^{\rp-2}\bu\cdot\bdiv(\btau) \\[3ex]
\ds\quad =\,  
- \frac{1}{\alpha} \int_\Omega \f\cdot\bdiv(\btau) 
+ \pil\btau\bn,\bu_\rD\pir_\Gamma
\quad \forall\,\btau\in \bbH(\bdiv;\Omega)\,.
\end{array}
\end{equation}
Since $\btau\in \bbH(\bdiv;\Omega)$ and $\rp\in [3,4]$,
the terms $(\bu\otimes\bu)^\rd:\btau$ and $|\bu|^{\rp-2}\bu\cdot\bdiv(\btau)$ forces the velocity $\bu$, and 
consequently the test function $\bv$, to live in $\bL^{\rq}(\Omega)$, with $\rq=2(\rp-1)\in [4,6]$.
In order to deal with this fact, we first observe, applying Cauchy--Schwarz 
and H\"older's inequalities, and then the continuous injection $\bi_4$ (resp. $\bi_{\rq}$) of $\bH^1(\Omega)$ 
into $\bL^4(\Omega)$ (resp. $\bL^{\rq}(\Omega)$) (cf. \cite[Theorem~1.3.4]{Quarteroni-Valli}), that
\begin{equation}\label{eq:continuous-injection-H1-L4}
\left| \int_\Omega (\bw\otimes\bu)^\rd:\btau \right| 
\,\leq\, \|\bw\|_{0,4;\Omega}\,\|\bu\|_{0,4;\Omega} \,\|\btau\|_{0,\Omega}
\,\leq\, \|\bi_4\|^2\, \|\bw\|_{1,\Omega}\,\|\bu\|_{1,\Omega}\,\|\btau\|_{\bdiv;\Omega}
\end{equation}
and
\begin{equation}\label{eq:continuous-injection-H1-Lq}
\left| \int_\Omega |\bw|^{\rp-2}\bu\cdot\bdiv(\btau) \right| 
\,\leq\, \|\bw\|^{\rp-2}_{0,\rq;\Omega}\,\|\bu\|_{0,\rq;\Omega} \,\|\bdiv(\btau)\|_{0,\Omega}
\,\leq\, \|\bi_{\rq}\|^{\rp-1}\, \|\bw\|^{\rp-2}_{1,\Omega}\,\|\bu\|_{1,\Omega}\,\|\btau\|_{\bdiv;\Omega}\,,
\end{equation}
for all $\bw,\bu\in\bH^1(\Omega)$ and $\btau\in\bbH(\bdiv;\Omega)$, 
where $\|\bi_4\|$ (resp. $\|\bi_{\rq}\|$) is the norm of the injection of $\bH^1(\Omega)$ into $\bL^4(\Omega)$ (resp. $\bL^{\rq}(\Omega)$).
However, we notice from \eqref{eq:CBF-weak-formulation-2} that the lack of a 
test function in the space where $\bu$ lives (now in $\bH^1(\Omega)$), makes 
the well-posedness analysis of \eqref{eq:CBF-weak-formulation-2} non-viable.
Then, aiming to circumvent this inconvenient, we propose to enrich our formulation 
with the following residual terms arising from the constitutive equation \eqref{eq:convective-Brinkman-Forchheimer-2a} and the Dirichlet boundary 
condition \eqref{eq:convective-Brinkman-Forchheimer-2c}:
\begin{equation}\label{eq:augmented-terms}
\kappa_1\int_\Omega \Big\{ \nabla\bu - \frac{1}{\nu}\,\bsi^\rd - \frac{1}{\nu}\,(\bu\otimes\bu)^\rd \Big\}:\nabla\bv = 0\qan
\kappa_2 \int_\Gamma \bu\cdot\bv = \kappa_2 \int_\Gamma \bu_\rD\cdot\bv \quad \forall\,\bv\in \bH^1(\Omega)\,, 
\end{equation}
where $\kappa_1$ and $\kappa_2$ are positive parameters 
to be specified later. 
According to the previous analysis, the weak formulation of the convective Brinkman--Forchheimer
problem \eqref{eq:convective-Brinkman-Forchheimer-2} reduces at first instance to:
Find $(\bsi,\bu)\in \bbH(\bdiv;\Omega)\times \bH^1(\Omega)$ such that \eqref{eq:convective-Brinkman-Forchheimer-2d}, \eqref{eq:CBF-weak-formulation-2} and \eqref{eq:augmented-terms} hold, for all $(\btau,\bv)\in \bbH(\bdiv;\Omega)\times \bH^1(\Omega)$.

However, for convenience of the subsequent analysis, we consider the decomposition (see, for instance, \cite{Gatica,Girault-Raviart}) 
\begin{equation*}
\bbH(\bdiv;\Omega) \,=\, \bbH_0(\bdiv;\Omega)\oplus \R\,\bbI\,,
\end{equation*}
where
\begin{equation*}
\bbH_0(\bdiv;\Omega) \,:=\, \Big\{ \btau\in\bbH(\bdiv;\Omega) :\quad \int_\Omega \tr(\btau) = 0 \Big\}\,.
\end{equation*}
More precisely, each $\btau\in \bbH(\bdiv;\Omega)$ can be decomposed uniquely as:
\begin{equation*}
\btau = \btau_0 + \jmath\,\bbI\,,\quad 
\mbox{ with } \btau_0\in \bbH_0(\bdiv;\Omega) \qan
\jmath := \frac{1}{d\,|\Omega|} \int_\Omega \tr(\btau) \in \R\,.
\end{equation*}
In particular, using from \eqref{eq:convective-Brinkman-Forchheimer-2d} that $\int_\Omega \tr(\bsi) = - \int_\Omega \tr(\bu\otimes\bu)$, we obtain
\begin{equation}\label{eq:sigma-decomposition}
\bsi = \bsi_0 + \ell\,\bI \quad\mbox{with} \quad \bsi_0\in \bbH_0(\bdiv;\Omega) \qan \ell:=-\frac{1}{d\,|\Omega|} \int_\Omega \tr(\bu\otimes\bu)\,,
\end{equation}
which says that $\ell$ is know explicitly in terms of $\bu$.
Therefore, in order to fully determine $\bsi$, it only remains to find its $\bbH_0(\bdiv;\Omega)$-component $\bsi_0$.
Moreover, noticing that $\btau^\rd = \btau^\rd_0$ 
and $\bdiv(\btau) = \bdiv(\btau_0)$, and using the compatibility condition 
\eqref{eq:compatibility-condition}, we deduce that both $\bsi_0$ and $\btau$ 
can be considered hereafter in $\bbH_0(\bdiv;\Omega)$.
Hence, bearing in mind the foregoing discussion, we rename $\bsi_0$ simply as $\bsi$ 
and arrive at the following mixed
formulation for the convective Brinkman--Forchheimer equations: 
Find $(\bsi,\bu)\in \bbH_0(\bdiv;\Omega)\times \bH^1(\Omega)$, such that
\begin{equation}\label{eq:CBF-augmented-weak-formulation}
\bA_{\bu}((\bsi,\bu),(\btau,\bv)) \,=\, \bF(\btau,\bv) \quad \forall\,(\btau,\bv)\in \bbH_0(\bdiv;\Omega)\times \bH^1(\Omega)\,,
\end{equation}
where, given $\bw\in \bH^1(\Omega)$, the bilinear form $\bA_{\bw}:(\bbH_0(\bdiv;\Omega)\times \bH^1(\Omega))\times (\bbH_0(\bdiv;\Omega)\times \bH^1(\Omega))\to \R$ is defined by
\begin{equation}\label{eq:bilinear-form-Aw}
\bA_\bw((\bzeta,\bz),(\btau,\bv)) := \bA((\bzeta,\bz),(\btau,\bv)) + \bB_\bw((\bzeta,\bz),(\btau,\bv))\,,
\end{equation}
with
\begin{equation}\label{eq:bilinear-form-A}
\bA((\bzeta,\bz),(\btau,\bv)) := \frac{1}{\nu}\int_\Omega \bzeta^\rd:\btau^\rd 
+ \frac{1}{\alpha} \int_\Omega \bdiv(\bzeta)\cdot\bdiv(\btau) + \kappa_1\int_\Omega \Big\{ \nabla\bz - \frac{1}{\nu}\bzeta^\rd \Big\}:\nabla\bv + \kappa_2 \int_\Gamma \bz\cdot\bv
\end{equation}
and
\begin{equation}\label{eq:bilinear-form-Bw}
\bB_\bw((\bzeta,\bz),(\btau,\bv)) :=
\frac{1}{\nu} \int_\Omega (\bw\otimes\bz)^\rd:\Big\{ \btau - \kappa_1\,\nabla\bv \Big\}
-\,\frac{\tF}{\alpha} \int_\Omega |\bw|^{\rp-2}\bz\cdot\bdiv(\btau)\,,
\end{equation}
for all $(\bzeta,\bz), (\btau,\bv)\in \bbH_0(\bdiv;\Omega)\times \bH^1(\Omega)$.
In turn, $\bF\in (\bbH_0(\bdiv;\Omega)\times \bH^1(\Omega))'$ is defined by
\begin{equation}\label{eq:functional-F}
\bF(\btau,\bv) := 
-\frac{1}{\alpha} \int_\Omega \f\cdot\bdiv(\btau)
+ \pil\btau\bn,\bu_\rD\pir_\Gamma 
+ \kappa_2 \int_\Gamma \bu_\rD\cdot\bv \,.
\end{equation}


\section{Analysis of the continuous problem}\label{sec:analysis-continuous-problem}

In this section we combine the Lax--Milgram theorem with the classical Schauder and
Banach fixed-point theorems, to prove the well-posedness of \eqref{eq:CBF-augmented-weak-formulation} 
under suitable smallness assumptions on the data.

\subsection{Preliminary results}

We begin by discussing the stability properties of the forms involved 
in \eqref{eq:CBF-augmented-weak-formulation}.
To that end, and for the sake of clarity, we set the notation
\begin{equation*}
\|(\btau,\bv)\|^2 \,:=\, \|\btau\|^2_{\bdiv;\Omega} + \|\bv\|^2_{1,\Omega} \quad \forall\,(\btau,\bv)\in \bbH_0(\bdiv;\Omega)\times \bH^1(\Omega)\,.
\end{equation*}
Next, given $\bw\in \bH^1(\Omega)$, using \eqref{eq:continuous-injection-H1-L4}--\eqref{eq:continuous-injection-H1-Lq}
and performing simple computations, we deduce from \eqref{eq:bilinear-form-A} 
and \eqref{eq:bilinear-form-Bw} that the bilinear forms $\bA$ and $\bB_\bw$, 
are bounded as indicated in what follows  
\begin{equation}\label{eq:continuity-A}
\big| \bA((\bzeta,\bz),(\btau,\bv)) \big| \,\leq\, C_\bA\,\|(\bzeta,\bz)\|\|(\btau,\bv)\| \qan
\end{equation}
\begin{equation}\label{eq:continuity-Bw}
\big| \bB_\bw((\bzeta,\bz),(\btau,\bv)) \big| 
\,\leq\, \left(\frac{1}{\nu}\,(1+\kappa_1)\,\|\bi_4\|^2\,\|\bw\|_{1,\Omega} + \frac{\tF}{\alpha}\,\|\bi_{\rq}\|^{\rp-1}\,\|\bw\|^{\rp-2}_{1,\Omega}\right)
\|\bz\|_{1,\Omega}\,\|(\btau,\bv)\| \,,
\end{equation}
where $C_\bA$ is a positive constant depending on $\nu, \alpha, \kappa_1$, and $\kappa_2$.
In addition, employing Cauchy--Schwarz's inequality, the continuity of 
the normal trace of $\bbH(\bdiv;\Omega)$ (cf. \cite[Theorem 1.7]{Gatica}) and
the trace inequality (see, e.g., \cite[Theorem 1.5]{Gatica}): $\|\bv\|_{0,\Gamma} \,\leq\, C_\Gamma\,\|\bv\|_{1,\Omega} \quad \forall\,\bv\in \bH^1(\Omega)$,
it is readily seen that $\bF$ (cf. \eqref{eq:functional-F}) is bounded:
\begin{equation}\label{eq:continuity-F}
| \bF(\btau,\bv)| 
\,\leq\, C_\bF\,\Big\{ \|\f\|_{0,\Omega} + \|\bu_\rD\|_{1/2,\Gamma} + \|\bu_\rD\|_{0,\Gamma} \Big\}\,\|(\btau,\bv)\| \,,
\end{equation}
with $C_\bF := \max\big\{1,1/\alpha,\kappa_2\,C_\Gamma \big\}$.
On the other hand, for later use, we recall from \cite[Lemma 5.3]{gm1975} that for all $p>2$ and $\bz,\by\in \R^d$ there exists a constant $c_p>0$ independent of $\bz$ and $\by$, such that
\begin{equation}\label{eq:glowinsky}
\big| |\bz|^{p-2}\bz - |\by|^{p-2}\by \big|
\,\leq\, c_p\,\big(|\bz| + |\by|\big)^{p-2} |\bz - \by| \,.
\end{equation}

Finally, we recall that there exist positive constants 
$c_1(\Omega)$ and $c_2(\Omega)$, such that (see \cite[Lemma~2.3]{Gatica} and \cite[Theorem~5.11.2]{kjf1977}, respectively, for details)
\begin{equation}\label{eq:tau-d-div-tau-inequality}
\|\btau^\rd\|^2_{0,\Omega} + \|\bdiv(\btau)\|^2_{0,\Omega} 
\,\geq\, c_1(\Omega)\,\|\btau\|^2_{0,\Omega} \quad \forall\, \btau\in \bbH_0(\bdiv;\Omega)
\end{equation}
\begin{equation}\label{eq:grad-v-v-inequality}
\qan \|\nabla\bv\|^2_{0,\Omega} + \|\bv\|^2_{0,\Gamma} 
\,\geq\, c_2(\Omega)\,\|\bv\|^2_{1,\Omega} \quad \forall\, \bv\in \bH^1(\Omega).
\end{equation}

Then, we establish next the ellipticity of the bilinear form $\bA$.
\begin{lem}\label{lem:Ellipticity-A}
Assume that $\kappa_1\in (0,2\,\nu)$ and $\kappa_2\in (0,+\infty)$. 
Then, there exists $\alpha_\bA > 0$, such that there holds
\begin{equation}\label{eq:ellipticity-A}
\bA((\btau,\bv),(\btau,\bv)) \,\geq\, \alpha_\bA\,\|(\btau,\bv)\|^2 \quad 
\forall\,(\btau,\bv)\in \bbH_0(\bdiv;\Omega)\times \bH^1(\Omega)\,.
\end{equation}
\end{lem}
\begin{proof}
Let $(\btau,\bv)\in \bbH_0(\bdiv;\Omega)\times \bH^1(\Omega)$.	
Then, from the definition of $\bA$ (cf. \eqref{eq:bilinear-form-A}), using Young's inequality 
and simple algebraic computations, we find that
\begin{equation*}
\bA((\btau,\bv),(\btau,\bv)) \,\geq\, \frac{1}{\nu}\left(1 - \frac{\kappa_1}{2\,\nu}\right)\,\|\btau^\rd\|^2_{0,\Omega} 
+ \frac{1}{\alpha}\,\|\bdiv(\btau)\|^2_{0,\Omega} 
+ \frac{\kappa_1}{2}\,\|\nabla\bv\|^2_{0,\Omega} + \kappa_2\,\|\bv\|^2_{0,\Gamma} \,. 
\end{equation*}
Then, assuming the stipulated ranges on $\kappa_1$ and $\kappa_2$,
and applying inequalities \eqref{eq:tau-d-div-tau-inequality} and \eqref{eq:grad-v-v-inequality},
we can define the positive constants
\begin{equation*}
\alpha_0 := \min\left\{ \frac{1}{\nu}\left(1 - \frac{\kappa_1}{2\,\nu}\right), \frac{1}{2\,\alpha} \right\} \,,\quad
\alpha_1 := \min\left\{ \alpha_0\,c_1(\Omega), \frac{1}{2\,\alpha} \right\} \,,\qan
\alpha_2 := c_2(\Omega)\,\min\left\{ \frac{\kappa_1}{2}, \kappa_2 \right\} \,,
\end{equation*}
which allows us to conclude \eqref{eq:ellipticity-A} with $\alpha_\bA := \min\{ \alpha_1, \alpha_2 \}$.
\end{proof}

\begin{rem}\label{rem:kappa-parameters}
We note that for computational purposes, and in order to maximize 
the ellipticity constant $\alpha_{\bA}$ (cf. \eqref{eq:ellipticity-A}), 
we can choose explicitly the parameter $\kappa_1$ and $\kappa_2$ by taking 
$\kappa_1$ as the middle point of its feasible range and $\kappa_2 \ge \frac{\kappa_1}{2}$.
More precisely, we can simply take
\begin{equation*}
\kappa_1 = \nu \qan \kappa_2 \ge \frac{\nu}{2}\,.
\end{equation*}
\end{rem}

\subsection{A fixed point strategy}

We begin the solvability analysis of \eqref{eq:CBF-augmented-weak-formulation} by defining the operator
$\bT:\bH^1(\Omega)\to \bH^1(\Omega)$ by
\begin{equation}\label{eq:operator-T}
\bT(\bw) \,:=\, \bu \quad \forall\,\bw\in \bH^1(\Omega) \,,
\end{equation}
where $\bu$ is the second component of the unique solution (to be confirmed below) 
of the problem: Find $(\bsi, \bu)\in \bbH_0(\bdiv;\Omega)\times \bH^1(\Omega)$ such that 
\begin{equation}\label{eq:CBF-augmented-weak-formulation-T}
\bA_{\bw}((\bsi,\bu),(\btau,\bv)) \,=\, \bF(\btau,\bv) \quad \forall (\btau,\bv)\in \bbH_0(\bdiv;\Omega)\times \bH^1(\Omega) \,.
\end{equation} 
Hence, it is not difficult to see that $(\bsi,\bu)\in \bbH_0(\bdiv;\Omega)\times \bH^1(\Omega)$
is a solution of \eqref{eq:CBF-augmented-weak-formulation} if and only if $\bu\in \bH^1(\Omega)$ is a fixed-point of $\bT$, that is
\begin{equation}\label{eq:fixed-point-T}
\bT(\bu) \,=\, \bu \,.
\end{equation}

In this way, in what follows we focus on proving that $\bT$ possesses a unique fixed-point.
However, we remark in advance that the definition of $\bT$ will make sense only in a closed ball of $\bH^1(\Omega)$.

We begin by establishing a result that provides sufficient conditions under which the operator $\bT$ (cf. \eqref{eq:operator-T}) is well-defined, or equivalently, the problem \eqref{eq:CBF-augmented-weak-formulation-T} is well-posed.
\begin{lem}\label{lem:T-well-defined}
Assume $\kappa_1$ and $\kappa_2$ as in Lemma \ref{lem:Ellipticity-A}.
Let $r\in (0,r_0]$, with $r_0 \,=\, \min\left\{ r_1, r_2 \right\}$, and
\begin{equation}\label{eq:radius-r0-r1-r2}
r_1 \,=\, \frac{\nu\,\alpha_\bA}{4\,(1 + \kappa_1)\,\|\bi_4\|^2} \qan
r_2 \,=\, \left(\frac{\alpha\,\alpha_\bA}{4\,\tF\,\|\bi_{\rq}\|^{\rp-1}}\right)^{1/(\rp-2)}\,, 
\end{equation} 
and let $\f\in \bL^2(\Omega)$ and $\bu_\rD\in \bH^{1/2}(\Gamma)$.
Then, the problem \eqref{eq:CBF-augmented-weak-formulation-T} has 
a unique solution $(\bsi,\bu)\in \bbH_0(\bdiv;\Omega)\times \bH^1(\Omega)$ 
for each $\bw\in \bH^1(\Omega)$ such that $\|\bw\|_{1,\Omega} \leq r$.
Moreover, there holds
\begin{equation}\label{eq:Tw-bound}
\|\bT(\bw)\|_{1,\Omega} \,=\, \|\bu\|_{1,\Omega} 
\,\leq\,\|(\bsi,\bu)\|
\,\leq\, \frac{2\,C_{\bF}}{\alpha_\bA}\,\Big\{ \|\f\|_{0,\Omega} + \|\bu_\rD\|_{1/2,\Gamma} + \|\bu_\rD\|_{0,\Gamma} \Big\} \,,
\end{equation}
with $C_\bF$ and $\alpha_\bA$ satisfying \eqref{eq:continuity-F} and \eqref{eq:ellipticity-A}, respectively.
\end{lem}
\begin{proof}
First, given $\bw\in \bH^1(\Omega)$, we observe from \eqref{eq:bilinear-form-Aw}, \eqref{eq:bilinear-form-A} and \eqref{eq:bilinear-form-Bw} that 
$\bA_\bw, \bA$, and $\bB_\bw$ are clearly bilinear forms.
Then, using the ellipticity of $\bA$ and the continuity bound of $\bB_{\bw}$ (cf. \eqref{eq:ellipticity-A}, \eqref{eq:continuity-Bw}), we deduce that for all $(\btau,\bv)\in \bbH_0(\bdiv;\Omega)\times \bH^1(\Omega)$ there holds 
\begin{equation*}
\bA_\bw((\btau,\bv),(\btau,\bv)) 
\,\geq\, \left\{\alpha_\bA - \left(\frac{1}{\nu}\,(1 + \kappa_1)\,\|\bi_4\|^2\|\bw\|_{1,\Omega} + \frac{\tF}{\alpha}\,\|\bi_{\rq}\|^{\rp-1}\|\bw\|^{\rp-2}_{1,\Omega}\right)\right\}\|(\btau,\bv)\|^2 \,. 
\end{equation*}
Consequently, requiring now $\|\bw\|_{1,\Omega} \leq r_0$, with $r_0:=\min\{r_1,r_2\}$
and $r_1,r_2$ as in \eqref{eq:radius-r0-r1-r2}, we get 
\begin{equation}\label{eq:two-bounds-for-w}
\frac{1}{\nu}\,(1 + \kappa_1)\,\|\bi_4\|^2\|\bw\|_{1,\Omega} \,\leq\, \frac{\alpha_\bA}{4} \qan
\frac{\tF}{\alpha}\|\bi_{\rq}\|^{\rp-1}\|\bw\|^{\rp-2}_{1,\Omega} \,\leq\, \frac{\alpha_\bA}{4}\,,
\end{equation}
which yields
\begin{equation}\label{eq:ellipticity-Aw}
\bA_\bw((\btau,\bv),(\btau,\bv)) 
\,\geq\, \frac{\alpha_\bA}{2}\,\|(\btau,\bv)\|^2 
\quad \forall\, (\btau,\bv)\in \bbH_0(\bdiv;\Omega)\times \bH^1(\Omega) \,.
\end{equation}
In turn, using \eqref{eq:continuity-A}, \eqref{eq:continuity-Bw}, and \eqref{eq:two-bounds-for-w}, we deduce that $\bA_{\bw}$ is bounded as follows
\begin{equation}\label{eq:continuity-Aw}
\big| \bA_\bw((\bzeta,\bz),(\btau,\bv)) \big| 
\,\leq\, \Big( C_\bA + \frac{\alpha_\bA}{2} \Big)\,\|(\bzeta,\bz)\|\,\|(\btau,\bv)\| \,,
\end{equation}
for all $(\bzeta,\bz), (\btau,\bv)\in \bbH_0(\bdiv;\Omega)\times \bH^1(\Omega)$.

Summing up, and owing to the hypotheses on $\kappa_1$ and $\kappa_2$, we
have proved that for any $\bw\in \bH^1(\Omega)$ such that $\|\bw\|_{1,\Omega} \leq r_0$, 
the bilinear form $\bA_\bw$ and the functional $\bF$ satisfy the hypotheses of the Lax--Milgram theorem 
(see, e.g., \cite[Theorem 1.1]{Gatica}), which guarantees the well-posedness of \eqref{eq:CBF-augmented-weak-formulation-T}.
Finally, testing \eqref{eq:CBF-augmented-weak-formulation-T} with $(\btau,\bv)=(\bsi,\bu)$,
using \eqref{eq:ellipticity-Aw}, and 
the continuity bound of $\bF$ (cf. \eqref{eq:continuity-F}), we readily obtain that
\begin{equation*}
\frac{\alpha_\bA}{2}\,\|(\bsi,\bu)\| \,\leq\, C_\bF\,\Big\{ \|\f\|_{0,\Omega} + \|\bu_\rD\|_{1/2,\Gamma} + \|\bu_\rD\|_{0,\Gamma} \Big\} \,,
\end{equation*}
which implies \eqref{eq:Tw-bound} and complete the proof.
\end{proof}

\subsection{Well-posedness of the continuous problem}

Having proved the well-posedness of the problem \eqref{eq:CBF-augmented-weak-formulation-T},
which ensure that the operator $\bT$ is well defined, we now aim to establish the existence
of a unique fixed point of the operator $\bT$. For this purpose, in what follows we verify
the hypothesis of the Schauder and Banach fixed-point theorems.
We begin the analysis with the following straightforward consequence of Lemma \ref{lem:T-well-defined}.
\begin{lem}\label{lem:imagen-T-in-Wr}
Given $r\in (0,r_0]$, with $r_0:=\min\{r_1,r_2\}$ and $r_1,r_2$ as in \eqref{eq:radius-r0-r1-r2}, we let $\bW_r$
be the closed and convex subset of $\bH^1(\Omega)$ defined by
\begin{equation}\label{eq:ball-Wr}
\bW_r \,:=\, \Big\{ \bw\in \bH^1(\Omega) :\quad \|\bw\|_{1,\Omega} \,\leq\, r \Big\} \,.
\end{equation}
In addition, we take the stabilization parameters $\kappa_1$ and $\kappa_2$ as in Lemma \ref{lem:Ellipticity-A}, and assume that the data satisfy
\begin{equation}\label{eq:data-assumption-1}
\frac{2\,C_{\bF}}{\alpha_\bA}\,\Big\{ \|\f\|_{0,\Omega} + \|\bu_\rD\|_{1/2,\Gamma} + \|\bu_\rD\|_{0,\Gamma} \Big\} 
\,\leq\, r \,,
\end{equation}
with $C_\bF$ and $\alpha_\bA$ satisfying \eqref{eq:continuity-F} and \eqref{eq:ellipticity-A}, respectively.
Then there holds $\bT(\bW_r) \subseteq \bW_r$.
\end{lem}

We continue by providing an estimate needed to derive the 
continuity and compactness properties of the operator $\bT$.
To that end, we first observe that there exists $c_\rp>0$ such that
\begin{equation}\label{eq:glowinsky-2}
\big| |\bw|^{\rp-2} - |\wt{\bw}|^{\rp-2} \big|
\,=\, \big| |\bw|^{\rp-3}(|\bw|,\0) - |\wt{\bw}|^{\rp-3}(|\wt{\bw}|,\0) \big|
\,\leq\, c_\rp\,\big(|\bw| + |\wt{\bw}|\big)^{\rp-3} |\bw - \wt{\bw}| \,,
\end{equation}
which follows from \eqref{eq:glowinsky} with the setting 
$p=\rp-1\in [2,3]$, $\bz=(|\bw|,\0)$, $\by=(|\wt{\bw}|,\0)$, and $\0\in \R^{d-1}$.
The aforementioned result is established now.
\begin{lem}\label{lem:continuity-compactness-T}
Let $r\in (0,r_0]$, with $r_0:=\min\{r_1,r_2\}$ 
and $r_1,r_2$ as in \eqref{eq:radius-r0-r1-r2}, and let $\bW_r$ given by \eqref{eq:ball-Wr}.
Then, for each $\rq\in [4,6]$, there holds
\begin{equation}\label{eq:continuity-compactness-T}
\begin{array}{l}
\ds \|\bT(\bw) - \bT(\wt{\bw})\|_{1,\Omega}  \\[2ex]
\ds\quad \leq\, 
\frac{C_{\bF}}{\alpha_\bA\,r_0}
\,\Big\{ \|\f\|_{0,\Omega} + \|\bu_\rD\|_{1/2,\Gamma} + \|\bu_\rD\|_{0,\Gamma} \Big\}\, \bigg(\frac{1}{\|\bi_4\|}\,\|\bw - \wt{\bw}\|_{0,4;\Omega} + \frac{2^{\rp-3}\,c_\rp}{\|\bi_\rq\|}\,\|\bw - \wt{\bw}\|_{0,\rq;\Omega} \bigg) \,,
\end{array}
\end{equation}
for all $\bw, \wt{\bw}\in \bW_r$.
\end{lem}
\begin{proof}
Given $\bw, \wt{\bw}\in \bW_r$, we let $\bu:=\bT(\bw)$ and $\wt{\bu}:=\bT(\wt{\bw})$.
According to the definition of $\bT$ (cf. \eqref{eq:CBF-augmented-weak-formulation-T})
and the definitions of the forms $\bA_{\bw}$ and $\bB_{\bw}$ (cf. \eqref{eq:bilinear-form-Aw}, \eqref{eq:bilinear-form-Bw}), it follows that
\begin{equation*}
\bA_{\wt{\bw}}((\bsi,\bu) - (\wt{\bsi},\wt{\bu}), (\btau,\bv) ) 
\,=\, -(\bB_{\bw} - \bB_{\wt{\bw}})((\bsi,\bu), (\btau,\bv) ) 
\quad \forall\, (\btau,\bv)\in \bbH_0(\bdiv;\Omega)\times \bH^1(\Omega) \,.
\end{equation*}
Hence, taking $(\btau,\bv) = (\bsi,\bu) - (\wt{\bsi},\wt{\bu})$ in the foregoing identity,
and then employing the ellipticity of $\bA_{\bw}$ (cf. \eqref{eq:ellipticity-Aw}),
the continuity of $\bB_{\bw}$ (cf. \eqref{eq:continuity-Bw})
in combination with Cauchy--Schwarz and H\"older's inequalities, and 
\eqref{eq:glowinsky-2}, we readily get
\begin{equation*}
\begin{array}{l}
\ds \frac{\alpha_\bA}{2}\,\|(\bsi - \wt{\bsi},\bu - \wt{\bu})\|^2
\,\leq\, \bigg(\frac{1}{\nu}\,(1 + \kappa_1)\,\|\bi_4\|\,\|\bw - \wt{\bw}\|_{0,4;\Omega} \\[2ex] 
\ds\quad  
+\, \frac{\tF}{\alpha}\,c_\rp\,\|\bi_{\rq}\|^{\rp-2}\big( \|\bw\|_{1,\Omega} + \|\wt{\bw}\|_{1,\Omega} \big)^{\rp-3}\,\|\bw - \wt{\bw}\|_{0,\rq;\Omega} \bigg)
\,\|\bu\|_{1,\Omega}\,
\|(\bsi - \wt{\bsi},\bu - \wt{\bu})\| \,.
\end{array}
\end{equation*}
Then, using the definition of $r_1$ and $r_2$ (cf. \eqref{eq:radius-r0-r1-r2}), 
the fact that both $\|\bw\|_{1,\Omega}$ and $\|\wt{\bw}\|_{1,\Omega}$
are bounded by $r_2$, and simple algebraic manipulations, we obtain
\begin{equation}\label{eq:continuity-T-1}
\|(\bsi - \wt{\bsi},\bu - \wt{\bu})\| 
\,\leq\, \frac{1}{2}\,\|\bu\|_{1,\Omega}\,\bigg(\frac{1}{r_1\,\|\bi_4\|}\,\|\bw - \wt{\bw}\|_{0,4;\Omega} + \frac{2^{\rp-3}\,c_\rp}{r_2\,\|\bi_\rq\|}\,\|\bw - \wt{\bw}\|_{0,\rq;\Omega} \bigg) \,.
\end{equation}
Finally, from \eqref{eq:continuity-T-1}, noting that both $1/r_1$ and $1/r_2$ are bounded by $1/r_0$ and bounding $\|\bu\|_{1,\Omega}$ by \eqref{eq:Tw-bound} instead of by $r$, 
we obtain \eqref{eq:continuity-compactness-T} and conclude the proof.
\end{proof}

Owing to the above analysis, we establish now the announced properties of the operator $\bT$.
\begin{lem}\label{lem:Hypotheses-Schauder-T}
Let $r\in (0,r_0]$, with $r_0 := \min\{r_1,r_2\}$, and $r_1, r_2$ as in \eqref{eq:radius-r0-r1-r2}. Assume that the stabilization parameters 
$\kappa_1$ and $\kappa_2$ are taken as in Lemma \ref{lem:Ellipticity-A}, 
and that the data satisfy \eqref{eq:data-assumption-1}. 
Then $\bT: \bW_r \to \bW_r$  is continuous and $\ov{\bT(\bW_r)}$ is compact.
\end{lem}
\begin{proof}
The required result follows straightforwardly from estimate \eqref{eq:continuity-compactness-T}, 
the compactness of the injections $\bi_4 :\bH^1(\Omega)\to \bL^4(\Omega)$ and 
$\bi_\rq :\bH^1(\Omega)\to \bL^\rq(\Omega)$, with $\rq\in [4,6]$ when $d=2$ 
and $\rq\in [4,6)$ when $d=3$ (see, e.g., \cite[Theorem~1.3.5]{Quarteroni-Valli}), 
and the well-known fact that every bounded sequence in a Hilbert space has a weakly 
convergent subsequence. 
We omit further details and refer to \cite[Lemma~3.8]{cgot2016}.
\end{proof}

Finally, the main result of this section is stated as follows.
\begin{thm}\label{thm:fixed-point-T}
Let $\rp\in [3,4]$.	
Assume the same hypothesis of Lemma \ref{lem:Hypotheses-Schauder-T}. Then
the operator $\bT$ has a fixed point $\bu\in \bW_r$ (cf. \eqref{eq:ball-Wr}).
Equivalently, the continuous problem \eqref{eq:CBF-augmented-weak-formulation} 
has a solution $(\bsi, \bu)\in \bbH_0(\bdiv;\Omega)\times \bH^1(\Omega)$ 
with $\bu\in \bW_r$. Moreover, there holds
\begin{equation}\label{eq:dependence-continuous-CBF}
\|(\bsi,\bu)\| 
\,\leq\, \frac{2\,C_{\bF}}{\alpha_\bA}\, \Big\{ \|\f\|_{0,\Omega} + \|\bu_\rD\|_{1/2,\Gamma} + \|\bu_\rD\|_{0,\Gamma} \Big\} \,.
\end{equation}	 
In addition, if the data satisfy
\begin{equation}\label{eq:data-assumption-2}
\big( 1 + 2^{\rp-3}\,c_\rp \big)\,\frac{C_{\bF}}{\alpha_\bA r_0}
\,\Big\{ \|\f\|_{0,\Omega} + \|\bu_\rD\|_{1/2,\Gamma} + \|\bu_\rD\|_{0,\Gamma} \Big\} < 1\,,
\end{equation}
then the aforementioned fixed point (equivalently, the solution of 
\eqref{eq:CBF-augmented-weak-formulation}) is unique. 
\end{thm}
\begin{proof}
The equivalence between \eqref{eq:CBF-augmented-weak-formulation} and the fixed point 
equation \eqref{eq:fixed-point-T}, together with Lemmas \ref{lem:imagen-T-in-Wr} and 
\ref{lem:Hypotheses-Schauder-T}, confirm the existence of solution of \eqref{eq:CBF-augmented-weak-formulation} 
as a direct application of the Schauder fixed-point theorem \cite[Theorem~9.12-1(b)]{Ciarlet}.
In addition, it is clear that the estimate \eqref{eq:dependence-continuous-CBF} follows from 
\eqref{eq:Tw-bound}. On the other hand, using the estimate 
\eqref{eq:continuity-compactness-T} and the continuous injection $\bi_4$ (resp. $\bi_\rq$) of $\bH^1(\Omega)$ into 
$\bL^4(\Omega)$ (resp. $\bL^\rq(\Omega)$, with $\rq\in [4,6]$), we easily obtain
\begin{equation}\label{eq:continuity-Tw}
\|\bT(\bw) - \bT(\wt{\bw})\|_{1,\Omega} 
\,\leq\, \big( 1 + 2^{\rp-3}\,c_\rp \big)\,\frac{C_{\bF}}{\alpha_\bA\,r_0}
\,\Big\{ \|\f\|_{0,\Omega} + \|\bu_\rD\|_{1/2,\Gamma} + \|\bu_\rD\|_{0,\Gamma} \Big\}\,\|\bw - \wt{\bw}\|_{1,\Omega} \,,
\end{equation} 
which, thanks to \eqref{eq:data-assumption-2} and the Banach fixed-point theorem, 
yields the uniqueness.
\end{proof}

\section{The Galerkin scheme}\label{sec:Galerkin-scheme}

In this section, we introduce and analyze the corresponding Galerkin scheme 
for the mixed formulation \eqref{eq:CBF-augmented-weak-formulation}. 
The solvability of this scheme is addressed following analogous tools to those employed
throughout Section \ref{sec:analysis-continuous-problem}.
Finally, we derive the corresponding C\'ea estimate and rates of convergence of the Galerkin scheme.

\subsection{Discrete setting}\label{sec:discrete-setting}

We first let $\{\cT_h\}_{h>0}$ be a regular family of triangulations of $\ov{\Omega}$ 
by triangles $T$ (respectively tetrahedra $T$ in $\R^3$), and set $h:=\max\big\{ h_T:\,\, T\in \cT_h \big\}$.
In turn, given an integer $l\geq 0$ and a subset $S$ of $\R^d$, we denote by $\rP_l(S)$
the space of polynomials of total degree at most $l$ defined on $S$.
Hence, for each integer $k\geq 0$ and for each $T\in \cT_h$, we define the local Raviart--Thomas 
space of order $k$ as
\begin{equation*}
\bRT_k(T) := \bP_k(T)\oplus \wt{\rP}_k(T)\,\bx\,,
\end{equation*}
where $\bx:= (x_1,\dots, x_d)^\rt$ is a generic vector of $\R^d$,
$\wt{\rP}_k(T)$ is the space of polynomials of total degree equal to $k$ defined on $T$,
and, according to the convention in Section \ref{sec:introduction}, we set $\bP_k(T):=[\rP_k(T)]^d$.
In this way, introducing the finite element subspaces:
\begin{subequations}\label{eq:subspaces}
\begin{align}
\ds \bbH_{h}^\bsi & := \Big\{ \btau_h\in \bbH_0(\bdiv;\Omega) :\quad \bc^\rt\btau_h|_T\in \bRT_k(T),\quad \forall\,\bc\in \R^d,\quad \forall\,T\in \cT_h \Big\} \,, \label{eq:subspaces-a} \\
\ds \bH_h^\bu & := \Big\{ \bv_h\in \bC(\ov{\Omega}) :\quad \bv_h|_T\in \bP_{k+1}(T) \,,\quad \forall\,T\in \cT_h \Big\} \,, \label{eq:subspaces-b}
\end{align}
\end{subequations}
the Galerkin scheme for \eqref{eq:CBF-augmented-weak-formulation} reads: 
Find $(\bsi_h, \bu_h)\in \bbH_h^\bsi \times \bH_h^\bu$ such that 
\begin{equation}\label{eq:CBF-augmented-discrete-formulation-H0}
\bA_{\bu_h}((\bsi_h,\bu_h),(\btau_h,\bv_h)) \,=\, \bF(\btau_h,\bv_h) \quad \forall\,(\btau_h,\bv_h)\in \bbH_h^\bsi \times \bH_h^\bu \,.
\end{equation} 

Similarly to the continuous context, in order to analyze problem \eqref{eq:CBF-augmented-discrete-formulation-H0} we rewrite it equivalently as a fixed-point problem.
Indeed, we define the operator $\bT_\ttd: \bH_h^\bu \to \bH_h^\bu$ by
\begin{equation*}
\bT_\ttd(\bw_h) \,=\, \bu_h \quad \forall\,\bw_h\in \bH_h^\bu\,,
\end{equation*}
where $(\bsi_h, \bu_h)$ is the unique solution of the discrete version of the problem \eqref{eq:CBF-augmented-weak-formulation-T}:
Find $(\bsi_h,\bu_h)\in \bbH_h^\bsi\times \bH_h^\bu$ such that 
\begin{equation}\label{eq:CBF-augmented-discrete-formulation-FixedP}
\bA_{\bw_h}((\bsi_h,\bu_h),(\btau_h,\bv_h)) \,=\, \bF(\btau_h,\bv_h) \quad \forall\,(\btau_h,\bv_h)\in \bbH_h^\bsi \times \bH_h^\bu \,,
\end{equation} 
where the bilinear form $\bA_{\bw_h}$ and the functional $\bF$ are defined in \eqref{eq:bilinear-form-Aw} (with $\bw_h$ instead of $\bw$) and \eqref{eq:functional-F}, respectively. 
Therefore solving \eqref{eq:CBF-augmented-discrete-formulation-H0} is equivalent to seeking a fixed point of the operator $\bT_\ttd$, that is:
Find $\bu_h\in \bH_h^\bu$ such that
\begin{equation}\label{eq:fixed-point-Td}
\bT_\ttd (\bu_h) = \bu_h \,.
\end{equation}

\subsection{Solvability Analysis}\label{sec:solv}

We begin by remarking that the same tools employed in the proof of Lemma 
\ref{lem:T-well-defined} can be used now to prove the unique solvability 
of \eqref{eq:CBF-augmented-discrete-formulation-H0}. In fact, under 
the same assumptions from Lemma \ref{lem:Ellipticity-A} 
on the stabilization parameters, we find that for each $\bw_h\in \bH^{\bu}_h$,
$\bA_{\bw_h}$ is bounded and elliptic on $\bbH^{\bsi}_h\times \bH^{\bu}_h$ with the same 
constants obtained in \eqref{eq:continuity-Aw} and \eqref{eq:ellipticity-Aw}, respectively. 
In turn, from \eqref{eq:functional-F} and \eqref{eq:continuity-F}, 
the functional $\bF$ is linear and bounded. 
The foregoing discussion and the Lax--Milgram theorem allow us to conclude the following result.
\begin{lem}\label{lem:Tdw-bound}
Let $r\in (0,r_0]$, with $r_0 := \min\{r_1,r_2\}$, and $r_1, r_2$ as in \eqref{eq:radius-r0-r1-r2} 
and assume $\kappa_1$, $\kappa_2$ as in Lemma \ref{lem:Ellipticity-A}. 
Then, for each $\bw_h \in \bH_h^\bu$ satisfying $\|\bw_h\|_{1,\Omega} \leq r$, 
the problem \eqref{eq:CBF-augmented-discrete-formulation-FixedP} has a unique solution 
$(\bsi_h,\bu_h)\in \bbH_h^\bsi \times \bH_h^\bu$. Moreover, there holds
\begin{equation}\label{eq:Tdw-bound}
\|\bT_\ttd(\bw_h)\|_{1,\Omega} \,=\, \|\bu_h\|_{1,\Omega}
\,\leq\, \|(\bsi_h,\bu_h)\| \,\leq\, \frac{2\,C_{\bF}}{\alpha_\bA}\,
\Big\{ \|\f\|_{0,\Omega} + \|\bu_\rD\|_{1/2,\Gamma} + \|\bu_\rD\|_{0,\Gamma} \Big\} \,,
\end{equation}
with $C_\bF$ and $\alpha_\bA$ satisfying \eqref{eq:continuity-F} and \eqref{eq:ellipticity-A}, respectively.
\end{lem}

We now proceed to analyze the fixed-point equation \eqref{eq:fixed-point-Td}.
We begin with the discrete version of Lemma \ref{lem:imagen-T-in-Wr}, whose proof, follows straightforwardly from Lemma \ref{lem:Tdw-bound}.
\begin{lem}\label{lem:clausure-wr}
Let $r\in (0,r_0]$, with $r_0 := \min\{r_1,r_2\}$, and $r_1, r_2$ as in \eqref{eq:radius-r0-r1-r2},
and let $\wt{\bW}_r$ be the bounded subset of $\bH_h^\bu$ defined by
\begin{equation}\label{eq:discrete-Wr}
\wt{\bW}_r \,:=\, \Big\{ \bw_h\in\bH_h^\bu : \quad \|\bw_h\|_{1,\Omega} \,\leq\, r \Big\} \,.
\end{equation}
Assume $\kappa_1$ and $\kappa_2$ as in Lemma \ref{lem:Ellipticity-A}
and that the data $\f$ and $\bu_\rD$ satisfy \eqref{eq:data-assumption-1}.
Then $\bT_\ttd(\wt{\bW}_r) \subseteq \wt{\bW}_r$.
\end{lem}

Next, we address the discrete counterpart of \eqref{eq:continuity-Tw} (see also Lemma \ref{lem:continuity-compactness-T}),
whose proof, being almost verbatim of the continuous one, is omitted.
Thus, we simply state the corresponding result as follows.
\begin{lem}\label{lem:continuity-Th}
Let $\rp\in [3,4]$ and $r\in (0,r_0]$, with $r_0 := \min\{r_1,r_2\}$, and $r_1, r_2$ as in \eqref{eq:radius-r0-r1-r2},
and let $C_\bF, \alpha_{\bA}$ satisfying \eqref{eq:continuity-F}, \eqref{eq:ellipticity-A}, respectively.
Then, there holds
\begin{equation}\label{eq:discrete-continuous-Th}
\|\bT_\ttd(\bw_h) - \bT_\ttd(\wt{\bw}_h)\|_{1,\Omega} 
\,\leq\, \big(1 + 2^{\rp-3}c_\rp\big)\,\frac{C_{\bF}}{\alpha_\bA\,r_0}\,
\Big\{ \|\f\|_{0,\Omega} + \|\bu_\rD\|_{1/2,\Gamma} + \|\bu_\rD\|_{0,\Gamma} \Big\}\,\|\bw_h - \wt{\bw}_h\|_{1,\Omega} \,,
\end{equation}
for all $\bw_h, \wt{\bw}_h \in \wt{\bW}_r$.
\end{lem}

We are now in position of establishing the well-posedness of \eqref{eq:CBF-augmented-discrete-formulation-H0}.
\begin{thm}\label{th:contractive}
Let $\rp\in [3,4]$.	
Assume the same hypothesis of Lemma \ref{lem:clausure-wr}.
Then, the operator $\bT_\ttd$ has a fixed point $\bu_h \in \wt{\bW}_r$ (cf. \eqref{eq:discrete-Wr}). 
Equivalently, the discrete problem \eqref{eq:CBF-augmented-discrete-formulation-H0} has a solution $(\bsi_h,\bu_h)\in \bbH_h^\bsi \times \bH_h^\bu$, with $\bu_h\in \wt{\bW}_r$. 
Moreover, there holds
\begin{equation}\label{eq:discrete-dependence-continuous-CBF}
\|(\bsi_h,\bu_h)\|
\,\leq\, \frac{2\,C_{\bF}}{\alpha_\bA}\,\Big\{ \|\f\|_{0,\Omega} + \|\bu_\rD\|_{1/2,\Gamma} + \|\bu_\rD\|_{0,\Gamma} \Big\} \,.
\end{equation}
In addition, if the data satisfy \eqref{eq:data-assumption-2},
%
then the aforementioned fixed point (equivalently, the solution of \eqref{eq:CBF-augmented-discrete-formulation-H0}) is unique.
\end{thm}
\begin{proof}
It follows similarly to the proof of Theorem \ref{thm:fixed-point-T}. 
Indeed, we first notice from Lemma \ref{lem:clausure-wr} that $\bT_\ttd$ maps the ball $\wt{\bW}_r$ 
into itself. In turn, it is easy to see from \eqref{eq:discrete-continuous-Th}  
that $\bT_\ttd:\wt{\bW}_r\to \wt{\bW}_r$ 
is continuous, and hence the existence result follows from 
the Brouwer fixed-point theorem \cite[Theorem 9.9-2]{Ciarlet}. 
In addition, it is clear that the estimate \eqref{eq:discrete-dependence-continuous-CBF} follows from 
\eqref{eq:Tdw-bound}.
On the other hand, the estimate \eqref{eq:discrete-continuous-Th} and the assumption  \eqref{eq:data-assumption-2} show that $\bT_\ttd$ is a contraction mapping,
which, thanks to the Banach fixed-point theorem, implies the uniqueness result and concludes the proof. 
\end{proof}

%
%
%
%
%
%
%

\subsection{\textit{A priori} error analysis}\label{sec:a-priori-error-analysis}

In this section, we first derive the C\'ea estimate for the Galerkin scheme \eqref{eq:CBF-augmented-discrete-formulation-H0} with the finite element subspaces given by \eqref{eq:subspaces-a}-\eqref{eq:subspaces-b},
and then use the approximation properties of the latter to establish the corresponding rates of convergence.
In fact, let $(\bsi,\bu)\in \bbH_0(\bdiv;\Omega)\times \bH^1(\Omega)$, with $\bu\in \bW_r$, be the unique solution of the problem \eqref{eq:CBF-augmented-weak-formulation}, and let
$(\bsi_h,\bu_h)\in \bbH_h^\bsi\times \bH_h^\bu$, with $\bu_h\in \wt{\bW}_r$, be the unique solution of the discrete problem \eqref{eq:CBF-augmented-discrete-formulation-H0}.
Then, we are interested in obtaining an {\it a priori} estimate for the error
\begin{equation*}
\|(\bsi,\bu) - (\bsi_h,\bu_h)\|^2
\,:=\, \|\bsi - \bsi_h\|^2_{\bdiv;\Omega} + \|\bu - \bu_h\|^2_{1,\Omega} \,.
\end{equation*}
For this purpose, we establish next an ad-hoc Strang-type estimate.
Hereafter, given a subspace $H_h$ of a generic Hilbert space $(H,\|\cdot\|_H)$, we set as usual
\begin{equation*}
\dist(x,H_h) \,:=\, \inf_{x_h\in H_h} \|x - x_h\|_H \quad \mbox{for all } x\in H\,.
\end{equation*}

\begin{lem}\label{lem:Strang-lemma}
Let $H$ be a Hilbert space, $F\in H'$, and let $a: H \times H \to \R$ be a bounded and $H$-elliptic bilinear form, with respective constants $\|a\|$ and $\alpha$.
In addition, let $\{ H_h \}_{h>0}$ be a sequence of finite dimensional subspaces of $H$, and for each $h>0$ consider a bounded bilinear form $a_h:H_h\times H_h\to \R$, with boundedness constant $\|a_h\|$ independent of $h$.
Assume that the family $\{ a_h \}_{h>0}$ is uniformly elliptic, that is, there exists a constant $\wt{\alpha}>0$, independent of $h$, such that
\begin{equation*}
a_h(v_h,v_h) \,\geq\, \wt{\alpha}\,\|v_h\|^2_H \quad \forall\,v_h\in H_h\,,\quad \forall\,h>0\,.
\end{equation*}
In turn, let $u\in H$ and $u_h\in H_h$ such that
\begin{equation}\label{eq:Strang-problems}
a(u,v)\,=\, F(v) \quad \forall\,v \in H
\qan
a_h(u_h,v_h) \,=\, F(v_h)\quad \forall\,v_h \in H_h\,.
\end{equation}
Then, for each $h>0$, there holds
\begin{equation*}
\|u - u_h\|_H \,\leq\, C_{S,1}\,\dist(u,H_h) + C_{S,2}\,\sup_{0\neq v_h\in H_h} \frac{\big| a(u,v_h) - a_h(u,v_h) \big|}{\|v_h\|_H} \,,
\end{equation*}
where $C_{S,1}$ and $C_{S,2}$ are the positive constants given by 
\begin{equation*}
C_{S,1} \,:=\, \left(1 + \frac{2\,\|a\|}{\wt{\alpha}} + \frac{\|a_h\|}{\wt{\alpha}} \right) \qan 
C_{S,2} \,:=\, \frac{1}{\wt{\alpha}} \,.
\end{equation*}
\end{lem}
\begin{proof}
It is basically a suitable modification of the proof of \cite[Lemma 5.1]{cgos2020}, which in turn,
is a modification of \cite[Theorem 11.1]{Roberts-Thomas}.
We omit further details and just stress that the inf-sup conditions of the respective linear operator $a_h$ from \cite[Lemma 5.1]{cgos2020} is now replaced by the corresponding uniform ellipticity of the present bilinear form $a_h$.
\end{proof}

%

We now establish the main result of this section.
\begin{thm}\label{thm:error-sighuh}
Assume that the data $\f\in \bL^2(\Omega)$ and $\bu_\rD\in \bH^{1/2}(\Gamma)$ satisfy
\begin{equation}\label{eq:data-assumption-Cea}
\big(1 + 2^{\rp-3}c_\rp\big)\,\frac{C_{\bF}}{\alpha_\bA\,r_0}\,
\Big\{ \|\f\|_{0,\Omega} + \|\bu_\rD\|_{1/2,\Gamma} + \|\bu_\rD\|_{0,\Gamma} \Big\} 
\,\leq\, \frac{1}{2} \,.
\end{equation}
Then, there exists a positive constant $C$, independent of $h$, such that
\begin{equation}\label{eq:Cea-estimate}
\|(\bsi,\bu)-(\bsi_h,\bu_h)\|
\,\leq\, C\,\Big\{ \dist(\bsi,\bbH_h^\bsi) + \dist(\bu,\bH_h^\bu) \Big\} \,.
\end{equation}
\end{thm}
\begin{proof}
First, note that the continuous and discrete problems \eqref{eq:CBF-augmented-weak-formulation} 
and \eqref{eq:CBF-augmented-discrete-formulation-H0} have the structure of the ones in 
\eqref{eq:Strang-problems}.
In addition, using the fact that $\bu\in \bW_r$ and $\bu_h\in \wt{\bW}_r$, we observe from \eqref{eq:ellipticity-Aw} and \eqref{eq:continuity-Aw} that the bilinear forms $\bA_{\bu}$ and $\bA_{\bu_h}$ are elliptic and bounded with the same constants $\alpha_\bA/2$ and $C_\bA+ \alpha_{\bA}/2$, respectively.
In turn, $\bF$ and $\bF|_{(\bbH_h^\bsi\times \bH_h^\bu)'}$ are bounded and linear functional in $\bbH_0(\bdiv;\Omega)\times \bH^1(\Omega)$ and $\bbH_h^\bsi\times \bH_h^\bu$, respectively. 
Thus, as a direct application of Lemma \ref{lem:Strang-lemma}, we obtain
\begin{equation}\label{eq:Sob-abs-apply}
\begin{array}{l}
\ds \|(\bsi,\bu) - (\bsi_h,\bu_h)\|
\,\leq\, C_{S,1}\,\Big\{ \dist(\bsi,\bbH_h^\bsi) + \dist(\bu,\bH_h^\bu) \Big\} \\[2ex]
\ds\quad +\, C_{S,2}\,\sup_{\0\neq (\btau_h,\bv_h)\in \bbH_h^\bsi\times \bH_h^\bu} \frac{\big|\bB_\bu((\bsi,\bu),(\btau_h\bv_h)) - \bB_{\bu_h}((\bsi,\bu),(\btau_h,\bv_h))\big|} {\|(\btau_h,\bv_h)\|} \,,
\end{array}
\end{equation}
where
\begin{equation}\label{eq:constants-CS-1-2}
C_{S,1} \,:=\, 4 + 6\,\frac{C_\bA}{\alpha_\bA} \qan
C_{S,2} \,:=\, \frac{2}{\alpha_\bA} \,.
\end{equation}
Next, proceeding as in \eqref{eq:continuity-T-1}, using \eqref{eq:glowinsky-2} and the continuity of $\bB_{\bw}$ (cf. \eqref{eq:continuity-Bw}), it follows that
\begin{equation}\label{eq:Cea-estimate-Bu-Buh}
\begin{array}{l}
\ds \big|\bB_\bu((\bsi,\bu),(\btau_h\bv_h)) - \bB_{\bu_h}((\bsi,\bu),(\btau_h,\bv_h))\big| \\[1ex] 
\ds\quad
\leq\, \|\bu\|_{1,\Omega}\bigg(\frac{1}{\nu}\,(1 + \kappa_1) \|\bi_4\|^2  
+ \frac{\tF}{\alpha}\,c_\rp\,\|\bi_{\rq}\|^{\rp-1}\big( \|\bu\|_{1,\Omega} + \|\bu_h\|_{1,\Omega} \big)^{\rp-3} \bigg)
\|\bu - \bu_h\|_{1,\Omega}
\|(\btau_h,\bv_h)\| \,.
\end{array}
\end{equation}
Thus, replacing \eqref{eq:Cea-estimate-Bu-Buh} back into \eqref{eq:Sob-abs-apply}, 
using the explicit expression of $C_{S,2}$ (cf. \eqref{eq:constants-CS-1-2}),
and the fact that $\|\bu\|_{1,\Omega} + \|\bu_h\|_{1,\Omega} \leq 2\,r_2$, since $\bu\in \bW_r$ and $\bu_h\in \wt{\bW}_r$, with $r\in (0,r_0]$ and $r_0:=\min\{r_1,r_2\}$ (cf. \eqref{eq:radius-r0-r1-r2}), we find that
\begin{equation*}
\begin{array}{l}
\ds \|(\bsi,\bu) - (\bsi_h,\bu_h)\|
\,\leq\, C_{S,1}\,\Big\{ \dist(\bsi,\bbH_h^\bsi) + \dist(\bu,\bH_h^\bu) \Big\} \\[2ex]
\ds\quad +\, \frac{1}{2}\,\|\bu\|_{1,\Omega}\bigg( \frac{1}{r_1}  
+ \frac{2^{\rp-3}\,c_\rp}{r_2} \bigg)
\|(\bsi,\bu) - (\bsi_h,\bu_h)\| \,.
\end{array}
\end{equation*}
Finally, using the fact that $1/r_1, 1/r_2$ are bounded by $1/r_0$, and bounding now $\|\bu\|_{1,\Omega}$ as in \eqref{eq:Tw-bound} instead of directly by $r$, we get
\begin{equation}\label{eq:error-bound-data-term}
\begin{array}{l}
\ds \|(\bsi,\bu) - (\bsi_h,\bu_h)\|
\,\leq\, C_{S,1}\,\Big\{ \dist(\bsi,\bbH_h^\bsi) + \dist(\bu,\bH_h^\bu) \Big\} \\[2ex]
\ds\quad +\, \big(1 + 2^{\rp-3}\,c_\rp \big)\,\frac{C_{\bF}}{\alpha_\bA\,r_0}\,
\Big\{ \|\f\|_{0,\Omega} + \|\bu_\rD\|_{1/2,\Gamma} + \|\bu_\rD\|_{0,\Gamma} \Big\}\,
\|(\bsi,\bu) - (\bsi_h,\bu_h)\|\,,
\end{array}
\end{equation}
which, together with the data assumption \eqref{eq:data-assumption-Cea}, implies \eqref{eq:Cea-estimate} and conclude the proof.
\end{proof}

Now, in order to provide the theoretical rate of convergence of the Galerkin scheme \eqref{eq:CBF-augmented-discrete-formulation-H0}, we recall the approximation properties of the subspaces involved (see, e.g., \cite{Brezzi-Fortin,Ciarlet,Gatica}).
Note that each one of them is named after the unknown to which it is applied later on.

\medskip
\noindent $(\mathbf{AP}_h^\bsi)$ For each $l\in (0,k+1]$ and for each $\btau\in \bbH^l\cap\bbH_{0}(\bdiv;\Omega)$ with $\bdiv(\btau)\in \bH^l(\Omega)$, there holds
\begin{equation*}
\dist(\btau,\bbH_h^\bsi) \,:=\, \inf_{\btau_h\in\bbH_h^\bsi} \|\btau - \btau_h\|_{\bdiv;\Omega}
\,\leq\, C\,h^l\,\Big\{ \|\btau\|_{l,\Omega} + \|\bdiv(\btau)\|_{l,\Omega} \Big\}\,.
\end{equation*}

\noindent $(\mathbf{AP}_h^\bu)$ For each $l\in [0,k+1]$ and for each $\bv\in \bH^{l+1}(\Omega)$, there holds
\begin{equation*}
\dist(\bv,\bH_h^\bu) \,:=\, \inf_{\bv_h\in\bH_h^\bu} \|\bv - \bv_h\|_{1,\Omega}
\,\leq\, C\,h^l\,\|\bv\|_{l+1,\Omega} \,.
\end{equation*}

The following theorem provides the theoretical optimal rate of convergence of the Galerkin scheme \eqref{eq:CBF-augmented-discrete-formulation-H0}, under suitable regularity assumptions on the exact solution.
\begin{thm}\label{thm:approximation}
In addition to the hypotheses of Theorems \ref{thm:fixed-point-T}, \ref{th:contractive}, and \ref{thm:error-sighuh}, assume that there exists $l\in (0,k+1]$ such that $\bsi \in \bbH^l(\Omega)\cap \bbH_0(\bdiv;\Omega)$, $\bdiv(\bsi)\in \bH^l(\Omega)$ and $\bu\in\bH^{l+1}(\Omega)$. 
Then, there exists $C>0$, independent of $h$, such that
\begin{equation}
\|(\bsi,\bu) - (\bsi_h,\bu_h)\|
\,\leq\, C\,h^l\,\Big\{ \|\bsi\|_{l,\Omega} + \|\bdiv(\bsi)\|_{l,\Omega} + \|\bu\|_{l+1,\Omega} \Big\} \,.
\end{equation}
\end{thm}
\begin{proof}
The result follows from a direct application of Theorem \ref{thm:error-sighuh} and the approximation properties of the discrete subspaces. Further details are omitted.
\end{proof}


\section{A residual-based \textit{a posteriori} error estimator}\label{eq:first-a-posteriori-error-estimator}

In this section we derive a reliable and efficient residual based {\it a posteriori}
error estimator for the Galerkin scheme \eqref{eq:CBF-augmented-discrete-formulation-H0}.
To this end, in what follows we employ the notations and results from Appendix \ref{sec:appendix-B}, 
and assume the hypotheses from Theorems \ref{thm:fixed-point-T} and \ref{th:contractive}, 
which guarantee the existence of unique solutions 
$(\bsi,\bu)\in \bbH_0(\bdiv;\Omega)\times\bH^1(\Omega)$ and 
$(\bsi_h,\bu_h)\in \bbH_{h}^\bsi\times \bH_h^\bu$ of the continuous and discrete problems 
\eqref{eq:CBF-augmented-weak-formulation} and \eqref{eq:CBF-augmented-discrete-formulation-H0}, 
respectively. 
Then a first global {\it a posteriori} error estimator is defined by
\begin{equation}\label{eq:global-estimator-1}
\Theta_1 \,:=\, \left\{ \sum_{T\in \cT_h} \Theta_{1,T}^2 \right\}^{1/2} \,,
\end{equation}
where, for each $T\in \cT_h$, the local error indicator $\Theta^2_{1,T}$ is defined as follows:
\begin{equation}\label{eq:local-estimator-1}
\begin{array}{l}
\ds \Theta_{1,T}^2 \,:=\, \Big\|\nabla \bu_h - \frac{1}{\nu}\big(\bsi_h + (\bu_h\otimes\bu_h)\big)^\rd\Big\|_{0,T}^2
+ \|\alpha\,\bu_h + \tF\,|\bu_h|^{\rp-2}\bu_h - \bdiv(\bsi_h) - \f\|_{0,T}^2 \\[2ex]
\ds\quad +\, h_T^2\,\left\| \ubcurl\left(\frac{1}{\nu}\big(\bsi_h + (\bu_h\otimes\bu_h)\big)^\rd\right) \right\|_{0,T}^2 
+ \sum_{e\in \cE_{h,T}(\Omega)} h_e\, \left\|\jump{\ubgamma_*\left(\frac{1}{\nu}\big(\bsi_h + (\bu_h\otimes\bu_h)\big)^\rd\right)} \right\|_{0,e}^2 \\[2ex]
\ds\quad +\, \sum_{e\in\cE_{h,T}(\Gamma)} h_e\,\left\|\ubgamma_*\left(\nabla \bu_\rD - \frac{1}{\nu}\big( \bsi_h + (\bu_h\otimes\bu_h)\big)^\rd \right)\right\|_{0,e}^2
+ \sum_{e\in\cE_{h,T}(\Gamma)} \|\bu_\rD-\bu_h\|_{0,e}^2 \,.
\end{array}
\end{equation}
Notice that the fifth term in \eqref{eq:local-estimator-1} requires $\ubgamma_{*}(\nabla\bu_\rD)|_e\in \bL^2(e)$ for all $e\in \cE_h(\Gamma)$, which is overcome below (cf. Lemma \ref{lem:reliability-R1}) by simply assuming that $\bu_\rD\in \bH^1(\Gamma)$.
We observe in advance that alternatively to \eqref{eq:global-estimator-1} a second {\it a posteriori} error estimator for the Galerkin scheme \eqref{eq:CBF-augmented-discrete-formulation-H0} is derived and analyzed in Appendix \ref{sec:appendix-C}.

The main goal of the present section is to establish, under suitable assumptions, 
the reliability and efficiency of $\Theta_1$. We begin with the reliability of the estimator.

\subsection{Reliability}\label{sec:reliability-Theta-1}

The main result of this section is stated in the following theorem.
\begin{thm}\label{thm:reliability-Theta-1}
Assume that the data $\f$ and $\bu_\rD$ satisfy \eqref{eq:data-assumption-Cea}.
Then, there exists a positive constant $C_{\tt rel}$, independent of $h$, such that
\begin{equation}\label{eq:reliability-Theta-1}
\|(\bsi,\bu) - (\bsi_h,\bu_h)\| \,\leq\, C_{\tt rel}\,\Theta_1 \,.
\end{equation}
\end{thm}

We begin the derivation of \eqref{eq:reliability-Theta-1} with a preliminary lemma, 
for which we first note that, using the fact that $\bu\in \bW_r$ and \eqref{eq:ellipticity-Aw}, 
we have that the bilinear form $\bA_\bu$ is uniformly elliptic on 
$\bbH_{\tt CBF} \,:=\, \bbH_0(\bdiv;\Omega)\times \bH^1(\Omega)$
with positive constant $\alpha_{\bA}/2$ independent of $h$. This implies that
\begin{equation}\label{eq:global-inf-sup}
\sup_{\0\neq(\btau,\bv)\in\bbH_{\tt CBF}} \frac{\bA_\bu((\bzeta,\bz),(\btau,\bv))}{\|(\btau,\bv)\|} 
\,\geq\, \frac{\alpha_{\bA}}{2}\,\|(\bzeta,\bz)\| \quad \forall\,(\bzeta,\bz)\in \bbH_{\tt CBF} \,.
\end{equation}

\begin{lem}\label{lem:preliminary-reliability}
Assume that the data $\f$ and $\bu_\rD$ satisfy \eqref{eq:data-assumption-Cea}. 
Then, there exists a positive constant $C$, independent of $h$, such that
\begin{equation}\label{eq:error-bounded-by-R}
\|(\bsi,\bu) - (\bsi_h,\bu_h)\| \,\leq\, C\,\sup_{\0\neq (\btau,\bv)\in \bbH_{\tt CBF}} \frac{\big| \cR(\btau,\bv) \big|}{\|(\btau,\bv)\|} \,,
\end{equation}
where $\cR :\bbH_{\tt CBF}\to \R$ is the residual functional given by
\begin{equation*}
\cR(\btau,\bv) \,:=\, \bF(\btau,\bv) - \bA_{\bu_h}((\bsi_h,\bu_h),(\btau,\bv)) \quad 
\forall\,(\btau,\bv) \in \bbH_{\tt CBF} \,.
\end{equation*}
\end{lem}
\begin{proof}
First, applying the inf-sup condition \eqref{eq:global-inf-sup} to the error
$(\bzeta,\bz) = (\bsi - \bsi_h,\bu - \bu_h)$, adding and substracting 
$\bB_{\bu_h}((\bsi_h,\bu_h),(\btau,\bv))$, and using the first equation of 	
\eqref{eq:CBF-augmented-weak-formulation}, we deduce that
\begin{equation*}
\begin{array}{l}
\ds \frac{\alpha_{\bA}}{2}\,\|(\bsi - \bsi_h,\bu - \bu_h)\| 
\,\leq\, \sup_{\0\neq(\btau,\bv)\in\bbH_{\tt CBF}} \frac{\big|\cR(\btau,\bv)\big|}{\|(\btau,\bv)\|} 
+ \sup_{\0\neq (\btau,\bv)\in\bbH_{\tt CBF}} \frac{\big|\big(\bB_{\bu} - \bB_{\bu_h}\big)((\bsi_h,\bu_h),(\btau,\bv))\big|}{\|(\btau,\bv)\|} \,,
\end{array}
\end{equation*}
which, combined with the continuity bound of $\bB_{\bw}$ (cf. \eqref{eq:Cea-estimate-Bu-Buh}), \eqref{eq:glowinsky-2} and proceeding as in \eqref{eq:error-bound-data-term}, implies
\begin{equation*}
\begin{array}{l}
\ds \|(\bsi-\bsi_h,\bu-\bu_h)\| \,\leq\, \frac{2}{\alpha_{\bA}}\,\sup_{\0\neq (\btau,\bv)\in \bbH_{\tt CBF}} \frac{\big| \cR(\btau,\bv) \big|}{\|(\btau,\bv)\|} \\[3ex]
\ds\quad +\, \big( 1 + 2^{\rp-3}\,c_\rp\big)\,\frac{C_{\bF}}{\alpha_{\bA} r_0}\Big\{ \|\f\|_{0,\Omega} + \|\bu_\rD\|_{1/2,\Gamma} + \|\bu_\rD\|_{0,\Gamma} \Big\}\,\|\bu-\bu_h\|_{1;\Omega} \,,
\end{array}
\end{equation*}
which, together with the data assumption \eqref{eq:data-assumption-Cea}, 
yields \eqref{eq:error-bounded-by-R} with $C=4/\alpha_\bA$ concluding the proof.
\end{proof}

We now aim to bound the suprema in \eqref{eq:error-bounded-by-R}.
Indeed, in virtue of the definitions of the forms $\bA_\bw, \bA$ and $\bB_\bw$ (cf. \eqref{eq:bilinear-form-Aw}--\eqref{eq:bilinear-form-Bw}), we find that, for any $(\btau,\bv)\in \bbH_0(\bdiv;\Omega)\times \bH^1(\Omega)$, there holds
\begin{equation*}
\cR(\btau,\bv) \,=\, \cR_1(\btau) \,+\, \cR_2(\bv) \,,
\end{equation*}
where
\begin{equation}\label{eq:residual-functional-tau}
\begin{array}{l}
\ds \cR_1(\btau) \,=\, \langle \btau\bn,\bu_\rD\rangle_\Gamma 
- \frac{1}{\nu} \int_\Omega \big( \bsi_h + (\bu_h \otimes \bu_h)\big)^\rd:\btau 
- \int_\Omega \bu_h\cdot\bdiv(\btau) \\[2ex]
\ds\quad +\, \frac{1}{\alpha} \int_\Omega \big( \alpha\,\bu_h + \tF\,|\bu_h|^{\rp-2}\bu_h 
- \bdiv(\bsi_h) - \f \big)\cdot\bdiv(\btau) 
\end{array}
\end{equation}
\begin{equation}\label{eq:residual-functional-v}
\qan \cR_2(\bv) \,=\,  
- \kappa_1\int_\Omega \Big( \nabla \bu_h - \frac{1}{\nu}\big(\bsi_h + (\bu_h\otimes\bu_h)\big)^\rd \Big):\nabla \bv
+ \kappa_2 \int_\Gamma (\bu_\rD-\bu_h)\cdot\bv \,.
\end{equation}
Notice that for convenience of the subsequent analysis we have added and subtracted the term given by $\int_\Omega \bu_h\cdot\bdiv(\btau)$ in \eqref{eq:residual-functional-tau}.
Then, the supremum in \eqref{eq:error-bounded-by-R} can be bounded in terms of $\cR_1$ and $\cR_2$ as follows
\begin{equation}\label{eq:error-bound-by-R1-R2}
\|(\bsi,\bu) - (\bsi_h,\bu_h)\| 
\,\leq\, C\,\Big\{ \|\cR_1\|_{\bbH_0(\bdiv;\Omega)'}
+ \|\cR_2\|_{\bH^1(\Omega)'} \Big\} \,,
\end{equation}
and hence our next purpose is to derive suitable upper bounds for each one of the terms on the right-hand side of \eqref{eq:error-bound-by-R1-R2}.
We begin by establishing the corresponding estimate for $\cR_2$ (cf. \eqref{eq:residual-functional-v}),
which follows from a straightforward application of the Cauchy--Schwarz inequality.
\begin{lem}\label{lem:reliability-R2} 
There exists a positive constant $C$, independent of $h$, such that	
\begin{equation*}
\|\cR_2\|_{\bH^1(\Omega)'} \,\leq\, C\,\left\{
\sum_{T\in \cT_h} \left( \Big\|\nabla \bu_h - \frac{1}{\nu}\big(\bsi_h + (\bu_h \otimes \bu_h)\big)^\rd \Big\|^2_{0,T} 
+ \sum_{e\in \cE_{h,T}(\Gamma)} \|\bu_\rD - \bu_h\|^2_{0,e} \right)
\right\}^{1/2} \,.
\end{equation*}
\end{lem}

We now bound the term $\|\cR_1\|_{\bbH_0(\bdiv;\Omega)'}$. 
To this end, we first observe that integrating by parts the expression $\int_\Omega \bu_h\cdot\bdiv(\btau)$ in \eqref{eq:residual-functional-tau}, the functional $\cR_1$ can be rewritten as follows
\begin{equation}\label{eq:first-ressidual-alternative}
\begin{array}{l}
\ds \cR_1(\btau) \,=\, \langle \btau\bn,\bu_\rD - \bu_h\rangle_\Gamma 
+ \int_\Omega \Big( \nabla\bu_h - \frac{1}{\nu}\big(\bsi_h + (\bu_h\otimes\bu_h)\big)^\rd \Big):\btau \\[2ex]
\ds\quad +\, \frac{1}{\alpha} \int_\Omega \big( \alpha\,\bu_h + \tF\,|\bu_h|^{\rp-2}\bu_h - \bdiv(\bsi_h) - \f \big)\cdot \bdiv(\btau) \,.
\end{array}
\end{equation}
For simplicity, we prove the aforementioned result for the 3D case.
The two dimensional one proceeds analogously.
Given $\btau \in \bbH_0(\bdiv;\Omega)$, it follows from part $b)$ (respectively $a)$ for the 2D setting) 
of Lemma \ref{lem:helmholtz-decomposition} that there exist $\bz\in \bH^2(\Omega)$ and 
$\bchi\in \bbH^1(\Omega)$ such that $\btau = \nabla\bz + \ubcurl(\bchi)$ in $\Omega$, and
\begin{equation}\label{eq:Helmholtz-decomposition}
\|\bz\|_{2,\Omega} + \|\bchi\|_{1,\Omega} \,\leq\, C_{\tt Hel}\,\|\btau\|_{\bdiv;\Omega} \,.
\end{equation}
Then, we set $\btau_h := \bPi_h^k(\nabla\bz) + \ubcurl(\bI_h(\bchi)) + \jmath_0\,\bbI$, 
where $\jmath_0\in \R$ is chosen so that $\int_\Omega \tr(\btau_h) = 0$. 
In addition, bearing in mind the definition of $\cR_1$ (cf. \eqref{eq:residual-functional-tau}),
and using the Galerkin scheme \eqref{eq:CBF-augmented-weak-formulation} and the compatibility 
condition \eqref{eq:compatibility-condition}, we deduce that 
$\cR_1(\btau_h) = 0$ and $\cR_1(\bbI)=0$, whence
\begin{equation}\label{eq:R1-Helmoltz-decomposition}
\cR_1(\btau) \,=\, \cR_1(\btau - \btau_h)
\,=\, \cR_1(\nabla\bz - \bPi_h^k(\nabla\bz)) + \cR_1(\ubcurl(\bchi-\bI_h(\bchi))) \,.
\end{equation}

The following lemma establishes the estimate for $\cR_1$.
\begin{lem}\label{lem:reliability-R1}
Assume that $\bu_\rD\in \bH^1(\Gamma)$. 
Then, there exists a positive constant $C$, independent of $h$, such that
\begin{equation*}
\|\cR_1\|_{\bbH_0(\bdiv;\Omega)'} 
\,\leq\, C\,\left\{ \sum_{T\in \cT_h} \wt{\Theta}_{1,T}^2 \right\}^{1/2} \,,
\end{equation*}
where
\begin{equation}\label{eq:local-estimator-1-tilde}
\begin{array}{l}
\ds \wt{\Theta}_{1,T}^2 \,:=\, \|\alpha\,\bu_h + \tF\,|\bu_h|^{\rp-2}\bu_h - \bdiv(\bsi_h) - \f\|_{0,T}^2 
+ h_T^2\,\Big\|\nabla \bu_h - \frac{1}{\nu}\big(\bsi_h + (\bu_h\otimes\bu_h)\big)^\rd \Big\|_{0,T}^2 \\[2ex]
\ds\quad +\, h_T^2\left\|\ubcurl\left(\frac{1}{\nu}\big(\bsi_h + (\bu_h \otimes \bu_h)\big)^\rd \right)\right\|_{0,T}^2
+ \sum_{e\in \cE_{h,T}(\Omega)} h_e\, \left\|\jump{\ubgamma\left(\frac{1}{\nu}\big(\bsi_h + (\bu_h \otimes \bu_h)\big)^\rd \right)}\right\|_{0,e}^2 \\[2ex] 
\ds\quad +\, \sum_{e\in\cE_{h,T}(\Gamma)} h_e\,\left\|\ubgamma_{*}\left(\nabla \bu_\rD - \frac{1}{\nu}\big(\bsi_h + (\bu_h\otimes\bu_h)\big)^\rd \right)\right\|_{0,e}^2
+ \sum_{e\in\cE_{h,T}(\Gamma)} h_e\,\|\bu_\rD-\bu_h\|_{0,e}^2 \,.
\end{array}
\end{equation}
\end{lem}
\begin{proof}
We begin by considering the expression for $\cR_1$ given by \eqref{eq:first-ressidual-alternative}. 
Thus, proceeding analogously as in the proof of \cite[Lemma 4.4]{gms2010}, i.e., 
using the Cauchy--Schwarz inequality, and applying the approximation properties of 
$\bPi_h^k$ (cf. Lemma \ref{lem:Phih-properties}), we find that
\begin{equation}\label{eq:R1-reliability-1}
\begin{array}{l}
\ds \big|\cR_1(\nabla\bz - \bPi_h^k(\nabla\bz)) \big| \,\leq\, C_1\,\Bigg\{
\sum_{T\in\cT_h} \| \alpha\,\bu_h + \tF\,|\bu_h|^{\rp-2}\bu_h - \bdiv(\bsi_h) - \f \|_{0,T}^2 \\[3ex]
\ds\quad +\, \sum_{T\in\cT_h} h_T^2\,\Big\|\nabla \bu_h - \frac{1}{\nu}\big(\bsi_h + (\bu_h\otimes\bu_h)\big)^\rd \Big\|_{0,T}^2 
+ \sum_{e\in\cE_h(\Gamma)} h_e\,\|\bu_\rD - \bu_h\|_{0,e}^2
\Bigg\}^{1/2}\,\|\bz\|_{2,\Omega} \,,
\end{array}
\end{equation}
with $C_1>0$, independent of $h$. 
In turn, in order to bound $\cR_1(\ubcurl(\bchi - \bI_h(\bchi)))$, we appeal to the original definition
of $\cR_1$ in \eqref{eq:residual-functional-tau}, and proceed as in \cite[Lemma 4.3]{gms2010} 
by using the integration by parts formula on the boundary $\Gamma$ obtained from \cite[Chapter I, eq. (2.17) and Theorem 2.11]{Girault-Raviart} (see also \cite[Lemma 3.5, eq. (3.34) for 2D case]{dgm2015}):
\begin{equation*}
\langle \ubcurl(\bchi - \bI_h(\bchi))\bn , \bu_\rD \rangle_\Gamma 
\,=\, -\,\langle \nabla\bu_\rD\times\bn, \bchi - \bI_h(\bchi) \rangle_\Gamma 
\,=\, -\,\langle \ubgamma_{*}(\nabla\bu_\rD),\bchi - \bI_h(\bchi) \rangle_\Gamma \,,
\end{equation*}
so that applying local integration by parts, the Cauchy--Schwarz inequality, and 
the approximation properties of $\bI_h$ (cf. Lemma \ref{lem:clement}), we obtain	
\begin{equation}\label{eq:R1-reliability-2}
\begin{array}{l}
\ds \big| \cR_1(\ubcurl(\bchi - \bI_h(\bchi))) \big| \,\leq\, C_2\,\Bigg\{
\sum_{T\in\cT_h} h_T^2\,\left\|\ubcurl\left(\frac{1}{\nu}\big(\bsi_h + (\bu_h \otimes \bu_h)\big)^\rd \right)\right\|_{0,T}^2 \\[3ex]
\ds\quad +\, \sum_{e\in \cE_h(\Omega)} h_e\,\left\|\jump{\ubgamma_{*}\left(\frac{1}{\nu}\big(\bsi_h + (\bu_h\otimes\bu_h)\big)^\rd \right)}\right\|_{0,e}^2 \\[3ex]
\ds\quad +\, \sum_{e\in\cE_h(\Gamma)} h_e\,\left\|\ubgamma_{*}\left(\nabla \bu_\rD - \frac{1}{\nu}\big(\bsi_h + (\bu_h\otimes\bu_h)\big)^\rd \right)\right\|_{0,e}^2 \Bigg\}^{1/2}\,\|\bchi\|_{1,\Omega}\,,
\end{array}
\end{equation}
where the term involving the set $\cE_h(\Gamma)$ remain valid if $\bu_\rD\in \bH^1(\Gamma)$.
The conclusion follows directly from \eqref{eq:R1-Helmoltz-decomposition}, 
\eqref{eq:R1-reliability-1}, \eqref{eq:R1-reliability-2}, and the stability of 
the Helmholtz decomposition (cf. \eqref{eq:Helmholtz-decomposition}).
\end{proof}

We end this section by observing that the estimate \eqref{eq:reliability-Theta-1} is a
straightforward consequence of Lemmas \ref{lem:preliminary-reliability} 
and \ref{lem:reliability-R2}--\ref{lem:reliability-R1}, the definition of the global estimator 
$\Theta_1$ (cf. \eqref{eq:global-estimator-1}), and the fact that the terms 
$h_T^2\,\big\|\nabla \bu_h - \frac{1}{\nu}\big(\bsi_h + (\bu_h\otimes\bu_h)\big)^\rd \big\|_{0,T}^2$ 
and $h_e\,\|\bu_\rD - \bu_h\|_{0,e}^2$, which form part of $\wt{\Theta}_{1,T}$ 
(cf. \eqref{eq:local-estimator-1-tilde}), are dominated by 
$\big\|\nabla \bu_h - \frac{1}{\nu}\big(\bsi_h + (\bu_h\otimes\bu_h)\big)^\rd \big\|_{0,T}^2$ 
and $\|\bu_\rD - \bu_h\|_{0,e}^2$, respectively.

\subsection{Efficiency}\label{sec:efficiency-Theta-1}

We now aim to establish the efficiency estimate of $\Theta_1$ (cf. \eqref{eq:global-estimator-1}).
For this purpose, we will make extensive use of the original system of equations given by 
\eqref{eq:convective-Brinkman-Forchheimer-2}, which is recovered from the mixed continuous
formulation \eqref{eq:CBF-augmented-weak-formulation} by choosing suitable test functions 
and integrating by parts backwardly the corresponding equations.
The following theorem is the main result of this section.
\begin{thm}\label{th:efficience-estimator1}
Supose that the data $\f$ and $\bu_\rD$ satisfy \eqref{eq:data-assumption-Cea}.
Then, there exists a positive constant $C_{\tt eff}$, independent of $h$, such that 
\begin{equation}\label{eq:efficiency-estimator-1}
C_{\tt eff}\,\Theta_1 + {\tt h.o.t.} \,\leq\, \|(\bsi,\bu) - (\bsi_h,\bu_h)\| \,,
\end{equation}
where ${\tt h.o.t.}$ stands for one or several terms of higher order.
\end{thm}

Throughout this section we assume, without loss of generality, that $\f$ and $\bu_\rD$,
are all piecewise polynomials. Otherwise, if $\f$ and $\bu_\rD$ are sufficiently smooth,
one proceeds similarly to \cite[Section 6.2]{cmo2016}, so that higher order terms
given by the errors arising from suitable polynomial approximation of these functions
appear in \eqref{eq:efficiency-estimator-1}, which explains the eventual ${\tt h.o.t.}$ in 
this inequality.

We begin the derivation of the efficiency estimates with the following result.
\begin{lem}\label{lem:eff_bound_func-1}
There exist $C_1>0$ and $C_2>0$, independent of $h$, such that 
for each $T\in \cT_h$ there hold
\begin{equation}\label{eq:eff-bound-func1}
\begin{array}{l}
\ds \Big\|\nabla\bu_h - \frac{1}{\nu}\big(\bsi_h + (\bu_h\otimes\bu_h)\big)^\rd\Big\|_{0,T} \\[2ex]
\ds\quad \leq\, C_1\,\Big\{ \|\bu - \bu_h\|_{1,T} 
+ \|\bsi - \bsi_h\|_{0,T}
+ \|\bu\otimes\bu - \bu_h\otimes\bu_h\|_{0,T} \Big\} 
\end{array}
\end{equation}
and
\begin{equation}\label{eq:eff-bound-func2}
\begin{array}{l}
\ds \|\alpha\,\bu_h + \tF\,|\bu_h|^{\rp-2}\bu_h - \bdiv(\bsi_h) - \f\|_{0,T} \\[2ex]
\ds\quad \leq\, C_2\,\Big\{ \|\bsi - \bsi_h\|_{\bdiv;T} 
+ \|\bu-\bu_h\|_{0,T} 
+ \big\| |\bu|^{\rp-2}\bu - |\bu_h|^{\rp-2}\bu_h \big\|_{0,T} \Big\} \,.
\end{array}
\end{equation}
\end{lem}
\begin{proof}
It suffices to recall that $\nabla \bu = \dfrac{1}{\nu}\big(\bsi + (\bu\otimes\bu)\big)^\rd$ in $\Omega$
and $\f = \alpha\,\bu + \tF\,|\bu|^{\rp-2}\bu - \bdiv(\bsi)$ in $\Omega$ (cf. \eqref{eq:convective-Brinkman-Forchheimer-2a}--\eqref{eq:convective-Brinkman-Forchheimer-2b}).
We omit further details.
\end{proof}

Next, we provide the upper bound for the residual terms involving the Dirichlet datum $\bu_\rD$. 
\begin{lem}\label{lem:eff_bound_func-2}
There exists a positive constant $C_3$, independent of $h$, such that
\begin{equation*}
\sum_{e\in \mathcal E_h(	\Gamma)} \|\bu_{\rD} - \bu_{h}\|^2_{0,e} 
\,\leq\, C_3\,\|\bu - \bu_h\|^2_{1,\Omega} \,.
\end{equation*}
\end{lem}
\begin{proof}
It suffices to observe that 
\begin{equation*}
\sum_{e\in \mathcal E_h(\Gamma)} \|\bu_{\rD} - \bu_{h}\|^2_{0,e}  \,=\,
\sum_{e\in \mathcal E_h(\Gamma)} \|\bu - \bu_{h}\|^2_{0,e} \,=\,
\|\bu - \bu_h\|^2_{0,\Gamma} \,,
\end{equation*}
and then apply the trace inequality.
\end{proof}

The corresponding bounds for the remaining terms defining $\Theta_{1,T}$ 
(cf. \eqref{eq:local-estimator-1}) are stated in the following lemma.
\begin{lem}\label{lem:eff_bound_func-3}
There exist $C_4>0$ and $C_5>0$, independent of $h$, such that
\begin{equation}\label{eq:curl-bound-eff}
h_T\,\left\|\ubcurl\left(\frac{1}{\nu}\big(\bsi_h + (\bu_h\otimes\bu_h)\big)^\rd\right)\right\|_{0,T}
\,\leq\, C_4\,\Big\{ \|\bsi - \bsi_h\|_{0,T} + \|\bu\otimes\bu - \bu_h\otimes\bu_h\|_{0,T} \Big\}
\end{equation}
for all $T\in \cT_h$ and
\begin{equation}\label{eq:jump-bound-eff}
h_e^{1/2}\,\left\|\jump{\ubgamma_{*}\left(\frac{1}{\nu}\big(\bsi_h + (\bu_h\otimes\bu_h)\big)^\rd\right)}\right\|_{0,e}
\,\leq\, C_5\,\Big\{ \|\bsi - \bsi_h\|_{0,\omega_e} + \|\bu\otimes\bu - \bu_h\otimes\bu_h\|_{0,\omega_e} \Big\}
\end{equation}
for all $e\in \cE_h(\Omega)$, where $\omega_e$ denotes the union of the two elements of $\cT_h$ sharing the edge/face $e$.
Additionally, if $\bu_\rD$ is piecewise polynomial, there exists $C_6>0$, 
independent of $h$, such that
\begin{equation}\label{eq:gamma-bound-eff}
h_e^{1/2}\,\left\|\ubgamma_{*}\left(\nabla \bu_\rD - \frac{1}{\nu}\big(\bsi_h + (\bu_h\otimes\bu_h)\big)^\rd\right) \right\|_{0,e}
\,\leq\, C_6\,\Big\{ \|\bsi - \bsi_h\|_{0,T_e} + \|\bu\otimes\bu - \bu_h\otimes\bu_h\|_{0,T_e} \Big\}
\end{equation}
for all $e\in \cE_h(\Gamma)$, where $T_e$ is the element to which the boundary 
edge or boundary face $e$ belongs.
\end{lem}
\begin{proof} 
First, noting that $\ubcurl\left(\frac{1}{\nu}\big(\bsi+(\bu\otimes\bu)\big)^\rd\right) = \ubcurl\left(\nabla \bu\right) = \0$ in $\Omega$, we find that 
\eqref{eq:curl-bound-eff}--\eqref{eq:jump-bound-eff} follows from a slight adaptation of 
\cite[Lemma 4.11]{gms2010}, whereas for the proof of \eqref{eq:gamma-bound-eff} we refer
the reader to \cite[Lemma 4.15]{gms2010}.
\end{proof}

In order to complete the global efficiency given by \eqref{eq:efficiency-estimator-1} 
(cf. Theorem \ref{th:efficience-estimator1}), we now need to estimate the terms 
$\big\| |\bu|^{\rp-2}\bu - |\bu_h|^{\rp-2}\bu_h \big\|_{0,T}^2$ 
and $\|\bu\otimes\bu - \bu_h\otimes\bu_h\|^2_{0,T}$ appearing in the upper bounds provided by 
Lemmas \ref{lem:eff_bound_func-1} and \ref{lem:eff_bound_func-3}.
To this end, we first make use of \eqref{eq:glowinsky}, the H\"older inequality with $p=3/2$ and $q=3$
satisfying $1/p + 1/q=1$,
and simple algebraic manipulations, to obtain
\begin{equation*}\label{eq:Forch_bound-1}
\big\| |\bu|^{\rp-2}\bu - |\bu_h|^{\rp-2}\bu_h \big\|_{0,T}^2 
\,\leq\, \wh{c}_\rp\,\Big( \|\bu\|_{0,3(\rp-2);T}^{2(\rp-2)} + \|\bu_h\|_{0,3(\rp-2);T}^{2(\rp-2)}\Big) \|\bu - \bu_h\|_{0,6;T}^2 \,,
\end{equation*}
with $\wh{c}_\rp := 2^{2\,\rp - 5}\,c^2_\rp$.
Then, applying H\"older inequality and some algebraic computations, we find that
\begin{equation}\label{eq:Forch_bound-2}
\begin{array}{l}
\ds \sum_{T\in \cT_h} \big\| |\bu|^{\rp-2}\bu - |\bu_h|^{\rp-2}\bu_h \big\|_{0,T}^2 \\[2ex]
\ds\quad \,\leq\, \wh{c}_\rp\,\left\{ \sum_{T\in \cT_h} \Big( \|\bu\|_{0,3(\rp-2);T}^{2(\rp-2)} 
+ \|\bu_h\|_{0,3(\rp-2);T}^{2(\rp-2)}\Big)^{3/2} \right\}^{2/3} \left\{ \sum_{T\in \cT_h} \|\bu - \bu_h\|_{0,6;T}^6 \right\}^{1/3} \\[4ex]
\ds\quad 
\,\leq\, \sqrt[3]{2}\,\wh{c}_\rp\, \Big( \|\bu\|_{0,3(\rp-2);\Omega}^{2(\rp-2)} + \|\bu_h\|_{0,3(\rp-2);\Omega}^{2(\rp-2)}\Big) \|\bu - \bu_h\|_{0,6;\Omega}^2 \,.
\end{array}
\end{equation}
In this way, using the continuous injections $\bi_6: \bH^1(\Omega)\to \bL^6(\Omega)$
and $\bi_{3(\rp-2)}:\bH^1(\Omega)\to \bL^{3(\rp-2)}(\Omega)$, with $3(\rp-2)\in [3,6]$, and 
the fact that $\bu\in \bW_r$ and $\bu_h\in \wt{\bW}_r$,
we deduce from \eqref{eq:Forch_bound-2} that there exists a constant $C>0$, 
depending only on $r$ and other constants, and hence independent of $h$, such that
\begin{equation}\label{eq:Forch_bound-3}
\sum_{T\in\cT_h} \big\| |\bu|^{\rp-2}\bu - |\bu_h|^{\rp-2}\bu_h \big\|_{0,T}^2 
\,\leq\, C\,\|\bu-\bu_h\|_{1,\Omega}^2 \,.
\end{equation}

Similarly, adding and subtracting $\bu\otimes\bu_h$ (it also works with $\bu_h\otimes\bu$),
and applying H\"older's inequality, we deduce that
\begin{equation*}
\|\bu\otimes\bu - \bu_h\otimes\bu_h\|_{0,T}^2
\,\leq\, 2\,\big( \|\bu\|_{0,4;T}^2 + \|\bu_h\|_{0,4;T}^2 \big)\|\bu - \bu_h\|_{0,4;T}^2 \,,
\end{equation*}
so that proceeding analogously to \eqref{eq:Forch_bound-2},
and then using the continuous injection $\bi_4: \bH^1(\Omega)\to \bL^4(\Omega)$, and 
the fact that $\bu\in \bW_r$ and $\bu_h\in \wt{\bW}_r$, 
we are able to show that there exists a positive constant $\wt{C}$, independent of $h$, such that
\begin{equation}\label{eq:eff-convective-term}
\sum_{T\in\cT_h} \|\bu\otimes\bu - \bu_h\otimes\bu_h\|_{0,T}^2 
\,\leq\, \wt{C}\,\|\bu - \bu_h\|_{1,\Omega}^2 \,.
\end{equation}
Consequently, it is not difficult to see that \eqref{eq:efficiency-estimator-1} follows 
from the definition of $\Theta_1$ (cf. \eqref{eq:global-estimator-1}--\eqref{eq:local-estimator-1}),
Lemmas \ref{lem:eff_bound_func-1}, \ref{lem:eff_bound_func-2} and \ref{lem:eff_bound_func-3}, and 
the estimates \eqref{eq:Forch_bound-3} and \eqref{eq:eff-convective-term}.

%
%
\section{Numerical results}\label{sec:numerical-results}

This section serves to illustrate the performance and accuracy of the proposed augmented mixed 
finite element scheme \eqref{eq:CBF-augmented-discrete-formulation-H0} along with 
the reliability and efficiency properties of the {\it a posteriori} error estimators 
$\Theta_1$ (cf. \eqref{eq:global-estimator-1}) and $\wh{\Theta}_2$ 
(cf. \eqref{eq:global-estimator-2-hat}), in $2$D and $3$D domains.
Regarding $\wh{\Theta}_2$ it was established in Appendix \ref{sec:appendix-C} 
that it is reliable, but efficient only up to all its terms, except the last one in 
\eqref{eq:local-estimator-2-hat}.
Indeed, while the numerical results to be displayed below suggest that $\wh{\Theta}_2$
could actually verify both properties, the eventual efficiency is just a conjecture by now.
In what follows, we refer to the corresponding sets of finite element
subspaces generated by $k = 0$ and $k = 1$,  as simply 
$\bbRT_{0}-\bP_{1}$ and $\bbRT_{1}-\bP_{2}$, respectively. 
The implementation is based on a {\tt FreeFem++} code \cite{Hecht2012}. 
Regarding the implementation of the Newton iterative method associated to 
\eqref{eq:CBF-augmented-discrete-formulation-H0}, 
the iterations are terminated once the relative error of the entire coefficient vectors 
between two consecutive iterates, say $\coeff^{m}$ and $\coeff^{m+1}$, 
is sufficiently small, that is, 
\begin{equation*}
\frac{\|\coeff^{m+1} - \coeff^{m}\|}{\|\coeff^{m+1}\|} \,\leq\, \tol\,,
\end{equation*}
where $\|\cdot\|$ stands for the usual Euclidean norm in $\R^{\DOF}$, with $\DOF$ 
denoting the total number of degrees of freedom defining the finite element subspaces
$\bbH^{\bsi}_h$ and $\bH^{\bu}_h$ (cf. \eqref{eq:subspaces-a}--\eqref{eq:subspaces-b}), 
and $\tol$ is a fixed tolerance 
chosen as $\tol = 1\mathrm{E}-6$. The individual errors are denoted by:  
\begin{equation*}
\begin{array}{c}
\ds \re(\bsi) \,:=\, \|\bsi - \bsi_{h}\|_{\bdiv;\Omega} \,,\quad
\re(\bu) \,:=\, \|\bu - \bu_h\|_{1,\Omega} \,,\quad
\re(p) \,:=\, \|p - p_h\|_{0,\Omega} \,, \\[2ex]
\ds \ds \re(\bG) \,:=\, \|\bG - \bG_{h}\|_{0,\Omega} \,,\quad
\re(\bomega) \,:=\, \|\bomega - \bomega_h\|_{0,\Omega} \,,\qan
\re(\wt{\bsi}) \,:=\, \|\wt{\bsi} - \wt{\bsi}_h\|_{0,\Omega} \,,
\end{array}
\end{equation*} 
where the pressure $p$, the velocity gradient $\bG$, the vorticity $\bomega$, and the shear stress tensor $\wt{\bsi}$ are further variables of physical interest that are recovered by using the corresponding postprocessing formulae $p_h, \bG_h, \bomega_h$, and $\wt{\bsi}_h$ detailed in Appendix \ref{sec:appendix-A}.
In turn, the global error and effectivity indexes associated to the global estimators $\Theta_1$ and $\wh{\Theta}_2$ are denoted, respectively, by
\begin{equation*}
\re(\vec{\bsi}) := \re(\bsi) + \re(\bu)\,,\quad {\tt eff}(\Theta_1) := \frac{\re(\vec{\bsi})}{\Theta_1}\,,
\qan {\tt eff}(\wh{\Theta}_2) := \frac{\re(\vec{\bsi})}{\wh{\Theta}_2}\,.
\end{equation*}
Moreover, using the fact that $\DOF^{-1/d} \cong h$, the respective experimental rates of convergence
are computed as
\begin{equation*}
\sr(\star) \,:=\, -\,d\,\frac{\log\big(\re(\star)/\re'(\star)\big)}{\log(\DOF/\DOF')}\quad 
\mbox{for each }\, \star\in\big\{\bsi,\bu,p,\bG,\bomega,\wt{\bsi}, \vec{\bsi}\big\}\,,
\end{equation*}
where $\DOF$ and $\DOF'$ denote the total degrees of freedom associated to two consecutive 
triangulations with errors $\re(\star)$ and $\re'(\star)$, respectively.

\medskip
The examples to be considered in this section are described next. 
In all of them, for sake of simplicity, we take $\nu=1$, and choose the
parameters $\kappa_1, \kappa_2$ in agreement with Remark \ref{rem:kappa-parameters}, that is, 
$\kappa_1 = \nu$ and $\kappa_2 = \nu/2$.
In turn, in the first three examples we consider $\tF = 10$ and $\alpha=1$.
In addition, it is easy to see for these examples that the boundary datum 
$\bu_\rD := \bu|_{\Gamma}$ satisfies the required regularity
$\bu_\rD\in \bH^{1}(\Gamma)$ since the given exact solution $\bu$ is sufficiently regular.
Furthermore, the condition $\int_{\Omega} \tr(\bsi_{h}) = 0$ 
is imposed via a Lagrange multiplier strategy.
Example 1 is used to show the accuracy of the method and the behaviour of the effectivity 
indexes of the {\it a posteriori} error estimators $\Theta_1$ and $\wh{\Theta}_2$, 
whereas Examples 2--3 and 4 are utilized to illustrate the associated adaptive algorithm, 
with and without manufactured solutions, respectively, in both $2$D and $3$D domains.
The corresponding adaptivity procedure, taken from \cite{Verfurth}, is described as follows:

\begin{enumerate}
\item[(1)] Start with a coarse mesh $\cT_h$ of $\ov{\Omega}$.
	
\item[(2)] Solve the Newton iterative method associated with \eqref{eq:CBF-augmented-discrete-formulation-H0} on the current mesh.
	
\item[(3)] Compute the local indicator $\Theta_{1,T}$ (cf. \eqref{eq:local-estimator-1}) for each $T\in\cT_{h}$
	
\item[(4)] Check the stopping criterion and decide whether to finish or go to next step.
	
\item[(5)]
Use the automatic meshing algorithm {\sf adaptmesh} from \cite[Section 9.1.9]{Hecht2018} to refine each $T'\in \cT_h$ satisfying:
\begin{equation}\label{cri-5}
\Theta_{1,T'} \,\geq\, C_{\sf adm}\,\frac{1}{\#\,T} \sum_{T\in \cT_h} \Theta_{1,T},\quad \mbox{for some }\, C_{\sf adm}\in (0,1),
\end{equation}
where $\#\,T$ denotes the number of triangles of the mesh $\cT_h$.
	
\item[(6)] Define the resulting mesh as the current mesh $\cT_{h}$, and go to step (2).
\end{enumerate}

In particular, in Examples 2, 4 and 3 below we take $C_{\sf adm}$ (cf. \eqref{cri-5}) equal to $0.75$ and $0.8$, respectively.
Certainly, if the refinement is with respect to the local indicator $\wh{\Theta}_{2,T}$ (cf. \eqref{eq:local-estimator-2-hat}), 
we simply replace $\Theta_{1,T'}$ and $\Theta_{1,T}$ by $\wh{\Theta}_{2,T'}$ and
$\wh{\Theta}_{2,T}$, respectively, in the criterion \eqref{cri-5}.

\subsection*{Example 1: Accuracy assessment with a smooth solution in a square domain.}

We first concentrate on the accuracy of the mixed method
as well as the properties of the {\it a posteriori} error estimators through the 
effectivity indexes ${\tt eff}(\Theta_1)$ and ${\tt eff}(\wh{\Theta}_2)$, 
under a quasi-uniform refinement strategy.
We consider the square domain $\Omega=(0,1)^2$, inertial power $\rp = 3$, and adjust the data in \eqref{eq:convective-Brinkman-Forchheimer-2}
so that the exact solution is given by the smooth functions
\begin{equation*}
\bu(x_1,x_2) =
\begin{pmatrix}
\sin(\pi\,x_1)\cos(\pi\,x_2) \\[0.5ex]
-\cos(\pi\,x_1)\,\sin(\pi\,x_2)
\end{pmatrix},
\quad p(x_1,x_2) = \cos(\pi x_1)\sin\left( \frac{\pi}{2}x_2\right)\,.
\end{equation*}
Tables \ref{table:Ex1-k0} and \ref{table:Ex1-k1} show the convergence history for 
a sequence of quasi-uniform mesh refinements, including the average number of Newton iterations. 
Notice that we are able not only to approximate the original unknowns
but also the pressure field, the velocity gradient tensor, the vorticity, and the shear stress tensor
through the formula \eqref{eq:postprocess-vars-app}.
The results illustrate that the optimal rates of convergence $\cO(h^{k+1})$ 
established in Theorem \ref{thm:approximation} and Lemma \ref{lem:rate-of-convergence-further-variables} are attained for $k = 0, 1$.
In addition, the global {\it a posteriori} error indicators $\Theta_1, \wh{\Theta}_2$ 
and their respective effectivity indexes are also displayed there, from where 
we highlight that the latter remain always bounded.

\subsection*{Example 2: Adaptivity in a 2D L-shaped domain.}

We now aim at testing the features of adaptive mesh refinement after both 
{\it a posteriori} error estimators $\Theta_1$ and $\wh{\Theta}_2$ 
(cf. \eqref{eq:global-estimator-1}, \eqref{eq:global-estimator-2-hat}).  
We consider an L-shaped domain $\Omega := (-1,1)^2\setminus (0,1)^2$
and inertial power $\rp = 3.5$. 
The manufactured solution is given by
\begin{equation*}
\bu(x_1,x_2) =
\begin{pmatrix}
-\pi \sin(\pi x_1) \cos(\pi x_2)\\[0.5ex]
 \pi \cos(\pi x_1) \sin(\pi x_2)
\end{pmatrix},
\quad p(x_1,x_2)=\dfrac{10\,(1-x_1)}{(x_1 - 0.09)^2 + (x_2 - 0.09)^2} -p_{0}\,, 
\end{equation*}
where $p_0 \in \R$ is chosen so that $p\in \L^2_0(\Omega)$. Observe that 
the pressure field exhibit high gradients near the vertex $(0,0)$.
Figure \ref{figure:Ex2-loglog} summarizes the convergence history of the method when applied 
to quasi-uniform and adaptive (via $\Theta_1$ and $\wh{\Theta}_2$) refinements of the domain, 
which yield sub-optimal and optimal rates of convergence, respectively. 
For sake of simplicity and since the behavior of the method for $k=0$ and $k=1$ are similar, we only detail in Tables \ref{table:Ex2-k1-uniform}, \ref{table:Ex2-k1-adaptive-1}, and \ref{table:Ex2-k1-adaptive-2}, the case $k=1$, where the errors, rates of convergence, efficiency indexes, and Newton iterations are displayed for both refinements.
Notice how the adaptive algorithms improves the efficiency of the method 
by delivering quality solutions at a lower computational cost, to the point that 
it is possible to get a better one (in terms of $\re(\vec{\bsi})$) with approximately 
only the $4\%$ of the degrees of freedom of the last quasi-uniform mesh for 
the mixed scheme.
Furthermore, the inital mesh and some approximate solutions built using the 
$\bbRT_{1}-\bP_{2}$ scheme (via the indicator $\Theta_1$) 
with $25,187$ triangle elements (actually representing $354,680\,\DOF$), 
are shown in Figure \ref{figure:Ex2-k0-aprox}. 
In particular, we observe that the pressure exhibits high gradients 
near the contraction region of the L-shaped domain. 
In turn, examples of some adapted meshes for $k = 0$ guided by $\Theta_1$ 
and $\wh{\Theta}_2$ are collected in Figure \ref{figure:Ex2-Th-k0-k1}. 
We can observe a clear clustering of elements near 
the corner region of the contraction of the L-shaped domain as we expected.

\subsection*{Example 3: Adaptivity in a 3D L-shaped domain.}

Here we replicate the Example 2 in a three-dimensional setting by considering the 
3D L-shaped domain $\Omega := (-0.5,0.5) \times (0,0.5)\times (-0.5,0.5)\setminus (0,0.5)^3$,
inertial power $\rp = 3.7$, 
and the manufactured exact solution 
\begin{equation*}
\bu(x_1,x_2,x_3) =
\begin{pmatrix}
\sin(\pi x_1)\cos(\pi x_2)\cos(\pi x_3)  \\[0.5ex]
-2\cos(\pi x_1)\sin(\pi x_2)\cos(\pi x_3) \\[0.5ex]
\cos(\pi x_1)\cos(\pi x_2)\sin(\pi x_3)
\end{pmatrix},
\,\,\,  p(x_1,x_2,x_3) = \frac{10\,x_3}{(x_1 - 0.04)^2 + (x_3 - 0.04)^2} -p_{0}  \,.
\end{equation*}
Tables \ref{table:Ex3-k0-uniform}, \ref{table:Ex3-k0-adaptive-1}, and \ref{table:Ex3-k0-adaptive-2} 
confirm a disturbed convergence under quasi-uniform refinement,
whereas optimal convergence rates are obtained when adaptive refinements guided by 
the {\it a posteriori} error estimators $\Theta_1$ and $\wh{\Theta}_2$, with $k=0$, are used. 
In turn, the initial mesh and some approximated solutions after four mesh refinement steps 
(via $\Theta_1$) are collected in Figure \ref{figure:Ex3-k0-aprox}. In particular, we see there
that the pressure presents high values and hence, most likely, high gradients as well near 
the contraction region of the 3D L-shaped domain, as we expected. 
The latter is complemented with Figure \ref{figure:Ex3-Th-k0}, where snapshots of 
three meshes via $\Theta_1$ show a clustering of elements in the same region.  
The refinement obtained via $\wh{\Theta}_2$ is similar to the ones for $\Theta_1$, reason why the corresponding plots are omitted.

\subsection*{Example 4: Flow through a 2D porous media with fracture network.}

Inspired by \cite[Example 4, Section 6]{covy2022},
we finally focus on a flow through a porous medium with a fracture network considering strong jump
discontinuities of the parameters $\tF$ and $\alpha$ accross the two regions. 
We consider the square domain $\Omega=(-1,1)^2$  with an internal fracture network denoted 
as $\Omega_{\mathrm{f}}$ (see the first plot of Figure \ref{figure:Ex4-k0-aprox} below), 
and boundary $\Gamma$, whose left, right, upper and lower parts are given by
$\Gamma_{\text{left}} = \{-1\}\times (-1,1)$, $\Gamma_{\text{right}} = \{1\}\times (-1,1)$, 
$\Gamma_{\text{top}} = (-1,1)\times \{1\}$, and $\Gamma_{\text{bottom}} = (-1,1)\times \{-1\}$, 
respectively.
Note that the boundary of the internal fracture network is defined as a union of segments. 
The initial mesh file is available in 
\href{https://github.com/scaucao/Fracture_network-mesh}{\tt https://github.com/scaucao/Fracture\_network-mesh}.
We consider the convective Brinkman--Forchheimer equations 
\eqref{eq:convective-Brinkman-Forchheimer-2} in the whole domain $\Omega$,
with inertial power $\rp = 4$ 
but with different values of the parameters $\tF$ 
and $\alpha$ for the interior and the exterior of the fracture, namely
\begin{equation}\label{eq:parameters-on-fracture}
\tF = \left\{
\begin{array}{r r l}
10 & \text{in} & \Omega_{\mathrm{f}}\\[0.5ex]
1  & \text{in} & \overline{\Omega}\setminus\Omega_{\mathrm{f}}
\end{array}
\right.
\qan
\alpha = \left\{
\begin{array}{r r l}
1      & \text{in} & \Omega_{\mathrm{f}}\\[0.5ex]
1000  & \text{in} & \overline{\Omega}\setminus\Omega_{\mathrm{f}}
\end{array}
\right. \,.
\end{equation}
The parameter choice corresponds to increased inertial effect ($\tF=10$) in the fracture and
a high permeability ($\alpha=1$), compared to reduced inertial effect ($\tF=1$) in the 
porous medium and low permeability ($\alpha=1000$).
In turn, the body force term is $\f=\0$ and the boundaries conditions
are
\begin{equation}\label{eq:boundary-conditions-fracture}
\bsi\,\bn = \left\{
\begin{array}{cll}
(-0.5(y-1),\, 0)^\rt & \mbox{on} & \Gamma_{\mathrm{left}}\,, \\[1ex]
(0,\,-0.5(x-1))^\rt & \mbox{on} & \Gamma_{\mathrm{bottom}}\,,
\end{array} 
\right. \quad
\bsi\,\bn = (0,\,0)^\rt \qon \Gamma_{\mathrm{right}}\cup \Gamma_{\mathrm{top}}\,,
\end{equation}
which drives the flow in a diagonal direction from the left-bottom corner 
to the right-top corner of the square domain $\Omega$.
In Figure \ref{figure:Ex4-k0-aprox}, we display the initial mesh, 
the computed magnitude of the velocity, velocity gradient tensor, and pseudostress tensor, 
which were built using the $\bbRT_{1}-\bP_{2}$ scheme on a mesh with $142,867$ triangle 
elements (actually representing $2,008,424\,\DOF$) obtained via $\Theta_1$. 
We note that the velocity in the fractures is higher than the velocity in the porous medium, due to smaller fractures thickness and the parameter setting \eqref{eq:parameters-on-fracture}. 
Also, the velocity is higher in branches of the network where the fluid enters from the left-bottom corner and decreases toward the right-top corner of the domain. 
In addition, we observe a sharp velocity gradient across the interfaces between the fractures and the porous medium. The pseudostress is consistent with the boundary conditions \eqref{eq:boundary-conditions-fracture} and it is more diffused since it includes the pressure field. 
This example illustrates the ability of the method to provide accurate resolution and numerically stable results for heterogeneous inclusions with high aspect ratio and complex geometry, as presented in the network of thin fractures.
In turn, snapshots of some 
adapted meshes generated using $\Theta_1$ and $\wh{\Theta}_2$, are depicted in Figure \ref{figure:Ex4-Th-k0}.
Notice that the meshes obtained via the indicator $\Theta_{1}$ are slightly more refined in the interior of the domain than the meshes obtained via the indicator $\wh{\Theta}_2$. 
This fact is justified by the terms that capture the jumps between triangles, which arise from the Helmoltz decomposition applied to the convective Brinkman--Forchheimer equations and the consequent local integration by parts procedures. 
We conclude that the estimator $\Theta_{1}$ is slightly more sensible than $\wh{\Theta}_2$ to detect the strong jump discontinuities of the model parameters along the interface between the fracture and porous media and at the same time localizes the regions where the solutions are higher.
In this sense, and as suggested by the present example, the estimator $\Theta_{1}$ would be preferable if strong discontinuities of the model parameters and higher velocities in the fractures are considered.
\begin{table}[ht] 
\begin{center}
\small{		
\begin{tabular}{r | c | c | c | c | c | c | c | c | c | c }
\hline
$\DOF$ & $h$ & $\iter$ & $\re(\bsi)$ & $\sr(\bsi)$ & $\re(\bu)$ & $\sr(\bu)$ & $\re(p)$ & $\sr(p)$ & $\re(\bG)$ & $\sr(\bG)$ \\
\hline \hline
178    & 0.373 & 4 & 4.21E-00 &  --  & 8.91E-01 &  --  & 3.69E-01 &  --  & 7.34E-01 &  --  \\
674    & 0.196 & 4 & 1.82E-00 & 1.26 & 4.59E-01 & 1.00 & 1.57E-01 & 1.28 & 3.44E-01 & 1.14 \\
2546   & 0.097 & 4 & 9.08E-01 & 1.04 & 2.33E-01 & 1.02 & 7.40E-02 & 1.13 & 1.75E-01 & 1.02 \\
9922   & 0.048 & 4 & 4.41E-01 & 1.06 & 1.18E-01 & 1.00 & 3.50E-02 & 1.10 & 8.77E-02 & 1.02 \\
39210  & 0.025 & 4 & 2.23E-01 & 0.99 & 5.91E-02 & 1.00 & 1.81E-02 & 0.96 & 4.38E-02 & 1.01 \\
157610 & 0.013 & 4 & 1.10E-01 & 1.01 & 2.92E-02 & 1.01 & 8.75E-03 & 1.05 & 2.18E-02 & 1.01 \\
\hline
\end{tabular}
			
\smallskip
			
\begin{tabular}{c | c | c | c | c | c | c | c | c | c}
\hline
$\re(\bomega)$ & $\sr(\bomega)$ & $\re(\wt{\bsi})$ & $\sr(\wt{\bsi})$ & $\re(\vec{\bsi})$ & $\sr(\vec{\bsi})$ & $\Theta_1$ & ${\tt eff}(\Theta_1)$ & $\wh{\Theta}_2$ & ${\tt eff}(\wh{\Theta}_2)$  \\
\hline \hline
3.82E-01 &  --  & 1.36E-00 &  --  & 4.30E-00 &  --  & 5.45E-00 & 0.790 & 4.75E-00 & 0.906 \\
1.88E-01 & 1.06 & 6.17E-01 & 1.18 & 1.87E-00 & 1.25 & 2.47E-00 & 0.759 & 2.08E-00 & 0.900 \\
9.99E-02 & 0.95 & 3.06E-01 & 1.06 & 9.37E-01 & 1.04 & 1.27E-00 & 0.739 & 1.04E-00 & 0.905 \\
5.13E-02 & 0.98 & 1.51E-01 & 1.04 & 4.56E-01 & 1.06 & 6.37E-01 & 0.717 & 5.10E-01 & 0.895 \\
2.52E-02 & 1.04 & 7.62E-02 & 0.99 & 2.31E-01 & 0.99 & 3.18E-01 & 0.724 & 2.53E-01 & 0.910 \\
1.27E-02 & 0.99 & 3.75E-02 & 1.02 & 1.14E-01 & 1.01 & 1.58E-01 & 0.721 & 1.26E-01 & 0.909 \\
\hline
\end{tabular}
\caption{[Example 1] $\bbRT_{0}-\bP_{1}$ scheme with quasi-uniform refinement.}
\label{table:Ex1-k0}
}
\end{center}
\end{table} 
\begin{table}[ht] 
\begin{center}
\small{		
\begin{tabular}{r | c | c | c | c | c | c | c | c | c | c }
\hline
$\DOF$ & $h$ & $\iter$ & $\re(\bsi)$ & $\sr(\bsi)$ & $\re(\bu)$ & $\sr(\bu)$ & $\re(p)$ & $\sr(p)$ & $\re(\bG)$ & $\sr(\bG)$ \\
\hline \hline
570    & 0.373 & 4 & 4.66E-01 &  --  & 1.77E-01 &  --  & 5.06E-02 &  --  & 9.87E-02 &  --  \\
2258   & 0.196 & 4 & 1.12E-01 & 2.08 & 3.36E-02 & 2.41 & 9.64E-03 & 2.41 & 1.84E-02 & 2.44 \\
8714   & 0.097 & 4 & 2.84E-02 & 2.03 & 8.37E-03 & 2.06 & 2.42E-03 & 2.05 & 4.70E-03 & 2.02 \\
34338  & 0.048 & 4 & 7.26E-03 & 1.99 & 1.96E-03 & 2.12 & 5.78E-04 & 2.09 & 1.12E-03 & 2.10 \\
136462 & 0.025 & 4 & 1.82E-03 & 2.00 & 5.09E-04 & 1.95 & 1.48E-04 & 1.97 & 2.87E-04 & 1.97 \\
550094 & 0.013 & 4 & 4.43E-04 & 2.03 & 1.25E-04 & 2.02 & 3.59E-05 & 2.03 & 7.03E-05 & 2.02 \\
\hline
\end{tabular}
			
\smallskip
			
\begin{tabular}{c | c | c | c | c | c | c | c | c | c}
\hline
$\re(\bomega)$ & $\sr(\bomega)$ & $\re(\wt{\bsi})$ & $\sr(\wt{\bsi})$ & $\re(\vec{\bsi})$ & $\sr(\vec{\bsi})$ & $\Theta_1$ & ${\tt eff}(\Theta_1)$ & $\wh{\Theta}_2$ & ${\tt eff}(\wh{\Theta}_2)$  \\
\hline \hline
5.43E-02 &  --  & 1.80E-01 &  --  & 4.99E-01 &  --  & 9.29E-01 & 0.537 & 6.79E-01 & 0.734 \\
9.25E-03 & 2.57 & 3.45E-02 & 2.40 & 1.17E-01 & 2.11 & 1.89E-01 & 0.617 & 1.52E-01 & 0.770 \\
2.33E-03 & 2.05 & 8.87E-03 & 2.01 & 2.96E-02 & 2.03 & 4.79E-02 & 0.618 & 3.61E-02 & 0.819 \\
5.41E-04 & 2.13 & 2.11E-03 & 2.09 & 7.52E-03 & 2.00 & 1.16E-02 & 0.647 & 8.93E-03 & 0.843 \\
1.39E-04 & 1.97 & 5.44E-04 & 1.97 & 1.89E-03 & 2.00 & 2.97E-03 & 0.636 & 2.18E-03 & 0.869 \\
3.45E-05 & 2.00 & 1.33E-04 & 2.02 & 4.60E-04 & 2.03 & 7.28E-04 & 0.632 & 5.22E-04 & 0.882 \\
\hline
\end{tabular}
\caption{[Example 1] $\bbRT_{1}-\bP_{2}$ scheme with quasi-uniform refinement.}
\label{table:Ex1-k1}
}
\end{center}
\end{table}   
\begin{table}[ht] 
\begin{center}
\small{		
\begin{tabular}{r | c | c | c | c | c | c | c | c | c | c }
\hline
$\DOF$ & $h$ & $\iter$ & $\re(\bsi)$ & $\sr(\bsi)$ & $\re(\bu)$ & $\sr(\bu)$ & $\re(p)$ & $\sr(p)$ & $\re(\bG)$ & $\sr(\bG)$ \\
\hline \hline
1586   & 0.400 & 6 & 7.85E+02 &  --  & 9.89E-00 &  --  & 2.84E+01 &  --  & 3.20E+01 &  --  \\
6418   & 0.190 & 6 & 3.60E+02 & 1.12 & 4.43E-00 & 1.15 & 8.89E-00 & 1.66 & 1.44E+01 & 1.14 \\
24734  & 0.103 & 6 & 1.37E+02 & 1.43 & 9.74E-01 & 2.24 & 2.81E-00 & 1.71 & 4.92E-00 & 1.60 \\
97542  & 0.051 & 6 & 4.29E+01 & 1.70 & 2.02E-01 & 2.29 & 8.20E-01 & 1.79 & 1.43E-00 & 1.80 \\
389038 & 0.027 & 6 & 1.19E+01 & 1.86 & 4.33E-02 & 2.23 & 2.20E-01 & 1.90 & 3.96E-01 & 1.85 \\
\hline
\end{tabular}
			
\smallskip
			
\begin{tabular}{c | c | c | c | c | c | c | c | c | c}
\hline
$\re(\bomega)$ & $\sr(\bomega)$ & $\re(\wt{\bsi})$ & $\sr(\wt{\bsi})$ & $\re(\vec{\bsi})$ & $\sr(\vec{\bsi})$ & $\Theta_1$ & ${\tt eff}(\Theta_1)$ & $\wh{\Theta}_2$ & ${\tt eff}(\wh{\Theta}_2)$  \\
\hline \hline
1.59E+01 &  --  & 6.86E+01 &  --  & 7.85E+02 &  --  & 8.26E+02 & 0.950 & 7.83E+02 & 1.002 \\
6.62E+00 & 1.26 & 2.86E+01 & 1.25 & 3.60E+02 & 1.12 & 3.80E+02 & 0.947 & 3.60E+02 & 0.999 \\
2.19E-00 & 1.64 & 9.66E-00 & 1.61 & 1.37E+02 & 1.43 & 1.43E+02 & 0.958 & 1.37E+02 & 1.000 \\
6.09E-01 & 1.87 & 2.83E-00 & 1.79 & 4.29E+01 & 1.70 & 4.45E+01 & 0.965 & 4.29E+01 & 1.000 \\
1.76E-01 & 1.80 & 7.74E-01 & 1.87 & 1.19E+01 & 1.86 & 1.23E+01 & 0.965 & 1.19E+01 & 1.000 \\
\hline
\end{tabular}
\caption{[Example 2] $\bbRT_{1}-\bP_{2}$ scheme with quasi-uniform refinement.}
\label{table:Ex2-k1-uniform}
}
\end{center}
\end{table}   
\begin{table}[ht]
\begin{center}
\small{				
\begin{tabular}{r | c | c | c | c | c | c | c | c | c}
\hline
$\DOF$ & $\iter$ & $\re(\bsi)$ & $\sr(\bsi)$ & $\re(\bu)$ & $\sr(\bu)$ & $\re(p)$ & $\sr(p)$ & $\re(\bG)$ & $\sr(\bG)$ \\
\hline \hline
1586   & 6 & 7.85E+02 &  --  & 9.89E-00 &  --  & 2.84E+01 &  --  & 3.20E+01 &  --  \\
2256   & 6 & 2.48E+02 & 6.53 & 2.49E-00 & 7.83 & 5.57E-00 & 9.25 & 9.59E-00 & 6.85 \\
3434   & 6 & 5.22E+01 & 7.43 & 6.87E-01 & 6.13 & 1.46E-00 & 6.38 & 2.45E-00 & 6.50 \\
6936   & 6 & 1.77E+01 & 3.08 & 5.96E-01 & 0.41 & 8.72E-01 & 1.46 & 1.45E-00 & 1.48 \\
16656  & 6 & 7.59E-00 & 1.93 & 4.50E-01 & 0.64 & 3.79E-01 & 1.90 & 6.23E-01 & 1.93 \\
46862  & 6 & 2.71E-00 & 1.99 & 9.68E-02 & 2.97 & 1.22E-01 & 2.20 & 2.02E-01 & 2.18 \\
126940 & 6 & 1.01E-00 & 1.98 & 5.31E-02 & 1.20 & 5.02E-02 & 1.77 & 8.40E-02 & 1.76 \\
354680 & 6 & 3.62E-01 & 2.00 & 1.46E-02 & 2.52 & 1.74E-02 & 2.07 & 2.89E-02 & 2.07 \\
\hline
\end{tabular}
			
\smallskip
			
\begin{tabular}{c | c | c | c | c | c | c | c}
\hline
$\re(\bomega)$ & $\sr(\bomega)$ & $\re(\wt{\bsi})$ & $\sr(\wt{\bsi})$ & $\re(\vec{\bsi})$ & $\sr(\vec{\bsi})$ & $\Theta_1$ & ${\tt eff}(\Theta_1)$  \\
\hline \hline
1.59E+01 &  --  & 6.86E+01 &  --  & 7.85E+02 &  --  & 8.26E+02 & 0.950 \\
3.84E-00 & 8.09 & 1.93E+01 & 7.21 & 2.48E+02 & 6.53 & 2.57E+02 & 0.965 \\
9.54E-01 & 6.62 & 4.96E-00 & 6.47 & 5.22E+01 & 7.43 & 5.53E+01 & 0.944 \\
5.55E-01 & 1.54 & 2.96E-00 & 1.47 & 1.77E+01 & 3.08 & 2.08E+01 & 0.848 \\
2.46E-01 & 1.86 & 1.26E-00 & 1.94 & 7.60E-00 & 1.93 & 8.95E-00 & 0.849 \\
7.68E-02 & 2.25 & 4.12E-01 & 2.17 & 2.72E-00 & 1.99 & 3.08E-00 & 0.882 \\
3.24E-02 & 1.73 & 1.71E-01 & 1.77 & 1.01E-00 & 1.98 & 1.18E-00 & 0.861 \\
1.10E-02 & 2.10 & 5.89E-02 & 2.07 & 3.62E-01 & 2.01 & 4.15E-01 & 0.871 \\
\hline
\end{tabular}
\caption{[Example 2] $\bbRT_{1}-\bP_{2}$ scheme with adaptive refinement via $\Theta_1$.}
\label{table:Ex2-k1-adaptive-1}
}		
\end{center}
\end{table}
\begin{table}[ht]
\begin{center}
\small{				
\begin{tabular}{r | c | c | c | c | c | c | c | c | c}
\hline
$\DOF$ & $\iter$ & $\re(\bsi)$ & $\sr(\bsi)$ & $\re(\bu)$ & $\sr(\bu)$ & $\re(p)$ & $\sr(p)$ & $\re(\bG)$ & $\sr(\bG)$ \\
\hline \hline
1586   & 6 & 7.85E+02 &  --  & 9.89E-00 &  --  & 2.84E+01 &  --  & 3.20E+01 &  --  \\
2256   & 6 & 2.48E+02 & 6.53 & 2.49E-00 & 7.83 & 5.57E-00 & 9.25 & 9.59E-00 & 6.85 \\
3434   & 6 & 5.22E+01 & 7.43 & 6.87E-01 & 6.13 & 1.46E-00 & 6.38 & 2.45E-00 & 6.50 \\
6994   & 6 & 1.84E+01 & 3.12 & 6.00E-01 & 0.41 & 9.04E-01 & 1.43 & 1.51E-00 & 1.45 \\
16414  & 6 & 7.68E-00 & 1.95 & 5.01E-01 & 0.40 & 4.33E-01 & 1.64 & 6.99E-01 & 1.71 \\
42758  & 6 & 2.85E-00 & 2.07 & 1.76E-01 & 2.18 & 1.70E-01 & 1.96 & 2.84E-01 & 1.88 \\
120274 & 6 & 1.06E-00 & 1.91 & 8.09E-02 & 1.51 & 6.83E-02 & 1.76 & 1.14E-01 & 1.77 \\
319958 & 6 & 3.85E-01 & 2.07 & 2.13E-02 & 2.73 & 2.15E-02 & 2.36 & 3.63E-02 & 2.34 \\
\hline
\end{tabular}
			
\smallskip
			
\begin{tabular}{c | c | c | c | c | c | c | c}
\hline
$\re(\bomega)$ & $\sr(\bomega)$ & $\re(\wt{\bsi})$ & $\sr(\wt{\bsi})$ & $\re(\vec{\bsi})$ & $\sr(\vec{\bsi})$ & $\wh{\Theta}_2$ & ${\tt eff}(\wh{\Theta}_2)$  \\
\hline \hline
1.59E+01 &  --  & 6.86E+01 &  --  & 7.85E+02 &  --  & 7.83E+02 & 1.002 \\
3.84E-00 & 8.09 & 1.93E+01 & 7.21 & 2.48E+02 & 6.53 & 2.49E+02 & 0.999 \\
9.54E-01 & 6.62 & 4.96E-00 & 6.47 & 5.22E+01 & 7.43 & 5.23E+01 & 0.999 \\
5.79E-01 & 1.49 & 3.06E-00 & 1.44 & 1.84E+01 & 3.12 & 1.85E+01 & 0.997 \\
2.68E-01 & 1.72 & 1.43E-00 & 1.70 & 7.69E-00 & 1.95 & 7.79E-00 & 0.988 \\
1.10E-01 & 1.86 & 5.77E-01 & 1.90 & 2.85E-00 & 2.07 & 2.90E-00 & 0.986 \\
4.33E-02 & 1.80 & 2.32E-01 & 1.77 & 1.06E-00 & 1.91 & 1.07E-00 & 0.990 \\
1.41E-02 & 2.30 & 7.35E-02 & 2.35 & 3.86E-01 & 2.07 & 3.87E-01 & 0.997 \\
\hline
\end{tabular}
\caption{[Example 2] $\bbRT_{1}-\bP_{2}$ scheme with adaptive refinement via $\wh{\Theta}_2$.}
\label{table:Ex2-k1-adaptive-2}
}		
\end{center}
\end{table}
\begin{table}[ht]
\begin{center}
\small{				
\begin{tabular}{r | c | c | c | c | c | c | c | c | c | c}
\hline
$\DOF$ & $h$ & $\iter$ & $\re(\bsi)$ & $\sr(\bsi)$ & $\re(\bu)$ & $\sr(\bu)$ & $\re(p)$ & $\sr(p)$ & $\re(\bG)$ & $\sr(\bG)$ \\
\hline \hline
1221    & 0.354 & 4 & 1.72E+02 &  --  & 2.52E-00 &  --  & 8.67E-00 &  --  & 5.97E-00 &  --  \\
8559    & 0.177 & 4 & 1.84E+02 &  --  & 1.95E-00 & 0.40 & 6.64E-00 & 0.43 & 5.15E-00 & 0.23 \\
64059   & 0.088 & 4 & 1.37E+02 & 0.43 & 1.23E-00 & 0.69 & 3.55E-00 & 0.91 & 3.54E-00 & 0.56 \\
333609  & 0.051 & 4 & 9.28E+01 & 0.71 & 7.14E-01 & 0.99 & 1.98E-00 & 1.06 & 2.39E-00 & 0.72 \\
1276641 & 0.032 & 4 & 6.50E+01 & 0.80 & 4.15E-01 & 1.22 & 1.24E-00 & 1.04 & 1.64E-01 & 0.84 \\
\hline
\end{tabular}
		
\smallskip
		
\begin{tabular}{c | c | c | c | c | c | c | c | c | c}
\hline
$\re(\bomega)$ & $\sr(\bomega)$ & $\re(\wt{\bsi})$ & $\sr(\wt{\bsi})$ & $\re(\vec{\bsi})$ & $\sr(\vec{\bsi})$ & $\Theta_1$ & ${\tt eff}(\Theta_1)$ & $\wh{\Theta}_2$ & ${\tt eff}(\wh{\Theta}_2)$  \\
\hline \hline
3.29E-00 &  --  & 1.80E+01 &  --  & 1.72E+02 &  --  & 1.72E+02 & 1.001 & 1.72E+02 & 1.002 \\
2.56E-00 & 0.39 & 1.44E+01 & 0.34 & 1.84E+02 &  --  & 1.83E+02 & 1.001 & 1.83E+02 & 1.001 \\
1.62E-00 & 0.68 & 8.80E-00 & 0.74 & 1.37E+02 & 0.43 & 1.40E+02 & 1.001 & 1.37E+02 & 1.001 \\
1.05E-00 & 0.79 & 5.49E-00 & 0.86 & 9.28E+01 & 0.71 & 9.27E+01 & 1.001 & 9.28E+01 & 1.000 \\
7.00E-01 & 0.91 & 3.66E-00 & 0.91 & 6.50E+01 & 0.80 & 6.49E+01 & 1.000 & 6.50E+01 & 1.000 \\
\hline
\end{tabular}
\caption{[Example 3] $\bbRT_{0}-\bP_{1}$ scheme with quasi-uniform refinement.}
\label{table:Ex3-k0-uniform}
}		
\end{center}
\end{table}
\begin{table}[ht]
\begin{center}
\small{				
\begin{tabular}{r | c | c | c | c | c | c | c | c | c}
\hline
$\DOF$ & $\iter$ & $\re(\bsi)$ & $\sr(\bsi)$ & $\re(\bu)$ & $\sr(\bu)$ & $\re(p)$ & $\sr(p)$ & $\re(\bG)$ & $\sr(\bG)$ \\
\hline \hline
1221    & 4 & 1.72E+02 &  --  & 2.52E-00 &  --  & 8.67E-00 &  --  & 5.97E-00 &  --  \\
4761    & 4 & 1.85E+02 &  --  & 2.21E-00 & 0.29 & 6.64E-00 & 0.59 & 5.56E-00 & 0.16 \\
24309   & 4 & 1.39E+02 & 0.53 & 1.40E-00 & 0.84 & 3.54E-00 & 1.16 & 3.80E-00 & 0.70 \\
90255   & 4 & 8.74E+01 & 1.06 & 7.91E-01 & 1.31 & 1.84E-00 & 1.49 & 2.29E-00 & 1.16 \\
1040940 & 4 & 3.18E+01 & 1.24 & 3.28E-01 & 1.08 & 6.50E-01 & 1.28 & 9.10E-01 & 1.13 \\
\hline
\end{tabular}
			
\smallskip
			
\begin{tabular}{c | c | c | c | c | c | c | c}
\hline
$\re(\bomega)$ & $\sr(\bomega)$ & $\re(\wt{\bsi})$ & $\sr(\wt{\bsi})$ & $\re(\vec{\bsi})$ & $\sr(\vec{\bsi})$ & $\Theta_1$ & ${\tt eff}(\Theta_1)$  \\
\hline \hline
3.29E-00 &  --  & 1.80E+01 &  --  & 1.72E+02 &  --  & 1.71E+02 & 1.001 \\
2.80E-00 & 0.35 & 1.50E+01 & 0.41 & 1.85E+02 &  --  & 1.84E+02 & 1.002 \\
1.70E-00 & 0.92 & 9.16E-00 & 0.91 & 1.39E+02 & 0.53 & 1.39E+02 & 1.001 \\
9.12E-01 & 1.42 & 5.27E-00 & 1.26 & 8.74E+01 & 1.06 & 8.73E+01 & 1.000 \\
3.48E-01 & 1.18 & 2.02E-00 & 1.17 & 3.18E+01 & 1.24 & 3.17E+01 & 1.000 \\
\hline
\end{tabular}
\caption{[Example 3] $\bbRT_{0}-\bP_{1}$ scheme with adaptive refinement via $\Theta_1$.}
\label{table:Ex3-k0-adaptive-1}
}		
\end{center}
\end{table}
\begin{table}[ht]
\begin{center}
\small{				
\begin{tabular}{r | c | c | c | c | c | c | c | c | c}
\hline
$\DOF$ & $\iter$ & $\re(\bsi)$ & $\sr(\bsi)$ & $\re(\bu)$ & $\sr(\bu)$ & $\re(p)$ & $\sr(p)$ & $\re(\bG)$ & $\sr(\bG)$ \\
\hline \hline
1221    & 4 & 1.72E+02 &  --  & 2.52E-00 &  --  & 8.67E-00 &  --  & 5.97E-00 &  --  \\
4827    & 4 & 1.85E+02 &  --  & 2.21E-00 & 0.29 & 6.63E-00 & 0.59 & 5.56E-00 & 0.16 \\
24375   & 4 & 1.40E+02 & 0.51 & 1.42E-00 & 0.83 & 3.55E-00 & 1.16 & 3.80E-00 & 0.70 \\
112431  & 4 & 8.71E+01 & 0.93 & 7.62E-01 & 1.22 & 1.82E-00 & 1.31 & 2.29E-00 & 1.00 \\
1022031 & 4 & 3.18E+01 & 1.37 & 3.26E-01 & 1.15 & 6.60E-01 & 1.38 & 9.20E-01 & 1.24 \\
\hline
\end{tabular}
			
\smallskip
			
\begin{tabular}{c | c | c | c | c | c | c | c}
\hline
$\re(\bomega)$ & $\sr(\bomega)$ & $\re(\wt{\bsi})$ & $\sr(\wt{\bsi})$ & $\re(\vec{\bsi})$ & $\sr(\vec{\bsi})$ & $\wh{\Theta}_2$ & ${\tt eff}(\wh{\Theta}_2)$  \\
\hline \hline
3.29E-00 &  --  & 1.80E+01 &  --  & 1.72E+02 &  --  & 1.72E+02 & 1.002 \\
2.80E-00 & 0.35 & 1.50E+01 & 0.41 & 1.85E+02 &  --  & 1.84E+02 & 1.002 \\
1.71E-00 & 0.91 & 9.16E-00 & 0.91 & 1.40E+02 & 0.51 & 1.40E+02 & 1.001 \\
9.30E-01 & 1.20 & 5.24E-00 & 1.10 & 8.71E+01 & 0.93 & 8.71E+01 & 1.000 \\
3.50E-01 & 1.33 & 2.05E-00 & 1.28 & 3.18E+01 & 1.37 & 3.18E+01 & 1.000 \\
\hline
\end{tabular}
\caption{[Example 3] $\bbRT_{0}-\bP_{1}$ scheme with adaptive refinement via $\wh{\Theta}_2$.}
\label{table:Ex3-k0-adaptive-2}
}		
\end{center}
\end{table}
\begin{figure}[ht!]
\begin{center}
\includegraphics[height=5.5cm]{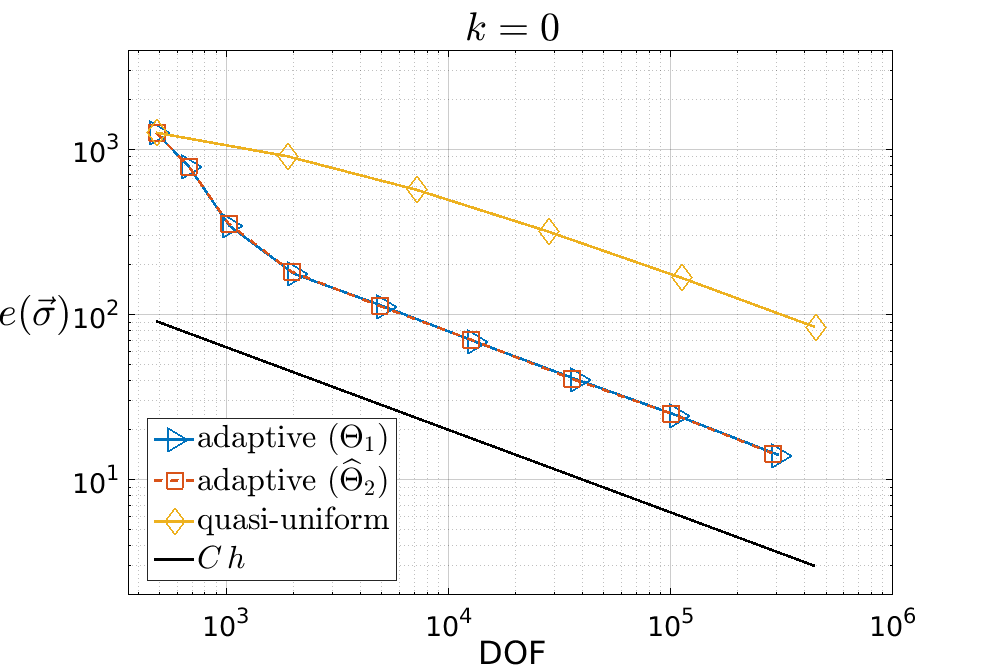} \hspace{0.5cm}
\includegraphics[height=5.5cm]{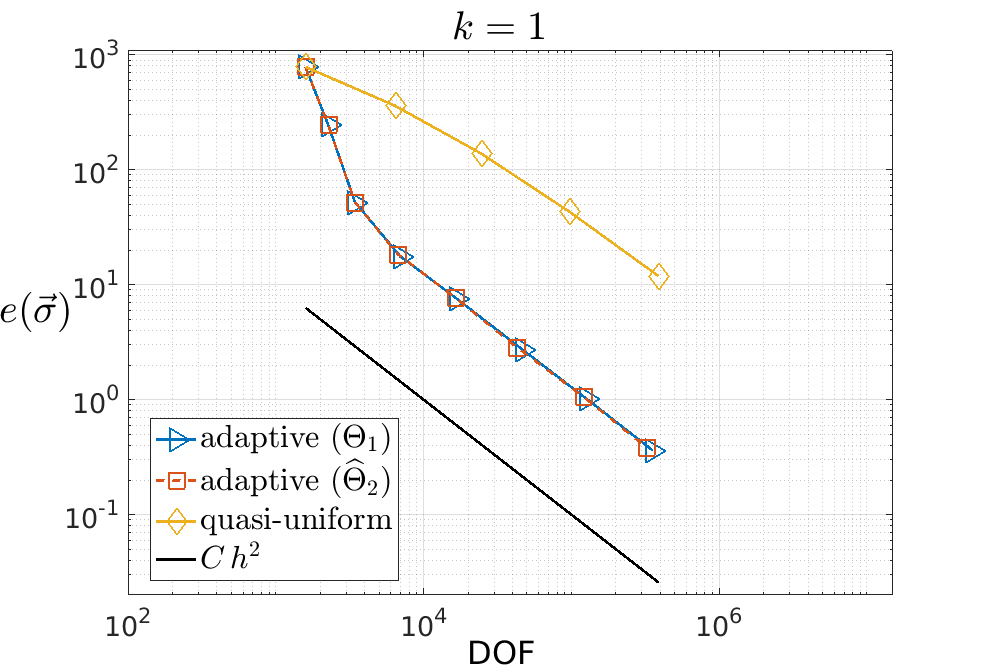}		

\vspace{-0.2cm}
		
\caption{[Example 2] Log-log plots of $\re(\vec{\bsi})$ vs. $\DOF$ for quasi-uniform/adaptative schemes via $\Theta_1$ and $\wh{\Theta}_2$ for $k = 0$ and $k = 1$ (left and right plots, respectively).}\label{figure:Ex2-loglog}
\end{center}
\end{figure}
\begin{figure}
\begin{center}
\includegraphics[width=4.15cm]{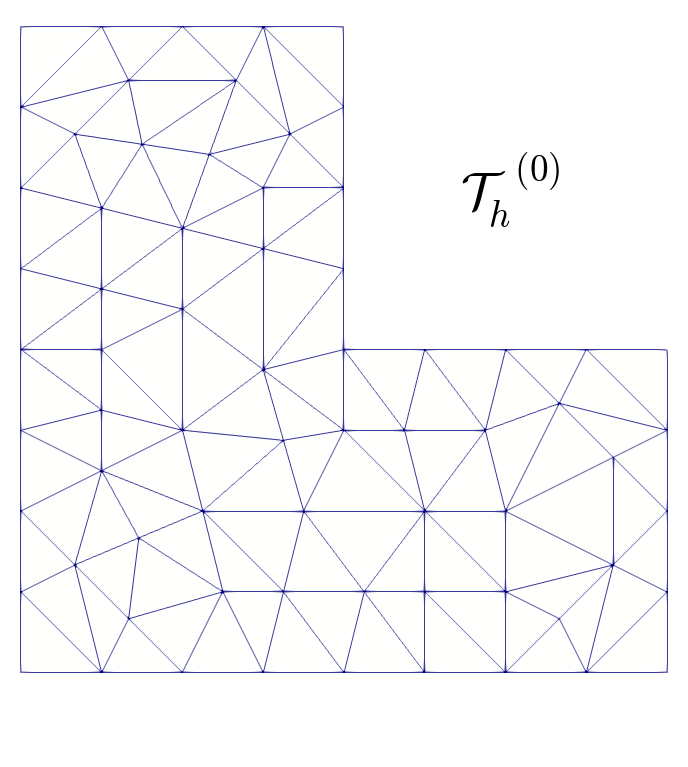} 
\includegraphics[width=4.15cm]{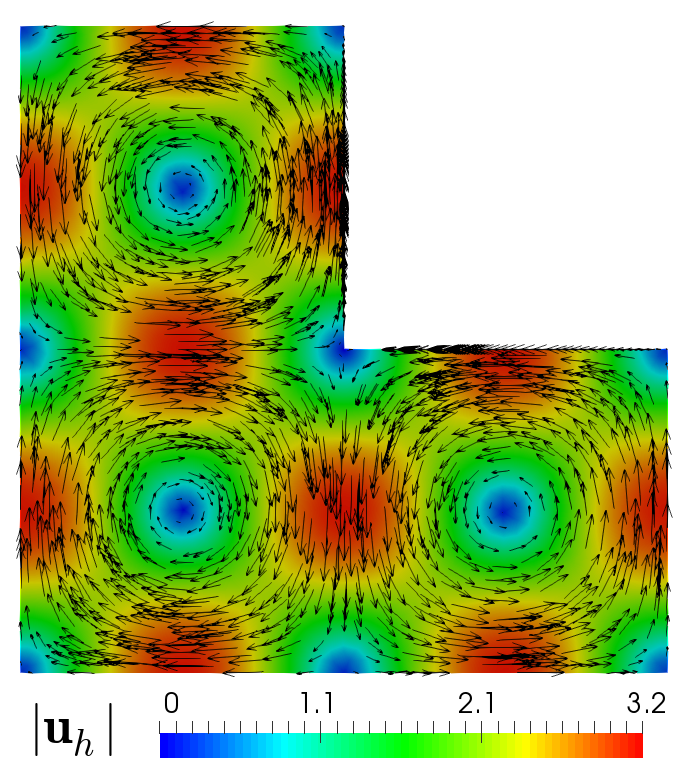} 
\includegraphics[width=4.15cm]{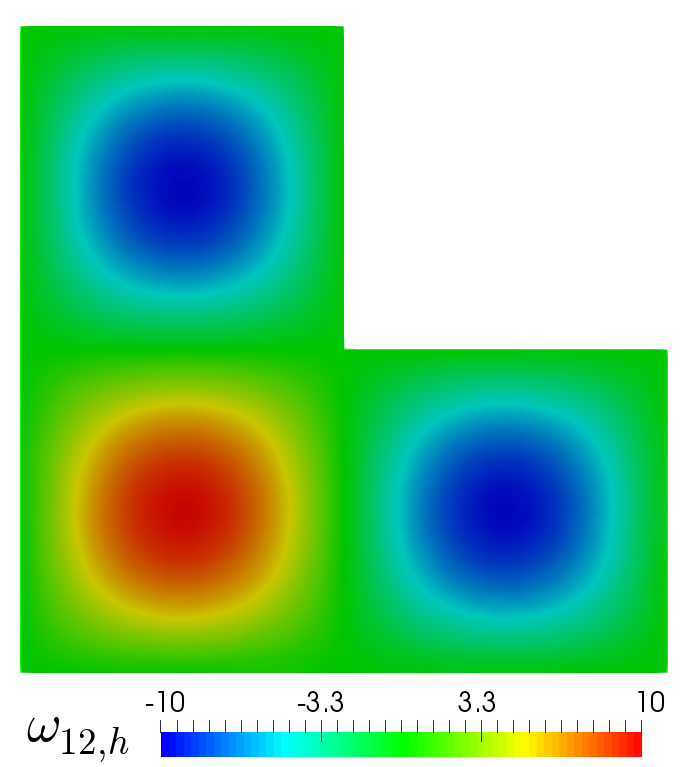} 
\includegraphics[width=4.15cm]{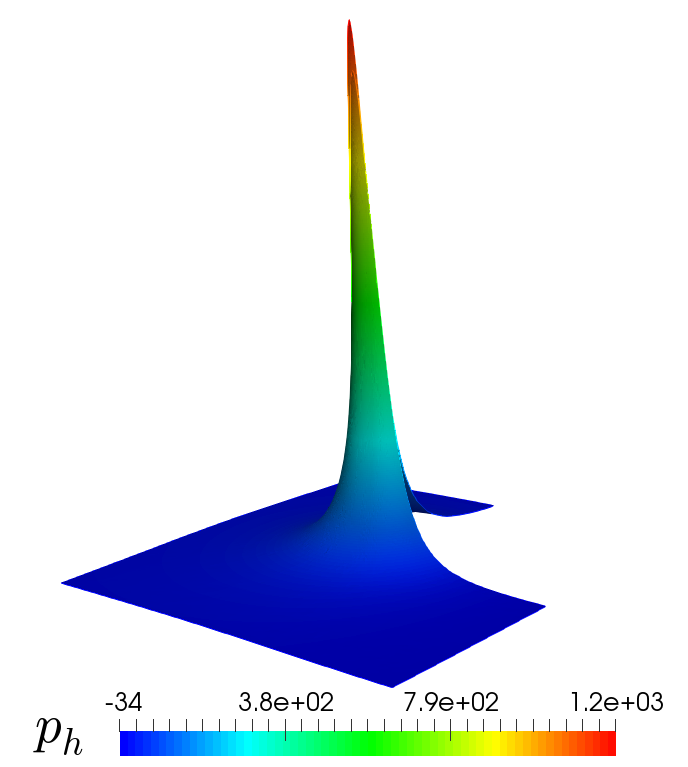} 

\vspace{-0.2cm}

\caption{[Example 2] Initial mesh, computed magnitude of the velocity, vorticity component, and pressure field.}\label{figure:Ex2-k0-aprox}
\end{center}
\end{figure}
\begin{figure}
\begin{center}
\includegraphics[width=4.15cm]{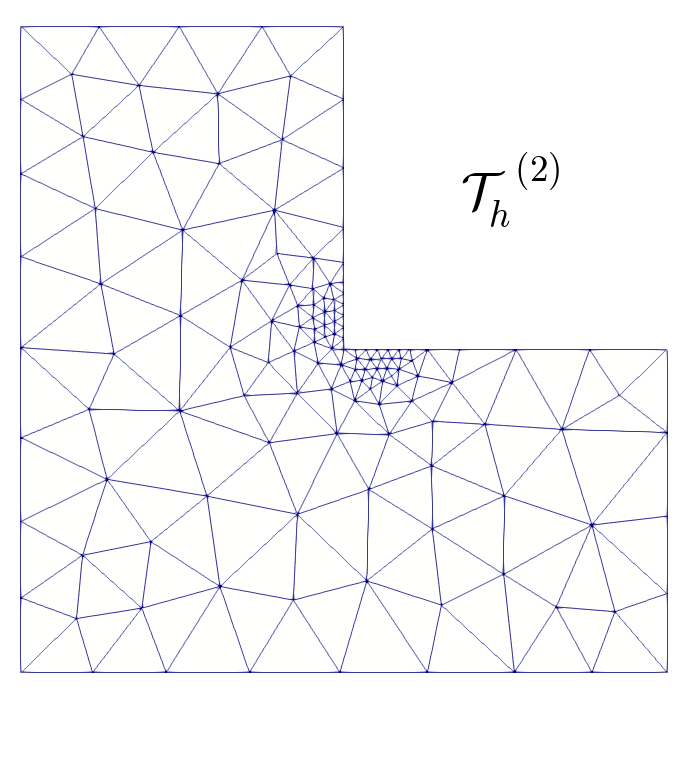} 
\includegraphics[width=4.15cm]{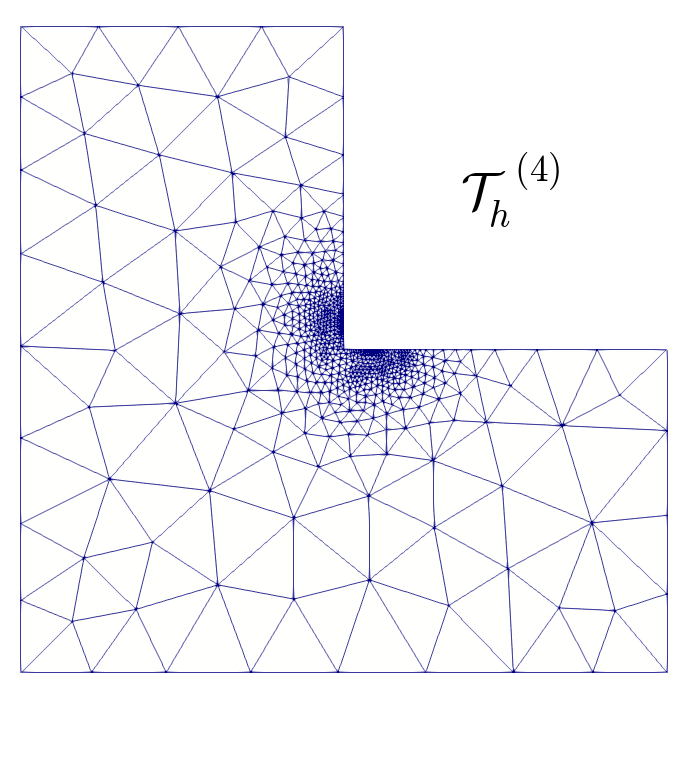} 
\includegraphics[width=4.15cm]{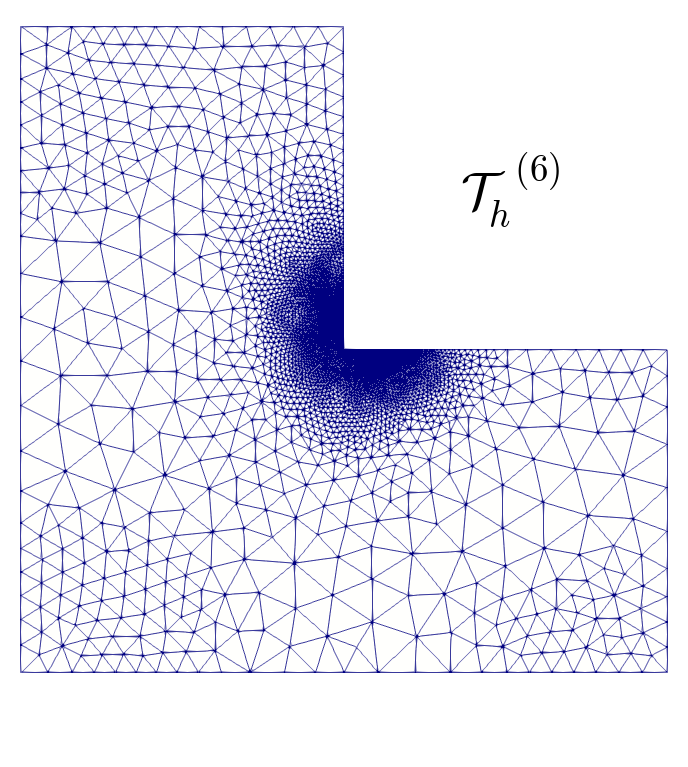} 
\includegraphics[width=4.15cm]{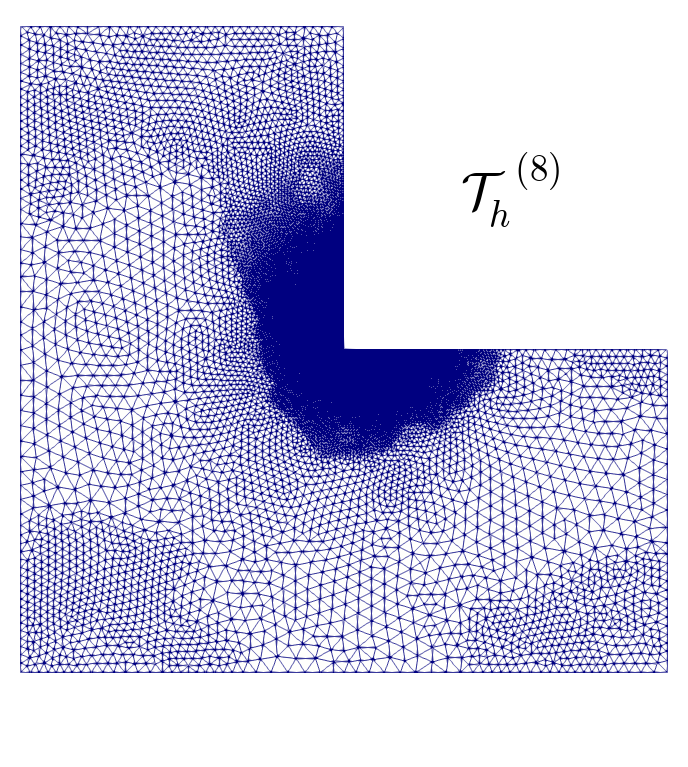} 
		
\includegraphics[width=4.15cm]{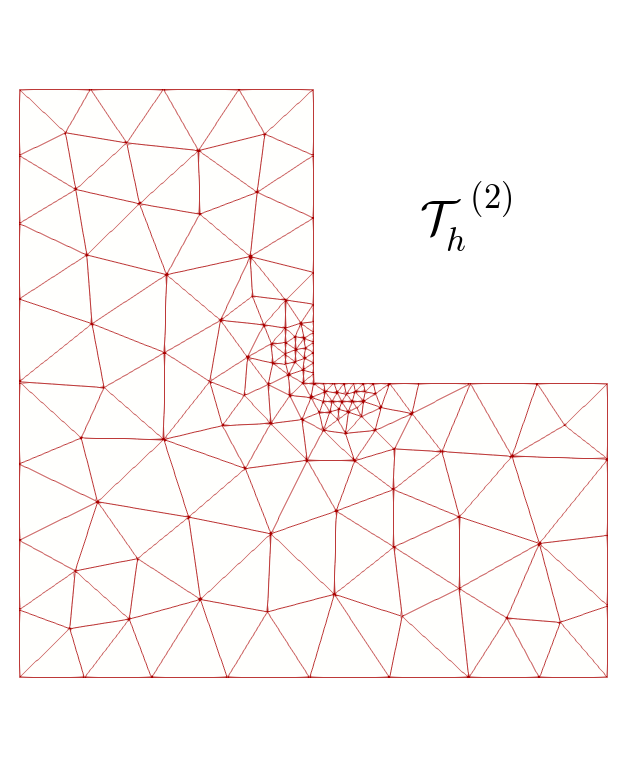} 
\includegraphics[width=4.15cm]{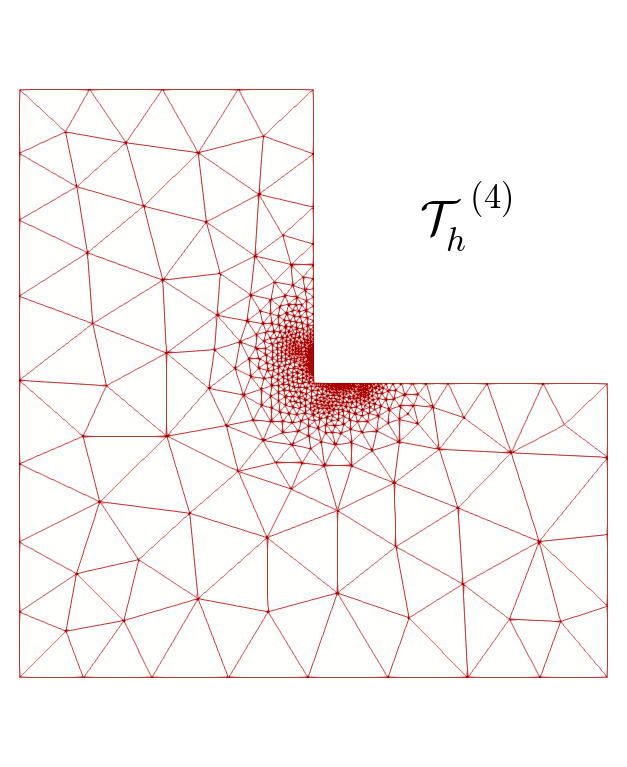} 
\includegraphics[width=4.15cm]{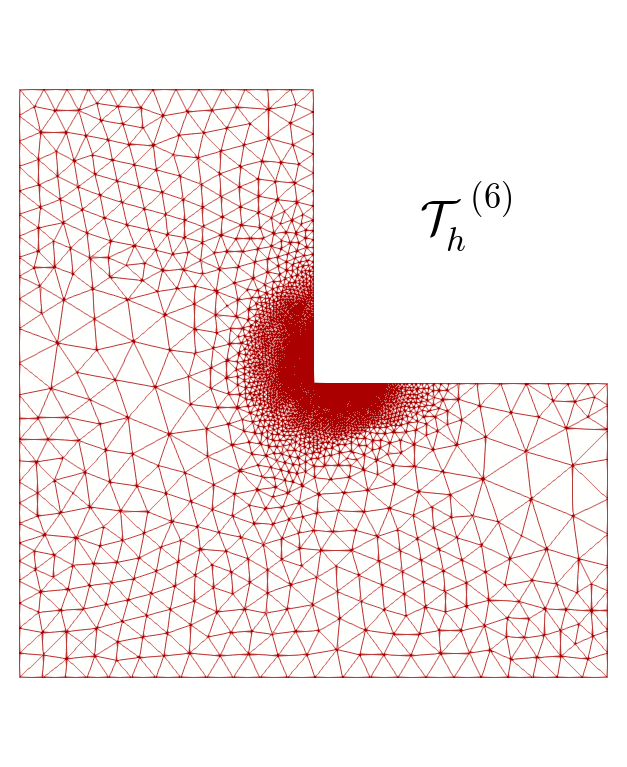} 
\includegraphics[width=4.15cm]{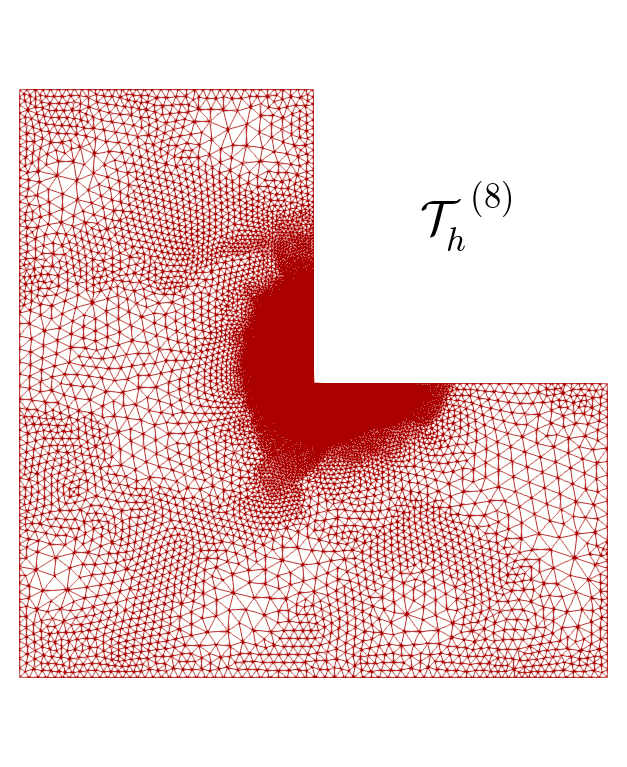} 

\vspace{-0.5cm}
		
\caption{[Example 2] Four snapshots of adapted meshes according to the indicators $\Theta_1$ and $\wh{\Theta}_2$ for $k = 0$ (top and bottom plots, respectively).}\label{figure:Ex2-Th-k0-k1}
\end{center}
\end{figure}
%
\begin{figure}
\begin{center}
\includegraphics[width=4.15cm]{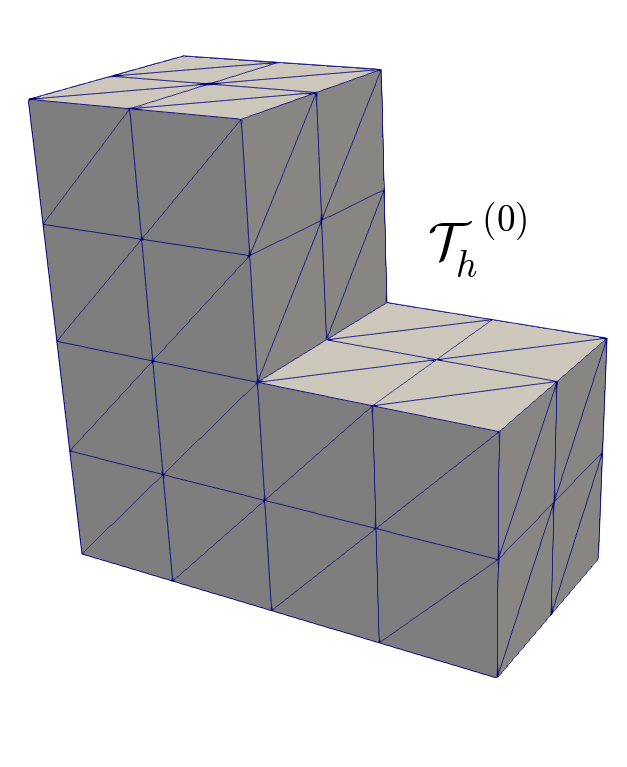} \qquad
\includegraphics[width=4.15cm]{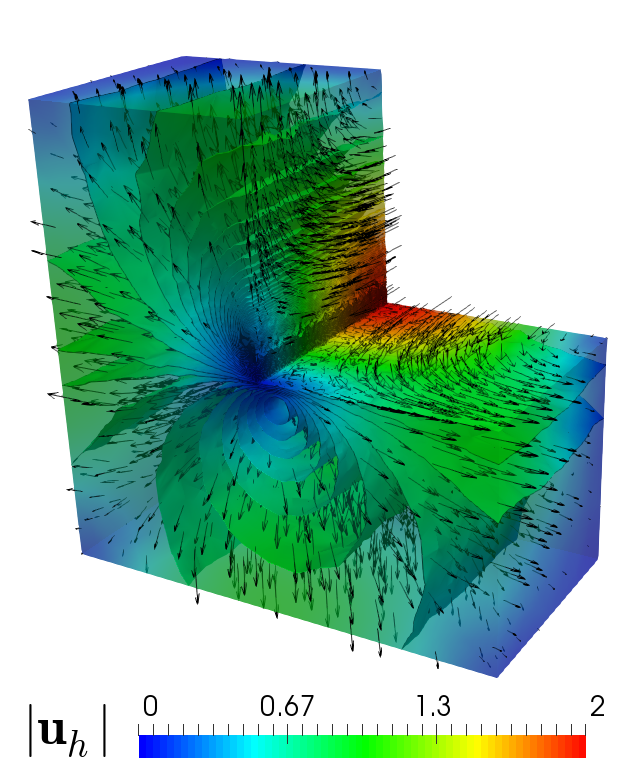} \qquad
\includegraphics[width=4.15cm]{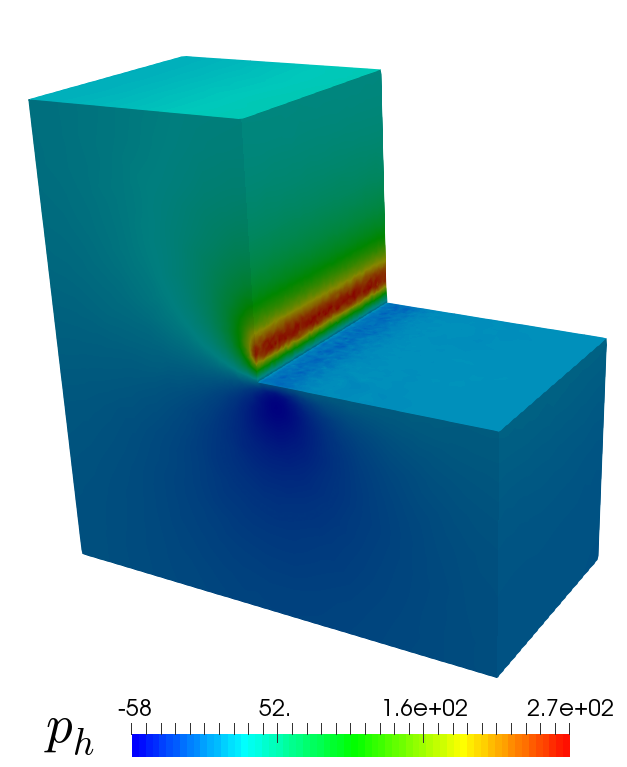} 

\vspace{-0.2cm}

\caption{[Example 3] Initial mesh, computed magnitude of the velocity, and pressure field.}
\label{figure:Ex3-k0-aprox}
\end{center}
\end{figure}

\begin{figure}
\begin{center}
\includegraphics[width=4.15cm]{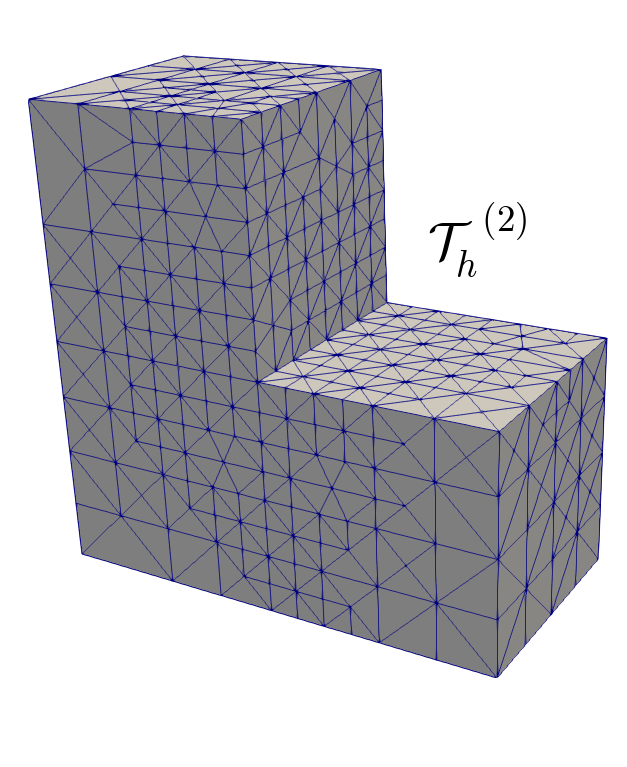} \qquad 
\includegraphics[width=4.15cm]{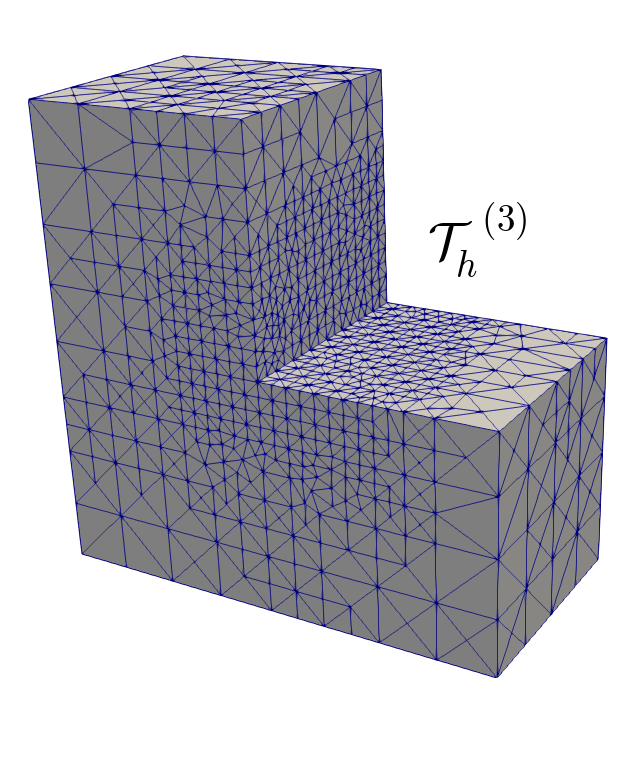} \qquad
\includegraphics[width=4.15cm]{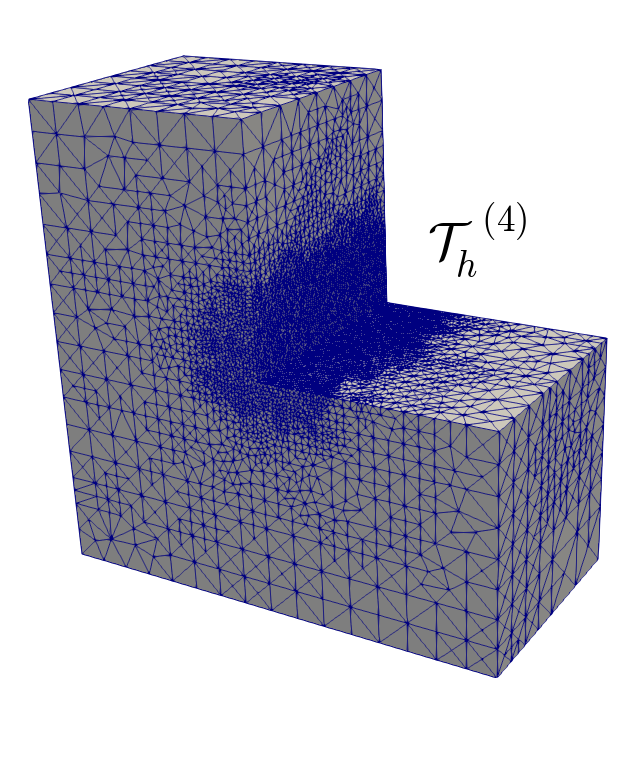}

\vspace{-0.6cm}
	
\caption{[Example 3] Three snapshots of adapted meshes according to the indicator $\Theta_1$ for $k = 0$.}\label{figure:Ex3-Th-k0}
\end{center}
\end{figure}
\begin{figure}
\begin{center}
\includegraphics[width=4.15cm]{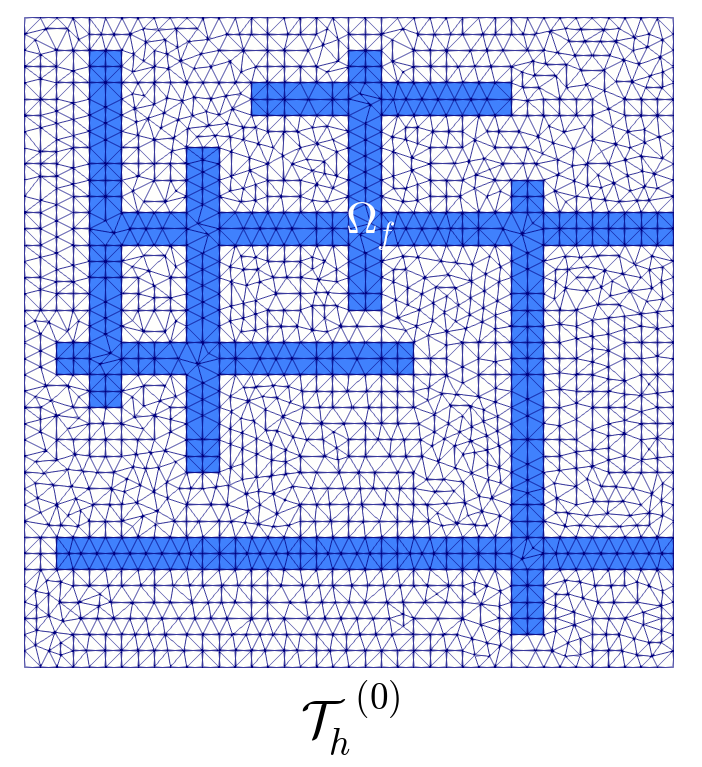}
\includegraphics[width=4.15cm]{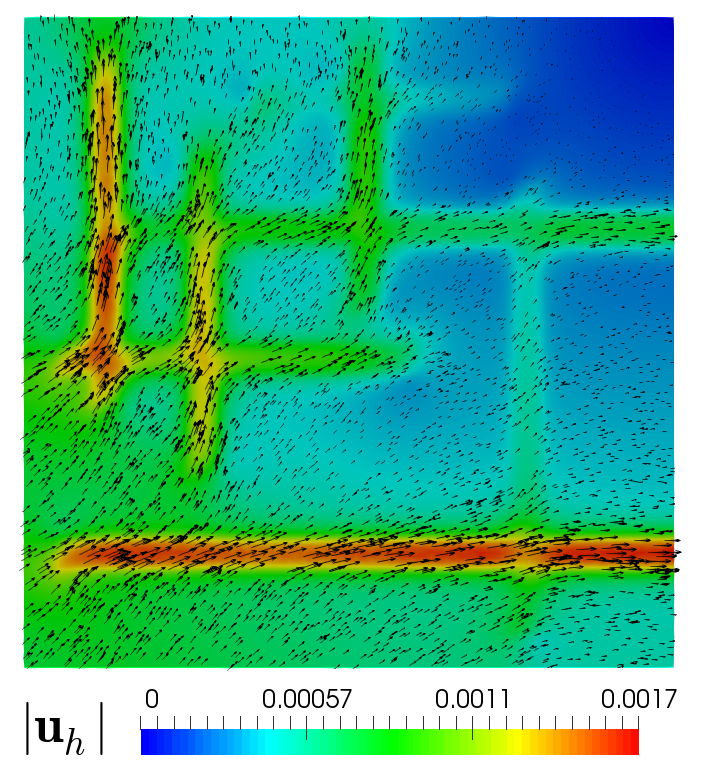} 
\includegraphics[width=4.15cm]{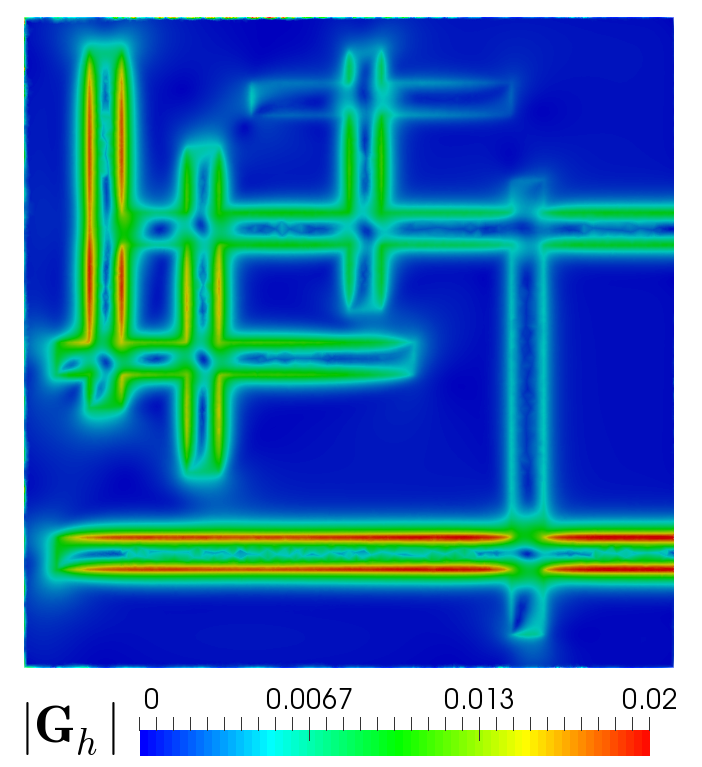}
\includegraphics[width=4.15cm]{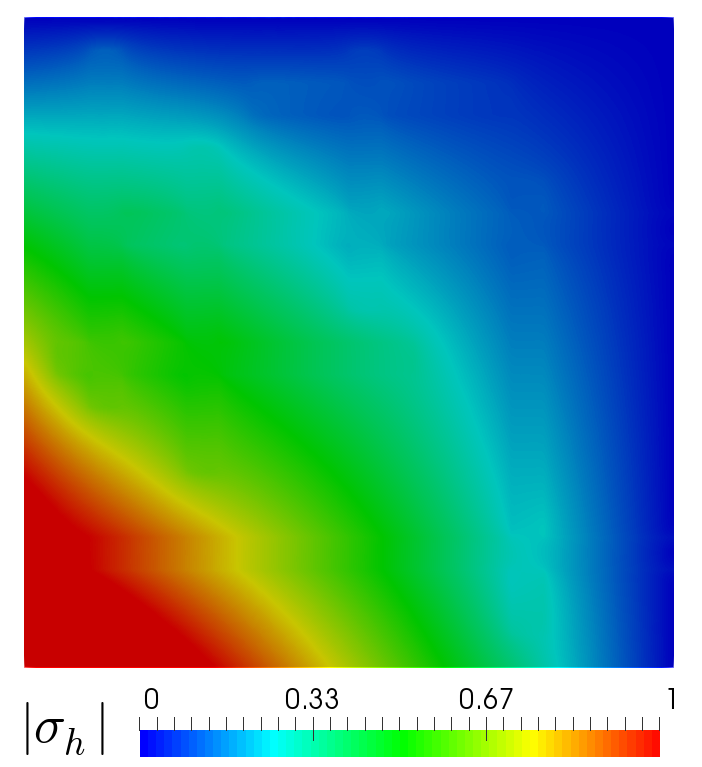} 
		
\caption{[Example 4] Initial mesh, computed magnitude of the velocity, velocity gradient tensor, and pseudostress tensor.}
\label{figure:Ex4-k0-aprox}
\end{center}
\end{figure}

\begin{figure}
\begin{center}
\includegraphics[width=4.15cm]{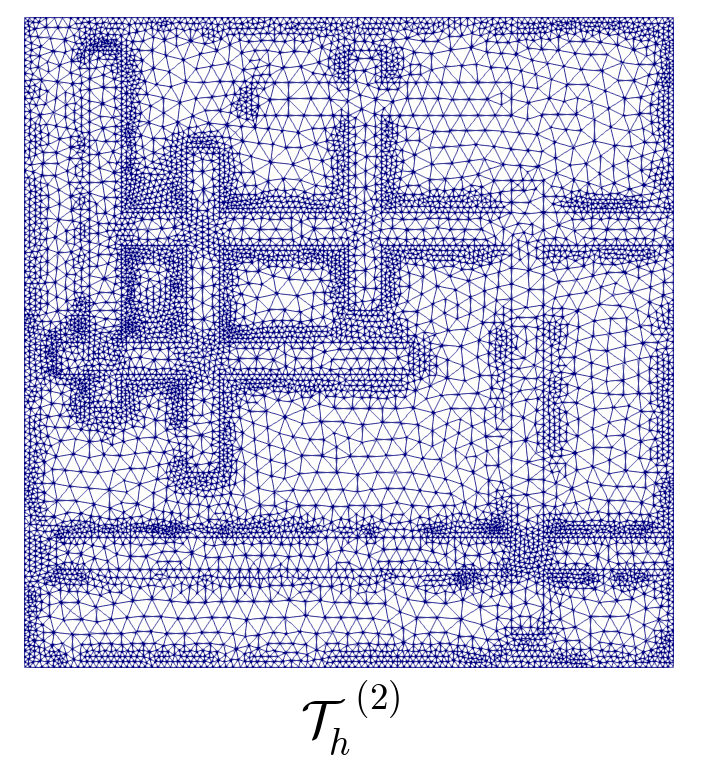}
\includegraphics[width=4.15cm]{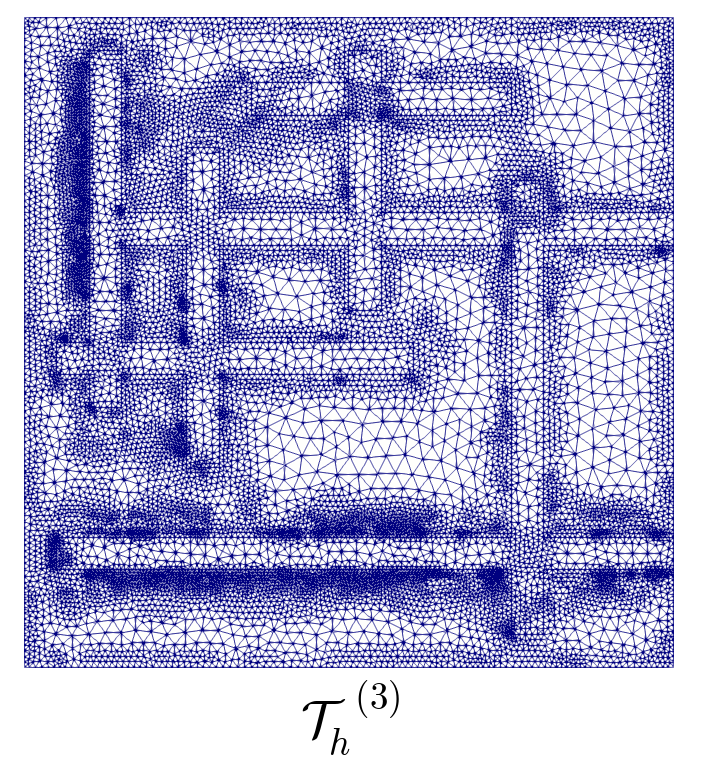}
\includegraphics[width=4.15cm]{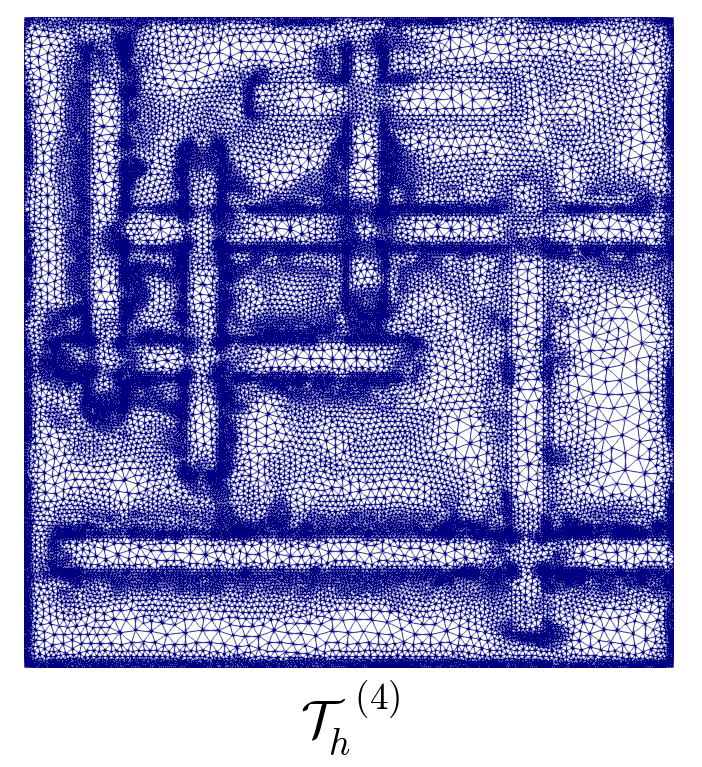}
\includegraphics[width=4.15cm]{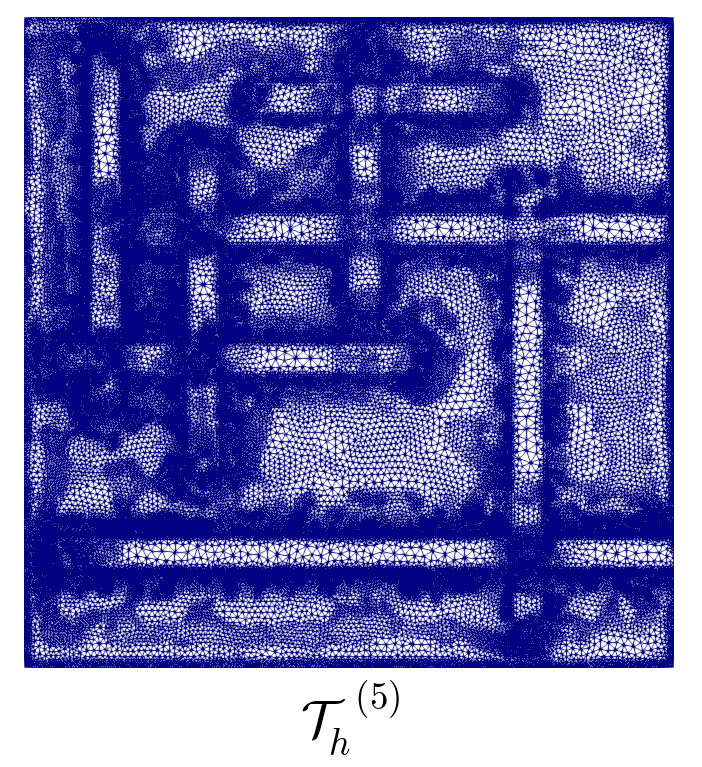}

\includegraphics[width=4.15cm]{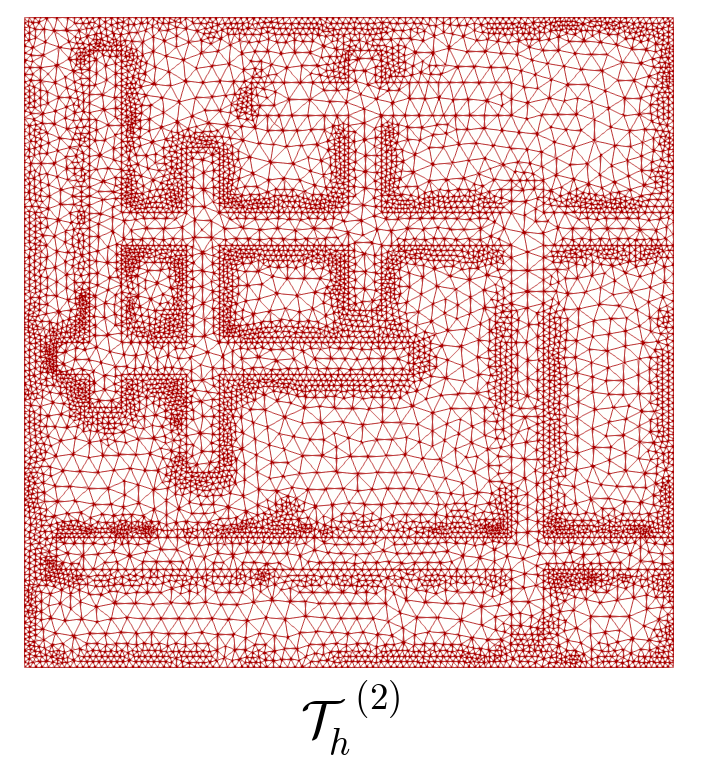}
\includegraphics[width=4.15cm]{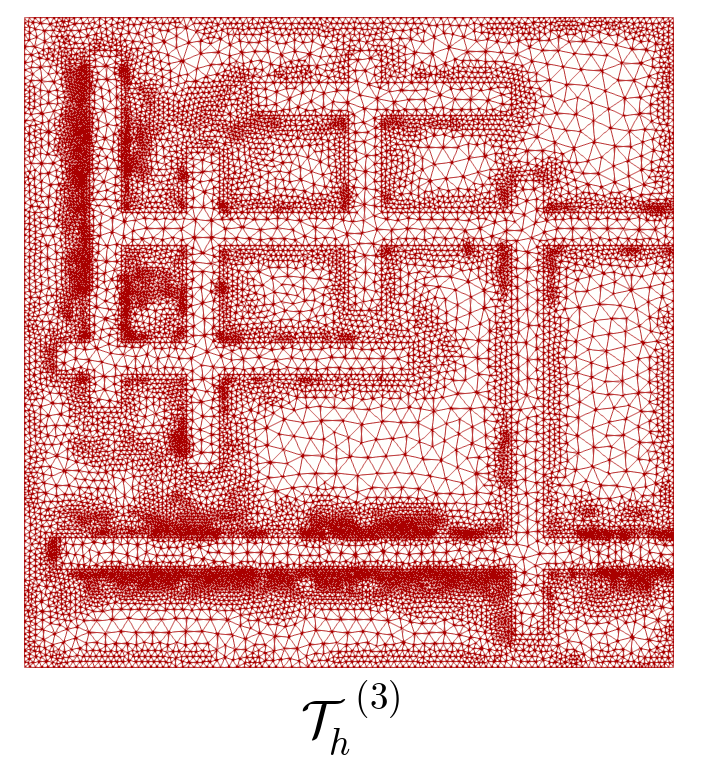}
\includegraphics[width=4.15cm]{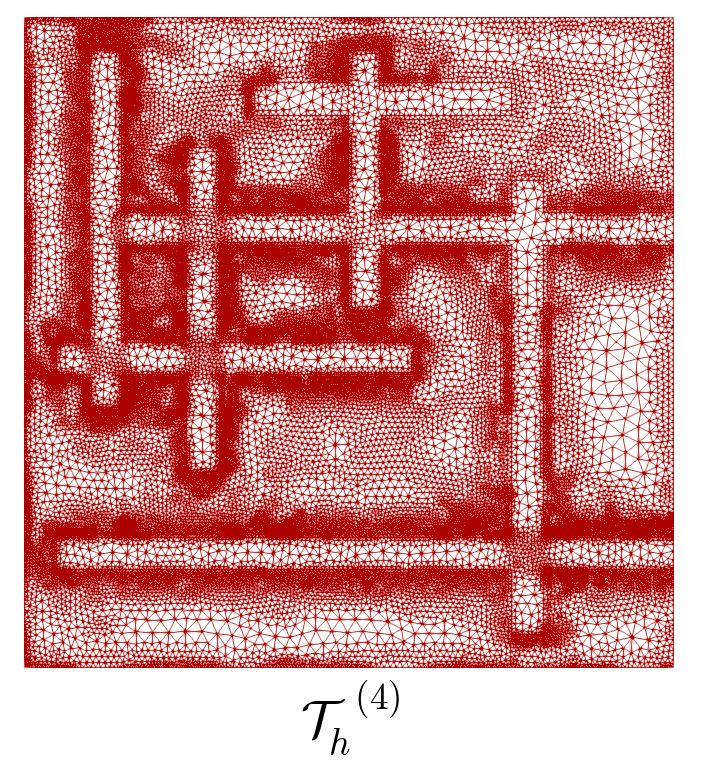}
\includegraphics[width=4.15cm]{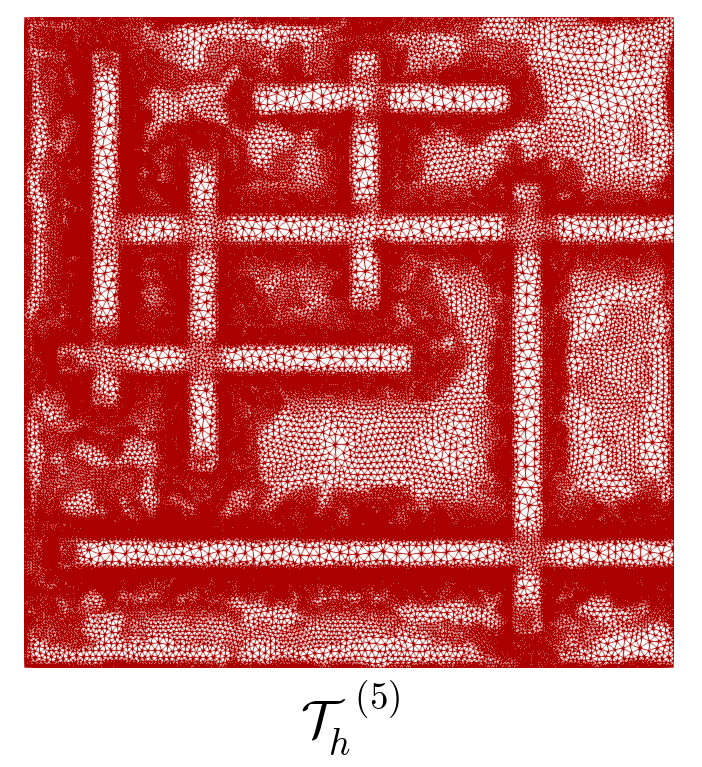}

\caption{[Example 4] Four snapshots of adapted meshes according to the indicators $\Theta_1$ and $\wh{\Theta}_2$ for $k = 1$ (top and bottom plots, respectively).}\label{figure:Ex4-Th-k0}
\end{center}
\end{figure}


\appendix

\section{Computing other variables of interest}\label{sec:appendix-A}

In this appendix we introduce suitable approximations for other variables of interest,
such as the pressure $p$, the velocity gradient $\bG:=\nabla\bu$,
the vorticity $\bomega:= \frac{1}{2}\left( \nabla\bu - (\nabla\bu)^\rt \right)$, and 
the shear stress tensor $\wt{\bsi} = \nu \left( \nabla\bu + (\nabla\bu)^\rt \right) - p\,\bI$, 
are all them written in terms of the solution of the discrete problem \eqref{eq:CBF-augmented-discrete-formulation-H0}.
In fact, using \eqref{eq:sigma-decomposition} and simple computations, we deduce that at the continuous level, there hold
\begin{equation}\label{eq:postprocess-vars}
\begin{array}{c}
\ds p = -\frac{1}{d} \tr(\bsi + \bu\otimes\bu) - \ell ,\quad
\bG = \frac{1}{\nu} \left( \bsi^\rd + (\bu\otimes \bu)^\rd \right) ,\quad
\bomega = \frac{1}{2 \nu} \left( \bsi - \bsi^\rt \right) ,\qan  \\[2ex]
\ds \wt{\bsi} = \bsi^\rd + (\bu\otimes \bu)^\rd + \bsi^\rt + (\bu\otimes\bu) + \ell\,\bbI ,\with
\ell = -\frac{1}{d\,|\Omega|} \int_\Omega \tr(\bu\otimes\bu) \,,
\end{array}
\end{equation}
provided the discrete solution $(\bsi_h,\bu_h)\in \bbH_h^\bsi\times \bH_h^\bu$ of problem \eqref{eq:CBF-augmented-discrete-formulation-H0}, we propose the following approximations for the aforementioned variables:
\begin{equation}\label{eq:postprocess-vars-app}
\begin{array}{c}
\ds p_h = -\frac{1}{d} \tr(\bsi_h + \bu_h\otimes\bu_h) - \ell_{h} ,\quad  
\bG_h = \frac{1}{\nu} \big(\bsi_h^\rd + (\bu_h\otimes\bu_h)^\rd \big) ,\quad
\bomega_h = \frac{1}{2 \nu} \big(\bsi_h - \bsi_h^\rt\big) ,\,\,\,\,\mbox{and}  \\[2ex]
\ds \wt{\bsi}_h = \bsi_h^\rd + (\bu_h\otimes\bu_h)^\rd + \bsi_h^\rt + (\bu_h\otimes\bu_h) + \ell_{h}\,\bbI\,, \with
\ell_{h} = -\frac{1}{d\,|\Omega|} \int_\Omega \tr(\bu_h\otimes\bu_h)\,.
\end{array}
\end{equation}
The following result, whose proof follows directly from Theorem \ref{thm:approximation}, 
establishes the corresponding approximation result for this post-processing procedure.
\begin{lem}\label{lem:rate-of-convergence-further-variables} 
Let $(\bsi,\bu)\in \bbH_0(\bdiv;\Omega)\times \bH^1(\Omega)$ be the unique solution of 
the continuous problem \eqref{eq:CBF-augmented-weak-formulation}, and 
let $p$, $\bG$, $\bomega$ and $\wt{\bsi}$ given by \eqref{eq:postprocess-vars}. 
In addition, let $p_h$, $\bG_h$, $\bomega_h$ 
and $\wt{\bsi}_h$ be the discrete counterparts introduced in \eqref{eq:postprocess-vars-app}.
Let $l\in (0,k+1]$ and assume that the hypotheses of the Theorem \ref{thm:approximation} be hold. 
Then, there exists $C>0$, independent of $h$, such that
\begin{equation}
\|p - p_h\|_{0,\Omega} + \|\bG - \bG_h\|_{0,\Omega} 
+ \|\bomega - \bomega_h\|_{0,\Omega} + \|\wt{\bsi} - \wt{\bsi}_h\|_{0,\Omega} 
\,\leq\, C\,h^l\,\Big\{ \|\bsi\|_{l,\Omega} + \|\bdiv(\bsi)\|_{l,\Omega} + \|\bu\|_{l+1,\Omega} \Big\} \,.
\end{equation}
\end{lem}
\begin{proof}
First, from \eqref{eq:postprocess-vars} and \eqref{eq:postprocess-vars-app}, adding and subtracting
$\bu\otimes\bu_h$ (also work with $\bu_h\otimes \bu$), employing the triangle and H\"older inequalities,
it is not difficult to see that that there exists a $C>0$, depending only on data and other constants,
all of them independent of $h$, such that
\begin{equation}
\begin{array}{l}
\ds \|p - p_h\|_{0,\Omega} + \|\bG - \bG_h\|_{0,\Omega} 
+ \|\bomega - \bomega_h\|_{0,\Omega} + \|\wt{\bsi} - \wt{\bsi}_h\|_{0,\Omega} \\[1ex]
\ds\quad \leq\, C\,\Big\{ \|\bsi -\bsi_h\|_{\bdiv;\Omega} + \big( \|\bu\|_{1,\Omega} + \|\bu_h\|_{1,\Omega} \big)\|\bu - \bu_h\|_{1,\Omega} \Big\}\,.
\end{array}
\end{equation}
Then, using the fact that $\bu\in \bW_r$ and $\bu_h\in \wt{\bW}_r$, the result follows from a direct application of Theorem \ref{thm:approximation}.
\end{proof}


\section{Preliminaries for the \textit{a posteriori} error analysis}\label{sec:appendix-B}

We start by introducing a few useful notations for describing local information on elements
and edges or faces depending on wether $d=2$ or $d=3$, respectively.
Let $\cE_h$ be the set of edges or faces of $\cT_h$, whose corresponding diameters are denoted 
by $h_e$, and define
\begin{equation*}
\cE_h(\Omega) \,:=\, \big\{ e \in \cE_h :\quad e\subseteq \Omega \big\} \qan 
\cE_h(\Gamma) \,:=\, \big\{ e \in \cE_h :\quad e\subseteq \Gamma \big\}\,.
\end{equation*}
For each $T\in \cT_h$, we let $\cE_{h,T}$ be the set of edges or faces of $T$, and denote
\begin{equation*}
\cE_{h,T}(\Omega) \,=\, \big\{ e \in \partial T :\quad e\subseteq \cE_h(\Omega) \big\} \qan 
\cE_{h,T}(\Gamma) \,=\, \big\{ e \in \partial T :\quad e\subseteq \cE_h(\Gamma) \big\} \,.
\end{equation*}
We also define the unit normal vector $\bn_e$ on each edge or face by
\begin{equation*}
\bn_e \,:=\, (n_1, \dots, n_d)^\rt \quad \forall\,e\in \cE_h \,.
\end{equation*}
Hence, when $d=2$ we can define the tangential vector $\bs_e$ by
\begin{equation*}
\bs_e \,:=\, (-n_2, n_1 )^\rt \quad \forall\,e\in \cE_h \,.
\end{equation*}
However, when no confusion arises, we will simply write $\bn$ and $\bs$ 
instead of $\bn_e$ and $\bs_e$, respectively.

The usual jump operator $\jump{\cdot}$ across internal edges or faces is defined 
for piecewise continuous matrix, vector, or scalar-valued functions $\bzeta$, by
\begin{equation*}
\jump{\bzeta} \,=\, (\bzeta|_{T_+})|_e - (\bzeta|_{T_-})|_e \textup{ with } e \,=\, \partial T_+ \cap \partial T_- \,,
\end{equation*}
where $T_+$ and $T_-$ are the elements of $\cT_h$ having $e$ as a common edge or face. 
Finally, for sufficiently smooth scalar $\psi$, vector $\bv := (v_1, \dots , v_d)^\rt$, 
and tensor fields $\btau := (\tau_{i j})_{i,j=1,d}$, we let
\[
\ubgamma_{*}(\btau)\,=\,\left\lbrace 
\begin{array}{cl}
\quad\btau\bs &,\, \text{ for } d=2 \,, \\[0.5ex]
\begin{pmatrix}
(\btau^\rt_1\times\bn)^\rt \\ (\btau^\rt_2\times\bn)^\rt \\ (\btau^\rt_3\times\bn)^\rt
\end{pmatrix} &,\, \text{ for } d=3 \,,
\end{array}\right.,\quad
\bcurl(\bv) \,:=\, \left(
\begin{array}{cc}
-\dfrac{\partial\,v_1}{\partial x_2} & \dfrac{\partial\,v_1}{\partial x_1} \\[2ex]  
-\dfrac{\partial\,v_2}{\partial x_2} & \dfrac{\partial\,v_2}{\partial x_1} 
\end{array}
\right)\,
\quad \text{for } d=2 \,,
\]
\[	
\ds\ucurl(\bv)\,:=\,\left\lbrace
\begin{array}{cl}
\ds\frac{\partial v_2}{\partial x_1} - \frac{\partial v_1}{\partial x_2}&,\, \text{ for } d=2 \,, \\[2ex]
\nabla\times\bv&,\, \text{ for } d=3 \,, 
\end{array} \right.\qquad
\ubcurl(\btau)\,=\,\left\lbrace 
\begin{array}{ll}
\begin{pmatrix}
\ucurl(\btau^\rt_1) \\ \ucurl(\btau^\rt_2)
\end{pmatrix} &,\, \text{ for } d=2 \,, \\[2.5ex]
\begin{pmatrix}
\ucurl(\btau^\rt_1)^\rt \\ \ucurl(\btau^\rt_2)^\rt \\ \ucurl(\btau^\rt_3)^\rt
\end{pmatrix} &,\, \text{ for } d=3 \,, 
\end{array}\right.
\]
where $\btau_i$ is the $i$-th row of $\btau$ and the derivatives involved 
are taken in the  distributional sense. 

Now, let $\bI_h:\bH^1(\Omega)\to \bH^1_h(\Omega)$ be the vector version of the usual Cl\'ement interpolation operator (cf. \cite{clement1975}), where
\begin{equation*}
\bH^1_h(\Omega) := 
\Big\{ \bv_h\in \bC(\ov{\Omega}) : \quad \bv_h|_T\in \bP_1(T) \quad \forall \, T\in\cT_h \Big\}\,,
\end{equation*}
and let $\bPi^k_h:\bbH^1(\Omega)\to\bbH^\bsi_h$ (cf. \eqref{eq:subspaces-a}) be the Raviart--Thomas interpolator, which, according to its characterization properties (see, e.g., \cite[Section 3.4.1]{Gatica}), verifies 
\begin{equation*}
\bdiv(\bPi^k_h(\btau)) = \cP^k_h(\bdiv(\btau)) \quad \forall\,\btau\in \bbH^1(\Omega) \,,
\end{equation*}
where $\cP^k_h$ is the vectorial version of the $\L^2(\Omega)$-orthogonal projector onto the picewise polynomials of degree $\leq k$ on $\Omega$. 
Further approximation properties of $\bI_h$ and $\bPi^k_h$ are summarized in the following lemmas
(see a proof in e.g. \cite{clement1975} and \cite[Lemma 3.16 and 3.18]{Gatica}, respectively).
\begin{lem}\label{lem:clement} 
There exist $c_1, c_2>0$, independent of $h$, 
such that for all $\bv\in \bH^1(\Omega)$ there hold
\begin{equation*}
\|\bv - \bI_h(\bv)\|_{0,T} \,\leq\, c_1\, h_T\,\|\bv\|_{1,\Delta(T)} \quad \forall\, T\in\cT_h,
\end{equation*}
and
\begin{equation*}
\|\bv - \bI_h(\bv)\|_{0,e} \,\leq\, c_2\,h^{1/2}_e\,\|\bv\|_{1,\Delta(e)} \quad \forall\, e\in\cE_h,
\end{equation*}
where 
$\Delta(T) := \cup \Big\{T'\in\cT_h: \,\, T'\cap T\neq \emptyset \Big\}$ and 
$\Delta(e) := \cup \Big\{T'\in\cT_h: \,\, T'\cap e \neq \emptyset \Big\}$.
\end{lem}

\begin{lem}\label{lem:Phih-properties}
There exist $C_1, C_2 >0$, independent of $h$, such that for all $\btau\in\bbH^1(\Omega)$ there hold
\begin{equation*}
\|\btau - \bPi^k_h(\btau)\|_{0,T} \,\leq\, C_1\,h_T\,\|\btau\|_{1,T} \quad \forall\, T\in\cT_h\,,
\end{equation*}
and
\begin{equation*}
\|\btau\bn - \bPi^k_h(\btau)\bn\|_{0,e} \,\leq\, C_2\,h^{1/2}_e\,\|\btau\|_{1,T_e} \quad \forall\, e\in\cE_h,
\end{equation*}
where $T_e$ is a triangle of $\cT_h$ containing the edge $e$ on its boundary.
\end{lem}

We end this appendix by recalling a stable Helmholtz decompositions for $\bbH(\bdiv;\Omega)$. 
More precisely, we have the following lemma.
\begin{lem}\label{lem:helmholtz-decomposition}
For each $\btau \in \bbH(\bdiv;\Omega)$ there exist
\begin{itemize}
\item[a)] $\bz\in \bH^{2}(\Omega)$ and  $\bchi\in \bH^1(\Omega)$ such that $\btau = \nabla\bz + \bcurl(\bchi)$ when $d=2$,
		
\item[b)] $\bz\in \bH^{2}(\Omega)$ and  $\bchi\in \bbH^1(\Omega)$ such that $\btau = \nabla\bz + \ubcurl(\bchi)$ when $d=3$.
\end{itemize}
In addition, in both cases, 
\begin{equation*}
\|\bz\|_{2;\Omega} + \|\bchi\|_{1,\Omega} \,\leq\, C_{\tt Hel}\,\|\btau\|_{\bdiv; \Omega}, 
\end{equation*}	
where $C_{\tt Hel}$ is a positive constant independent of all the foregoing variables.
\end{lem}
\begin{proof}
For the proof of $a)$ and $b)$ we refer to \cite[Lemma 3.7]{grt2016} and \cite[Theorem~3.1]{g2020}, respectively. 
We omit further details.
\end{proof}


\section{A second \textit{a posteriori} error estimator}\label{sec:appendix-C}

In this appendix we introduce and analyze another {\it a posteriori} error estimator 
for the augmented mixed finite element scheme \eqref{eq:CBF-augmented-discrete-formulation-H0}, 
which is not based on the Helmholtz decomposition for $\btau\in\bbH(\bdiv;\Omega)$. 
More precisely, this second estimator arises simply from a different way
of bounding $\|\cR_1\|_{\bbH_0(\bdiv;\Omega)'}$ in the preliminary estimate 
for the total error given by \eqref{eq:error-bound-by-R1-R2}. 
Then, with the same notations and discrete spaces from Section \ref{sec:discrete-setting} 
and Appendix \ref{sec:appendix-B}, we now introduce for each $T\in\cT_h$ the local error indicator 
\begin{equation*}
\begin{array}{l}
\ds \wt{\Theta}_{2,T}^2 \,:=\, \Big\|\nabla\bu_h - \frac{1}{\nu}\big(\bsi_h + (\bu_h\otimes\bu_h)\big)^{\rd}\Big\|_{0,T}^2 \\[2ex]
\ds\quad +\, \|\alpha\,\bu_h + \tF\,|\bu_h|^{\rp-2}\bu_h - \bdiv(\bsi_h) - \f\|_{0,T}^2  
+ \sum_{e\in\cE_{h,T}(\Gamma)} \|\bu_\rD - \bu_h\|_{0,e}^2 \,,
\end{array}
\end{equation*}
and define the following global residual error estimator
\begin{equation*}
\Theta_2 \,:=\, \left\{ \sum_{T\in\cT_h} \wt{\Theta}_{2,T}^2 
+ \|\bu_\rD - \bu_h\|_{1/2,\Gamma}^2 \right\}^{1/2} \,.
\end{equation*}

The reliability and efficiency of the {\it a posteriori} error estimator $\Theta_2$ are stated next.
\begin{thm}
Assume that the data $\f$ and $\bu_\rD$ satisfy \eqref{eq:data-assumption-Cea}. 
Then there exist positive constants $\wt{C}_{\tt rel}$ and $\wt{C}_{\tt eff}$, independent of $h$, such that
\begin{equation}\label{eq:estimator2-bound}
\wt{C}_{\tt eff}\,\Theta_2 + {\tt h.o.t.} \,\leq\, \|(\bsi,\bu) - (\bsi_h,\bu_h)\| 
\,\leq\, \wt{C}_{\tt rel}\,\Theta_2 \,.
\end{equation} 
\end{thm}
\begin{proof}
First, we observe that the proof of reliability reduces basically to derive another 
upper bound for $\|\cR_1\|_{\bbH_0(\bdiv;\Omega)'}$. Indeed, applying 
the Cauchy--Schwarz and trace inequalities in \eqref{eq:first-ressidual-alternative},
we readily deduce that
\begin{equation}\label{eq:first-ressidual-bound}
\begin{array}{l}
\ds \|\cR_1\|_{\bbH_0(\bdiv;\Omega)'} 
\,\leq\, C\,\bigg\{ \Big\|\nabla\bu_h - \frac{1}{\nu}\big(\bsi_h + (\bu_h\otimes\bu_h)\big)^\rd\Big\|_{0,\Omega} \\[2ex]
\ds\quad +\, \|\alpha\,\bu_h + \tF\,|\bu_h|^{\rp-2}\bu_h - \bdiv(\bsi_h) - \f\|_{0,\Omega}
+ \|\bu_\rD - \bu_h\|_{1/2,\Gamma} 
\bigg\} \,,
\end{array}
\end{equation}
where $C$ is a positive constant independent of $h$. In this way, replacing \eqref{eq:first-ressidual-bound} back into \eqref{eq:error-bound-by-R1-R2},
we obtain the upper bound in \eqref{eq:estimator2-bound} concluding the required estimate. 
On the other hand, for the efficiency estimate we simply observe, thanks to 
the trace theorem in $\bH^1(\Omega)$, that there exists a positive constant $c$, 
depending on $\Gamma$ and $\Omega$, such that
\begin{equation*}
\|\bu_\rD - \bu_h\|_{1/2,\Gamma}^2 
\,=\, \|\bu - \bu_h\|_{1/2,\Gamma}^2 
\,\leq\, c\,\|\bu - \bu_h\|_{1,\Omega}^2 \,.
\end{equation*}
The rest of the arguments are contained in the proof of Theorem \ref{th:efficience-estimator1}.
Further details are omitted. 
\end{proof}

We end this appendix by remarking that the eventual use of $\Theta_2$ in an adaptive algorithm
solving \eqref{eq:CBF-augmented-discrete-formulation-H0} would be discouraged by the 
non-local character of the expression $\|\bu_\rD - \bu_h\|_{1/2,\Gamma}$. 
In order to circumvent this situation, we now replace this term by a suitable upper bound, 
which yields a reliable and fully local {\it a posteriori} error estimator.
Note that unfortunately in exchange for this we lose the possibility of obtaining efficiency analytically.
\begin{thm}
Assume that the data $\f$ and $\bu_\rD$ satisfy \eqref{eq:data-assumption-Cea}, and let
\begin{equation}\label{eq:global-estimator-2-hat}
\wh{\Theta}_2 \,:=\, \left\{ \sum_{T\in\cT_h} \wh{\Theta}_{2,T}^2 \right\}^{1/2} \,,
\end{equation}
where
\begin{equation}\label{eq:local-estimator-2-hat}
\begin{array}{l}
\ds \wh{\Theta}_{2,T}^2 \,:=\, \Big\|\nabla\bu_h - \frac{1}{\nu}\big(\bsi_h + (\bu_h\otimes\bu_h)\big)^\rd\Big\|_{0,T}^2 \\[2ex]
\ds\quad +\, \|\alpha\,\bu_h + \tF\,|\bu_h|^{\rp-2}\bu_h - \bdiv(\bsi_h) - \f\|_{0,T}^2 
+ \sum_{e\in\cE_{h,T}(\Gamma)} \|\bu_\rD - \bu_h\|_{1,e}^2 \,.
\end{array}
\end{equation}
Then, there exists a positive constant $\wh{C}_{\tt rel}$, independent of $h$, such that 
\begin{equation}\label{eq:reliability-Theta2-hat}
\|(\bsi,\bu) - (\bsi_h,\bu_h)\| \,\leq\, \wh{C}_{\tt rel}\,\wh{\Theta}_2 \,.
\end{equation}
\end{thm}
\begin{proof}
It reduces to bound $\|\bu_\rD - \bu_h\|_{1/2,\Gamma}$.
In fact, since $\bH^1(\Gamma)$ is continuously embedded in $\bH^{1/2}(\Gamma)$, 
there exists a positive constant $C$, depending on $\Gamma$, such that 
\begin{equation*}
\|\bu_\rD - \bu_h\|_{1/2,\Gamma}^2 \,\leq\, C\,\|\bu_\rD - \bu_h\|_{1,\Gamma}^2 
\,=\, C\,\sum_{e\in\cE_h(\Gamma)} \|\bu_\rD - \bu_h\|_{1,e}^2 \,,
\end{equation*}
which, together with the upper bound of \eqref{eq:estimator2-bound}, 
implies \eqref{eq:reliability-Theta2-hat} and finishes the proof.
\end{proof}


\begin{thebibliography}{99}
	
\bibitem{Brezzi-Fortin}
{\sc F. Brezzi and M. Fortin},
{\it Mixed and Hybrid Finite Element Methods}.
Springer Series in Computational Mathematics, 15. Springer--Verlag, New York, 1991.

\bibitem{cgor2017}
{\sc J. Cama\~no, G.N. Gatica, R. Oyarz\'ua, and R. Ruiz-Baier},
{\it An augmented stress-based mixed finite element method for the steady state Navier-Stokes equations with nonlinear viscosity}.
Numer. Methods Partial Differential Equations 33 (2017), no. 5, 1692--1725.

\bibitem{cgot2016}
{\sc J. Cama\~no, G.N. Gatica, R. Oyarz\'ua, and G. Tierra},
{\it An augmented mixed finite element method for the Navier-Stokes equations with variable viscosity}.
SIAM J. Numer. Anal. 54 (2016), no. 2, 1069--1092.

\bibitem{cot2017-MC}
{\sc J. Cama\~no, R. Oyarz\'ua, and G. Tierra},
{\it Analysis of an augmented mixed-FEM for the Navier-Stokes problem}.
Math. Comp. 86 (2017), no. 304, 589--615.

\bibitem{cgo2019}
{\sc S. Caucao, G.N. Gatica, and R. Oyarz\'ua},
{\it A posteriori error analysis of an augmented fully mixed formulation for the nonisothermal Oldroyd-Stokes problem}.
Numer. Methods Partial Differential Equations 35 (2019), no. 1, 295--324. 

\bibitem{cgos2020}
{\sc S. Caucao, G.N. Gatica, R. Oyarz\'ua, and N. S\'anchez},
{\it A fully-mixed formulation for the steady double-diffusive convection system based upon Brinkman--Forchheimer equations}.
J. Sci. Comput. 85 (2020), no. 2, Paper No. 44, 37 pp.

\bibitem{cgoz2022}
{\sc S. Caucao, G.N. Gatica, R. Oyarz\'ua, and P. Z\'u\~niga},
{\it A posteriori error analysis of a mixed finite element method for the coupled Brinkman--Forchheimer and double-diffusion equations}.
J. Sci. Comput. 93 (2022), no. 2, Paper No. 50, 42 pp.

\bibitem{cmo2016}
{\sc S. Caucao, D. Mora, and R. Oyarz\'ua}, 
{\it A priori and a posteriori error analysis of a pseudostress-based mixed formulation of the Stokes problem with varying density}. 
IMA J. Numer. Anal. 36 (2016), no. 2, 947--983.

\bibitem{covy2022}
{\sc S. Caucao, R. Oyarz\'ua, S. Villa-Fuentes, and I. Yotov},
{\it A three-field Banach spaces-based mixed formulation for the unsteady Brinkman--Forchheimer equations}.
Comput. Methods Appl. Mech. Engrg. 394 (2022), Paper No. 114895, 32 pp. 

\bibitem{cy2021}
{\sc S. Caucao and I. Yotov,}
{\it A Banach space mixed formulation for the unsteady Brinkman-Forchheimer equations}. 
IMA J. Numer. Anal. 41 (2021), no. 4, 2708--2743.

\bibitem{cku2005}
{\sc A.O. Celebi, V.K. Kalantarov, and D. Ugurlu},
{\it Continuous dependence for the convective Brinkman--Forchheimer equations}.
Appl. Anal. 84 (2005), no. 9, 877--888.

\bibitem{Ciarlet}
{\sc P.G. Ciarlet},
{\it Linear and Nonlinear Functional Analysis with Applications}.
Society for Industrial and Applied Mathematics, Philadelphia, PA, 2013.

\bibitem{clement1975}
{\sc P. Cl\'ement},
{\it Approximation by finite element functions using local regularisation}.
RAIRO Mod\'elisation Math\'ematique et Analyse Num\'erique 9 (1975), 77--84.

\bibitem{cgo2017-CALCOLO}
{\sc E. Colmenares, G.N. Gatica, and R. Oyarz\'ua},
{\it An augmented fully-mixed finite element method for the stationary Boussinesq problem}.
Calcolo 54 (2017), no. 1, 167--205.

\bibitem{cgo2019-CMA}
{\sc E. Colmenares, G.N. Gatica, and R. Oyarz\'ua},
{\it A posteriori error analysis of an augmented fully-mixed formulation for the stationary Boussinesq model}. 
Comput. Math. Appl. 77 (2019), no. 3, 693--714.


\bibitem{dgm2015}
{\sc C. Domínguez, G.N. Gatica, and S. Meddahi},
{\it A posteriori error analysis of a fully-mixed finite element method for a two-dimensional fluid-solid interaction problem}.
J. Comput. Math. 33 (2015), no. 6, 606--641.


	
\bibitem{Gatica}
{\sc G.N. Gatica},
{\it A Simple Introduction to the Mixed Finite Element Method. Theory and Applications}.
SpringerBriefs in Mathematics. Springer, Cham, 2014.

\bibitem{g2020}
{\sc G.N. Gatica},
{\it A note on stable Helmholtz decompositions in 3D}.
Appl. Anal. 99 (2020), no. 7, 1110--1121.

\bibitem{ggm2014}
{\sc G.N. Gatica, L.F. Gatica, and A. M\'arquez},
{\it Analysis of a pseudostress-based mixed finite element method for the Brinkman model of porous media flow}.
Numer. Math. 126 (2014), no. 4, 635--677.

\bibitem{gms2010}
{\sc G.N. Gatica, A. M\'arquez, and M.A. S\'anchez},
{\it Analysis of a velocity-pressure-pseudostress formulation for the stationary Stokes equations}.
Comput. Methods Appl. Mech. Engrg. 199 (2010), no. 17-20, 1064--1079. 

\bibitem{gos2018}
{\sc L.F. Gatica, R. Oyarz\'ua, and N. S\'anchez},
{\it A priori and a posteriori error analysis of an augmented mixed-FEM for the Navier-Stokes-Brinkman problem}.
Comput. Math. Appl. 75 (2018), no. 7, 2420--2444. 

\bibitem{grt2016}
{\sc G.N. Gatica, R. Ruiz-Baier, and G. Tierra},
{\it A posteriori error analysis of an augmented mixed method for the Navier--Stokes equations with nonlinear viscosity}. 
Comput. Math. Appl. 72 (2016), no. 9, 2289--2310.  
	
\bibitem{Girault-Raviart}
{\sc V. Girault and P.A. Raviart},
{\it Finite Element Methods for Navier--Stokes Equations. Theory and Algorithms}.
Springer Series in Computational Mathematics, 5. Springer-Verlag, Berlin, 1986.


\bibitem{gm1975}
{\sc R. Glowinski and A. Marroco},
{\it Sur l'approximation, par \'el\'ements finis d'ordre un, et la r\'esolution, par p\'enalisation-dualit\'e, d'une classe de probl\`emes de Dirichlet non lin\'eaires}.
Rev. Fran\c{c}aise Automat. Informat. Recherche Op\'erationnelle S\'er. Rouge Anal. Num\'er. 9 (1975), no. R-2, 41--76. 
	
\bibitem{Hecht2012}
{\sc F. Hecht},
{\it New development in FreeFem++}.
J. Numer. Math. 20 (2012), 251--265.

\bibitem{Hecht2018}
{\sc F. Hecht},
{\it FreeFem++}.
Third Edition, Version 3.58-1. Laboratoire Jacques-Louis Li-
ons, Université Pierre et Marie Curie, Paris, 2018. [available in http://www.freefem.org/ff++].

\bibitem{kjf1977}
{\sc A. Kufner, O. Jhon, and S. Fu\v c\'ik},
{\it Function spaces. Monographs and Textbooks on Mechanics of Solids and Fluids; Mechanics: Analysis}.
Noordhoff International Publishing, Leyden; Academia, Prague, 1977.

\bibitem{ll2019}
{\sc D. Liu and K. Li},
{\it Mixed finite element for two-dimensional incompressible convective Brinkman-Forchheimer equations}.
Appl. Math. Mech. (English Ed.) 40 (2019), no. 6, 889--910. 




\bibitem{Quarteroni-Valli}
{\sc A. Quarteroni and A. Valli},
{\it Numerical Approximation of Partial Differential Equations}.
Springer Series in Computational Mathematics, 23. Springer-Verlag, Berlin, 1994.

\bibitem{Roberts-Thomas}
{\sc J.E. Roberts and J.M. Thomas},
{\it Mixed and hybrid methods}.
P. G. Ciarlet and J. L. Lions, editors, Handbookof Numerical Analysis, vol. II, Finite Element Methods (Part 1), North-Holland, Amsterdam, 1991.

\bibitem{Verfurth}
{\sc R. Verf\"urth},
{\it A Review of A-Posteriori Error Estimation and Adaptive Mesh-Refinement Techniques}.
Wiley Teubner, Chichester, 1996.

\bibitem{y2023}
{\sc H. Yu},
{\it Axisymmetric solutions to the convective Brinkman-Forchheimer equations}.
J. Math. Anal. Appl. 520 (2023), no. 2, Paper No. 126892, 12 pp.

\bibitem{zy2012}
{\sc C. Zhao and Y. You},
{\it Approximation of the incompressible convective Brinkman--Forchheimer equations}.
J. Evol. Equ. 12 (2012), no. 4, 767--788.
	
\end{thebibliography}
\end{document}